\documentclass[12pt]{article}
\usepackage{graphicx}
\usepackage{amssymb,amsmath,amsthm}
\usepackage{bbold}
\def\bydef{{\buildrel \rm def \over =}}
\def\c{c}
\def\n{n}
\def\GL{GL}
\def\SL{SL}
\def\sp{\mathop{span}}
\def\tr{\mathop{tr}}
\def\bA{{\boldsymbol{A}}}
\def\bF{{\boldsymbol{F}}}
\def\bH{{\boldsymbol{H}}}
\def\boC{{\boldsymbol{C}}}
\def\boS{{\boldsymbol{S}}}
\def\bF{{\boldsymbol{F}}}
\def\bN{\mathbb{N}}
\def\bZ{\mathbb{Z}}
\def\bC{\mathbb{C}}
\def\bR{\mathbb{R}}
\def\bS{\mathbb{S}}
\def\CPo{{\mathbb{C}\hbox{P}^1}}
\def\Rt{{\mathbb{R}^2}}
\def\RPt{{\mathbb{R}P^2}}

\def\RPnm{{\mathbb{R}\hbox{P}^{n-1}}}
\def\cA{{\cal A}}
\def\cAM{{{\cal A}_{M}}}

\def\cB{{\cal B}}
\def\cC{{\cal C}}
\def\cD{{\cal D}}
\def\cE{{\cal E}}
\def\cG{{\cal G}}
\def\cI{{\cal I}}
\def\cJ{{\cal J}}
\def\cH{{\cal H}}

\def\cP{{\cal P}}
\def\cQ{{\cal Q}}

\def\one{\mathbbb{1}}

\newtheorem{definition}{Definition}
\newtheorem{theorem}{Theorem}
\newtheorem*{theorem1}{Theorem 1}
\newtheorem*{theorem2}{Theorem 2}
\newtheorem*{theorem3}{Theorem 3}
\newtheorem*{theorem4}{Theorem 4}
\newtheorem*{theorem5}{Theorem 5}
\newtheorem{lemma}{Lemma}
\newtheorem{corollary}{Corollary}
\newtheorem{example}{Example}

\newtheorem{remark}{Remark}
\newtheorem{question}{Open Question}
\newtheorem{proposition}{Proposition}
\newtheorem{conjecture}[theorem]{Conjecture} 
\usepackage{array}
\newcolumntype{S}{>{\centering\arraybackslash} m{.475\linewidth}}
\newcolumntype{T}{>{\centering\arraybackslash} m{10.5cm}}
\newcolumntype{U}{>{\centering\arraybackslash} m{1.5cm}}
\title{%
  Exponential growth of norms\\
  in semigroups of linear automorphisms\\
  and Hausdorff dimension of self-projective IFS.
}
\author{Roberto De Leo}
\begin{document}
\maketitle
\begin{abstract}
  Given a finitely generated semigroup $S$ of the (normed) set of
  linear maps of a vector space $V$ into itself, we find sufficient conditions 
  for the 
  exponential growth of the number $N(k)$ of elements of the semigroup 
  contained in the sphere of radius $k$ as $k\to\infty$. We relate the growth rate 
  $\lim_{k\to\infty}\log N(k)/\log k$ to the exponent of a zeta function naturally
  defined on the semigroup and, in case $S$ is a semigroup of volume-preserving
  automorpisms, to the Hausdorff and box dimensions of the limit set of the 
  induced semigroup of automorphisms on the corresponding projective space.
\end{abstract}
\section{Introduction}
The asymptotic behaviour of the norms of products of some fixed finite 
set of square matrices has been extensively studied in the context of
the theory of random matrices. In particular, in a celebrated paper~\cite{FK60},
Furstenberg and Kesten proved that, given some finite number of square 
matrices $A_i$, under suitable conditions the norm of almost all products 
of $k$ of the $A_i$ grows as $\gamma^k$, where $\gamma$ is the Lyapunov 
exponent associated to the $A_i$.
In this paper we address the subject from a different point of view, namely
we consider {\em all} possible products of the $A_i$ and provide sufficient 
conditions for the existence and boundedness of $\lim_{k\to\infty}\log N(k)/\log k$, 
where $N(k)$ is the number of these products that lie inside the (closed)
sphere of radius $k$.
As a byproduct, we relate the rate of this growth to a zeta function naturally 
defined on semigroups of square matrices and, in particular cases, to the 
Hausdorff dimension of the limit set of the orbit of a point under the natural 
action induced by the semigroup on its corresponding projective space.
\vskip .35cm\noindent
{\bf Motivational Example 1: 
The {\em Cubic gasket} $\boC_3$.}
The real self-projective fractal $\boC_3\subset\bR P^2$ (see Fig.~\ref{fig:dc}) 
is the main reason for our interest in the subject of the present paper. 
\begin{figure}
  \centering
  \begin{tabular}{SS}
    \includegraphics[width=7cm,clip=true,trim=80 250 20 100]{gl3z8-l1-d7}&
    \begin{tabular}{c}
      \includegraphics[width=5.3cm]{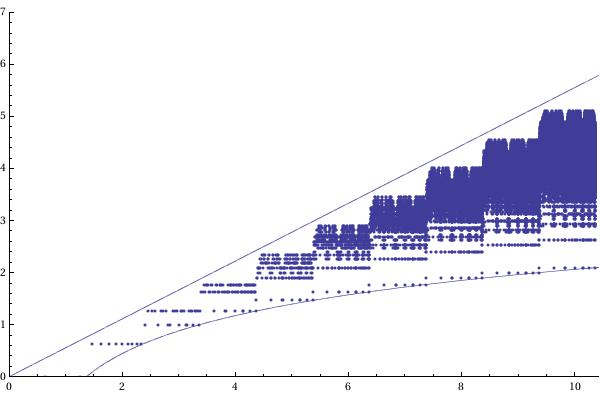}\\
      \includegraphics[width=5.3cm]{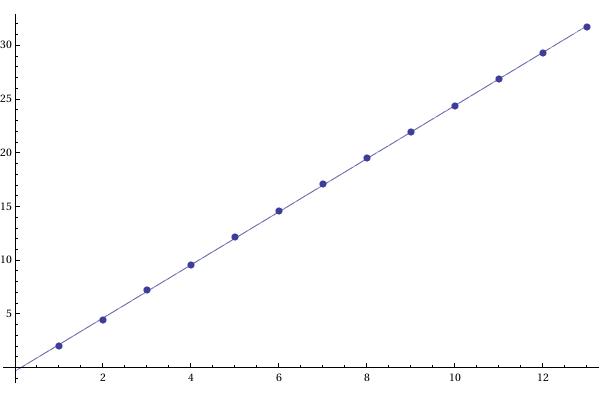}\\
    \end{tabular}\\
  \end{tabular}
  \caption{%
    \small 
    (left) The Cubic gasket $\boC_3\subset\bR P^2$ 
    in the triangle $T$ of vertices with homogeneous coordinates
    $[x:y:z]=[0:0:1]$, $[1:0:1]$, $[0:1:1]$. The picture shows (in green)
    the set $T_{7,\boC_3}$ in the affine chart $z=1$.
    (right, top) Log-log plot of the norms of matrices $C_I\in\boC_3$, $|I|\leq11$, 
    ordered in lexicographic order. The fastest growing norms are 
    $\|C_{i}\cdot C_{i+1}\cdots C_{i+k}\|\simeq\alpha_3^k$, where sums of indices 
    are intended ``modulo 3'' and $\alpha_3\simeq1.84$ is the Tribonacci constant.
    The slowest growing ones are $\|C_i^k\|=k$.
    (right, bottom) Log-log plot of the function $N_{\boC_3}(k)$ representing 
    the number of matrices of $\boC_3$ whose norm is not larger than $k$.
    Numerical data (the values of $N_{\boC_3}(k)$ shown in the graphic 
    are {\em exact}, see Table~\ref{tab:N}) 
    indicate that $N_{\boC_3}(k)\simeq Ak^s$ for $A\simeq0.967$ and $s\simeq 2.444$.
    According to Conjecture 1, this entails that $\dim_B\boC_3\geq1.63$.
  }
  \label{fig:dc}
\end{figure}
It was first introduced, in the author's knowledge,
by G.~Levitt~\cite{Lev93} and independently rediscovered more recently by the 
author and I.A.~Dynnikov in connection with the S.P. Novikov theory of plane 
sections of periodic surfaces~\cite{DD09}. We call it {\em cubic} because 
it is related to the topology of plane sections of the cubic polyhedron $\{4,6|4\}$
(see~\cite{DD09} for details) and {\em gasket} because it has the same 
topology of the Sierpinski and Apollonian gaskets. Like the Sierpinski gasket,
it can be thought as the set obtained by removing from the (projective)
triangle $T(\cE)$ with vertices $[e_1]$, $[e_2]$, $[e_3]$, where 
$\cE=\{e_1,e_2,e_3\}$ is any frame of $\bR^3$, the (projective) triangle
with vertices $[e_1+e_2]$, $[e_2+e_3]$, $[e_3+e_1]$ and repeating this
procedure recursively on the three triangles left. We denote by 
$T_{k,\boC_3}\subset T(\cE)$
the set obtained after repeating this procedure $k$ times. Clearly 
$\boC_3=\cap_{k=1}^\infty T_{k,\boC_3}$, i.e. we can get as close as we please
to $\boC_3$ by considering sets $T_{k,\boC_3}$ with large values of $k$.
$\boC_3$ 
can also be characterized as the (unique) 
subset of the triangle with vertices $[1:0:0]$, $[0:1:0]$ and $[0:0:1]$ which 
is invariant under the action of the (free) subsemigroup of $PSL_3(\bN)$ generated 
by the projective automorphsims $\psi_i$, $i=1,2,3$, induced by the following 
three $SL_3(\bN)$ matrices:
$$
C_1 = \begin{pmatrix}
  1&0&0\cr
  1&1&0\cr
  1&0&1\cr
  \end{pmatrix},\,
C_2 = \begin{pmatrix}
  1&1&0\cr
  0&1&0\cr
  0&1&1\cr
  \end{pmatrix},\,
C_3 = \begin{pmatrix}
  1&0&1\cr
  0&1&1\cr
  0&0&1\cr
  \end{pmatrix}.
$$
By abuse of notation, we denote by $\boC_3$ also the semigroup generated
by the $C_i$.
As a consequence of a conjecture of S.P. Novikov~\cite{Nov00}, the
set $\boC_3$ is supposed to have Hausdorff dimension {\em strictly} 
between 1 and 2.
What makes checking this conjecture non-trivial is that each $\psi_i$
has exactly one of the three vertices as fixed point and in that
point it has Jacobian equal to $\one_3$, namely the 
iterated function system (IFS)
$\{\psi_1,\psi_2,\psi_3\}$ is {\em parabolic} rather than 
{\em hyperbolic}\footnote{Recall that a IFS $\{f_1,\dots,f_m\}$ 
on a metric space $(M,d)$ is said {\em hyperbolic} when all 
$f_i$ are contractions with respect to $d$ and {\em parabolic}
when all $f_i$ are non-expanding maps.}. 

No analytical bound for this fractal is known to date. 
In Section~\ref{sec:fract}, based on numerical evidence and
Theorems 4 and 5, valid for $2\times2$ matrices,
we conjecture that the box dimension of $\boC_3$ is related by
the growth rate of the norms of the elements of the semigroup
generated by the $C_i$ (see Fig.~\ref{fig:dc}), namely by 
$s=\lim_{k\to\infty}\log N(k)/\log k$, where $N(k)$ is the number of 
matrices of $\boC_3$ inside the closed ball of $M_3(\bR)$ of radius $k$.
According to this conjecture, $\dim_B\boC_3\geq2s/3\simeq1.63$ 
(see Section~\ref{sec:rSG}).
%
%
\vskip .35cm\noindent
{\bf Motivational Example 2: 
The {\em Apollonian gasket}.}
The complex self-projective fractal $\bA_3\subset\bC P^1$ is possibly 
the fractal with the 
oldest ancestry, since its construction relies on a celebrated
result of the Hellenistic mathematician Apollonius of Perga 
(ca 262 BC -- ca 190 BC), known in his times as {\em The Great Geometer}.
Apollonius' result, contained in the now--lost book {\em Tangencies} but
fortunately reported by Pappus of Alexandria in his 
{\em Collection}~\cite{Pap40}
published about five centuries later, concerns the existence of circles
tangent to a given triple of objects that can be any combination of points,
straight lines and circles. In particular, given three circles which are
mutually externally tangent to each other (sometimes called the {\em four 
coins problem}~\cite{Old96}),
there exist 
exactly two new circles tangent to each of them, one externally and one 
internally (see Fig.~\ref{fig:apollonian}). 
The three given circles plus any one of the new 
ones\footnote{In Soddy's honor the two new circles are called {\em Soddy's circles}.}
form a {\em Descartes configuration}, since it was Descartes that stated
the following remarkable relation between the curvatures $c_1,\dots,c_4$
of the four circles (see~\cite{Cox68} for details)
memorialized three centuries later by the Chemistry
Nobelist Frederick Soddy in his poem ``The Kiss Precise''~\cite{Sod36} 
after rediscovering it independently: 
$2\sum_{i=1}^4c^2_i = \left(\sum_{i=1}^4c_i\right)^2$.

Since M\"obius transformations preserve circles and are transitive on
triples of distinct points, 
they also act transitively on
the set of all possible Descartes configurations; this fact suggests that
their most natural environment is the Riemann 
sphere $\bC P^1$ rather than the plane. Any Descartes configuration $D$ 
divides $\bC P^1$ in 4 curvilinear triangles $T_i$ in such a way that every 
circle of $D$ is one of the two Soddy circles of the remaining three circles 
of $D$. 
By drawing the new Soddy circle of each of the 4 triples we are left
with 4 new Descartes configurations. By repeating this process recursively
we generate an infinite osculating circle packing of $\bC P^1$ which, not
surprisingly, is called {\em Apollonian packing}.

Here we rather focus our attention on any one of the curvilinear triangles $T$
and call {\em Apollonian gasket} $\bA_3$ the set of points of $T$ left after
removing from $T$ the interior of all Soddy circles inside it.
Like in case of the cubic gasket, $\bA_3$ can be characterized as the 
invariant set of a complex self-projective 
{\em parabolic} IFS. The fact that, thanks to the complex
structure of $\bC P^1$, $\bA_3$ is {\em self-conformal} was exploited by 
Mauldin and Urbanski to prove some of its fundamental properties~\cite{MU98}. 
Unfortunately these techniques do not seem to extend to the previous (real) 
case, when the IFS maps are parabolic but not conformal. 


In 1967 K.E. Hirst~\cite{Hir67} introduced the
{\em Hirst semigroup} $\bH$, namely the subsemigroup of $SL_4(\bN)$
generated by the matrices
$$
H_1 = \begin{pmatrix}
  1&0&0&0\cr
  0&1&0&0\cr
  1&1&1&2\cr
  1&1&0&1\cr
  \end{pmatrix},\,
H_2 = \begin{pmatrix}
  1&0&0&0\cr
  0&0&1&0\cr
  1&1&1&2\cr
  1&0&1&1\cr
  \end{pmatrix},\,
H_3 = \begin{pmatrix}
  0&1&0&0\cr
  0&0&1&0\cr
  1&1&1&2\cr
  0&1&1&1\cr
  \end{pmatrix},
$$
as an effective tool to represent the radii of the Soddy circles
in the gasket.

In a series of fundamental contributions to the study of the Hausdorff
dimension of the gasket~\cite{Boy70,Boy71,Boy72,Boy73a,Boy73b,Boy82}, 
D.W. Boyd ultimately 
characterized this dimension in terms of the Hirst semigroup by proving (implicitly,
in terms of the circles' curvatures) that: 
1) the number $N_\bH(k)$ of the semigroup matrices 
with norm\footnote{Since all norms are equivalent in finite dimension, this
is true for any norm.} not larger than $k$ is logarithmically asymptotic
to $k^s$ for some $s>0$, namely 
$\lim_{k\to\infty}\frac{\log N_\bH(k)}{\log k}=s$; 2)
$\dim_H\bA_3=s$.

Later in this paper we show that $\bA_3$ can be seen 
as the invariant set of the {\em parabolic} Kleinian IFS 
corresponding to the subsemigroup of $SL_2(\bC)$ generated 
by the matrices
$$
A_1 = \begin{pmatrix}
  0&i\cr
  i&2\cr
  \end{pmatrix},\,
A_2 = \frac{1}{2}\begin{pmatrix}
  \phantom{-}1&1\cr
  -1&3\cr
  \end{pmatrix},\,
A_3 = \frac{1}{2}\begin{pmatrix}
  1&-1\cr
  1&\phantom{-}3\cr
  \end{pmatrix},
$$
which by abuse of notation we will denote too by $\bA_3$,
and conjecture that 
$\lim_{k\to\infty}\log N_{\bA_3}(k)/\log k=2\dim_H\bA_3$ based on the fact that this relation
holds for similar semigroups that induce {\em hyperbolic} IFSs and on numerical 
evidence.
\begin{figure}
  \centering
  \begin{tabular}{cc}
    \includegraphics[width=4.5cm,clip=true,trim=160 520 180 70]{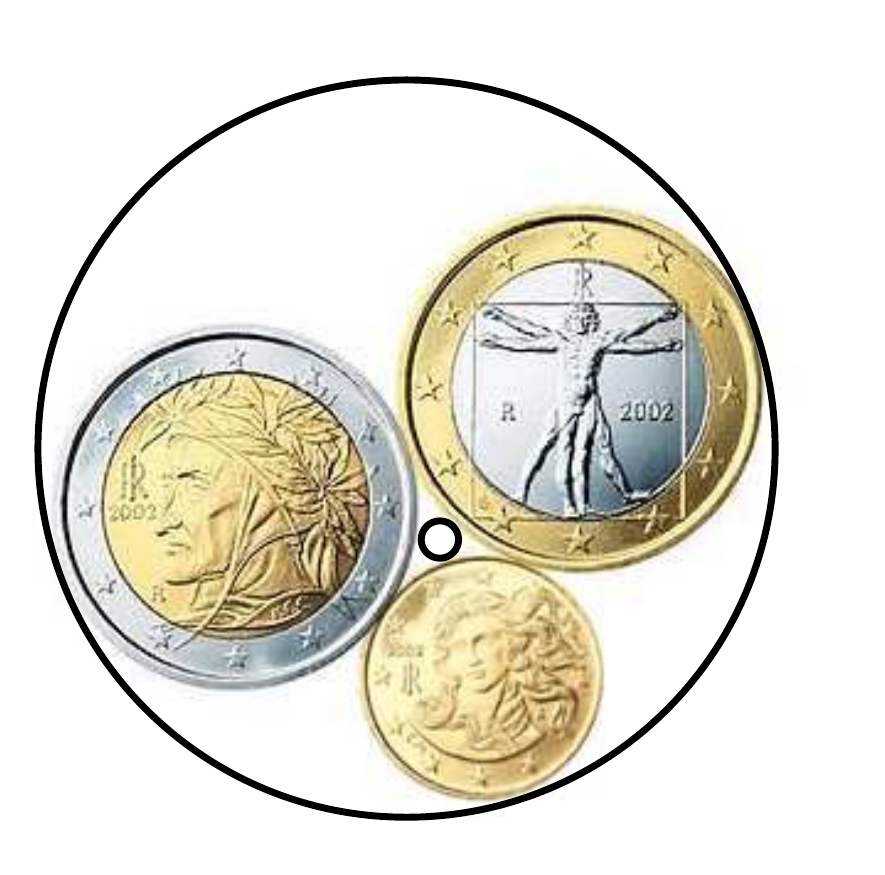}&
    \includegraphics[width=6cm]{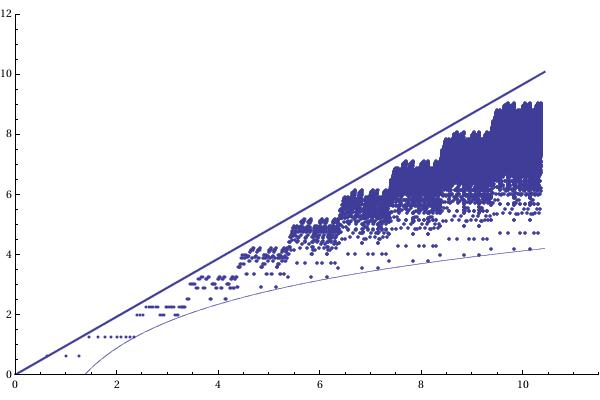}\\
    \includegraphics[width=7cm,clip=true,trim=150 560 130 100]{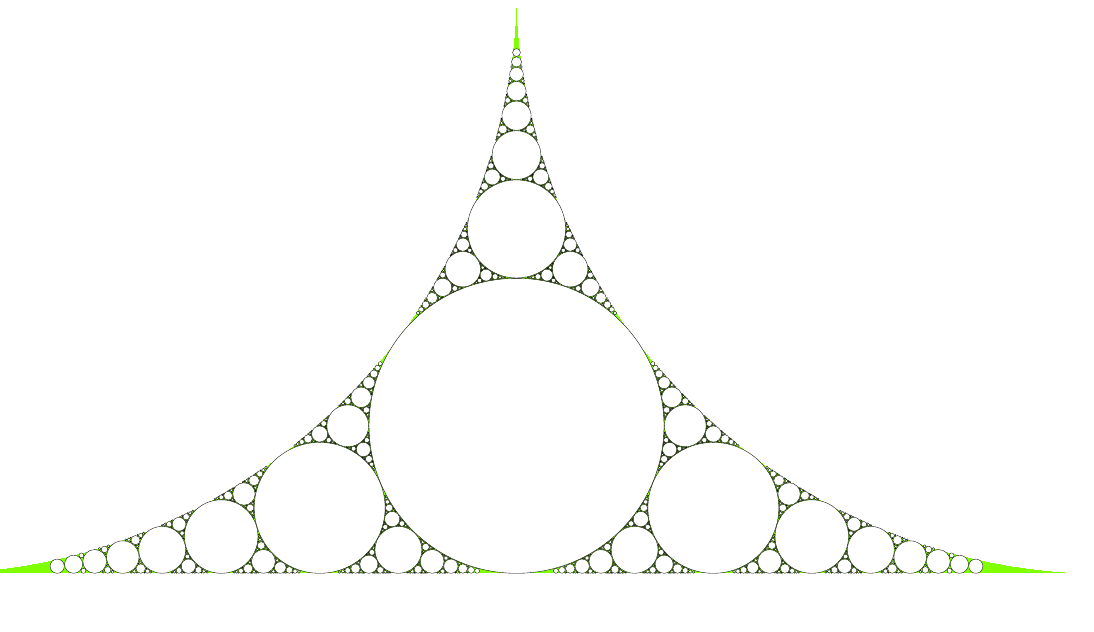}&
    \includegraphics[width=6cm]{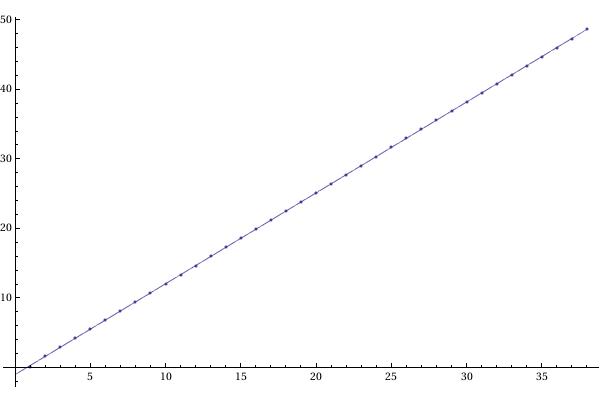}
  \end{tabular}
  \caption{%
    \footnotesize
    (top, left) Inscribed and circumscribed circles in the 
    {\em four coins problem}.
    (bottom, left) Apollonian gasket $\bA_3$ in the curvilinear triangle 
    with mutually tangent arcs as sides and the points of homogeneous
    coordinates $[z:w]$ equal to $[1:1]$, $[-1:1]$, $[i,1]$ as vertices, 
    represented in the affine chart $w=1$. In the image it is shown (in green) 
    the set $T_{7,\bA_3}$. (right) Log-log plots of the norms of the matrices 
    of the semigroup $\bH$ 
    in lexicographic order (top) and of the relative function $N_\bH(k)$
    (bottom) counting the number of matrices of $\bH$ whose norm is
    not larger than $k$ (see Table~\ref{tab:N} for the values of $N_\bH(k)$
    shown in the graph).
  } 
  \label{fig:apollonian}
\end{figure}
%
%
%
\vskip .35cm\noindent
{\bf Motivational Example 3: 
The {\em Sierpinski gasket}.}
The self-affine fractal set $\boS_3\subset\Rt$ is less ancient than the 
Apollonian one, having been introduced in the Mathematics literature
by W. Sierpinski only in 1915~\cite{Sie15}, but it does have a long history too
since its pattern has been known and used in art for about a 
millennium~\cite{PA02} (see Fig.~\ref{fig:Sie}). 
Its dimension is easily calculated: $\dim_H\boS_3=\log_23$ (e.g. see~\cite{Fal90}).
Since both $PSL_3(\bR)$ and $PSL_2(\bC)$ contain a subgroup isomorphic
to the group of affine transformations of the plane, the Sierpinski gasket 
can also be seen as a real (respectively complex) self-projective fractal of 
$\bR P^2$ (respectively $\bC P^1$).

A semigroup generating the Sierpinski fractal in $\RPt$ is,
for example, the one generated by the matrices
$$
S^\bR_1 = \frac{1}{^3\sqrt{2}}\begin{pmatrix}
  2&0&0\cr
  1&1&0\cr
  1&0&1\cr
  \end{pmatrix},\,
S^\bR_2 = \frac{1}{^3\sqrt{2}}\begin{pmatrix}
  1&1&0\cr
  0&2&0\cr
  0&1&1\cr
  \end{pmatrix},\,
S^\bR_3 = \frac{1}{^3\sqrt{2}}\begin{pmatrix}
  1&0&1\cr
  0&1&1\cr
  0&0&2\cr
  \end{pmatrix}.
$$
\begin{figure}
  \centering
  \includegraphics[width=5cm,clip=true,trim=80 250 20 100]{gl3z8-l2-d7}\hskip.75cm\includegraphics[height=5.5cm]{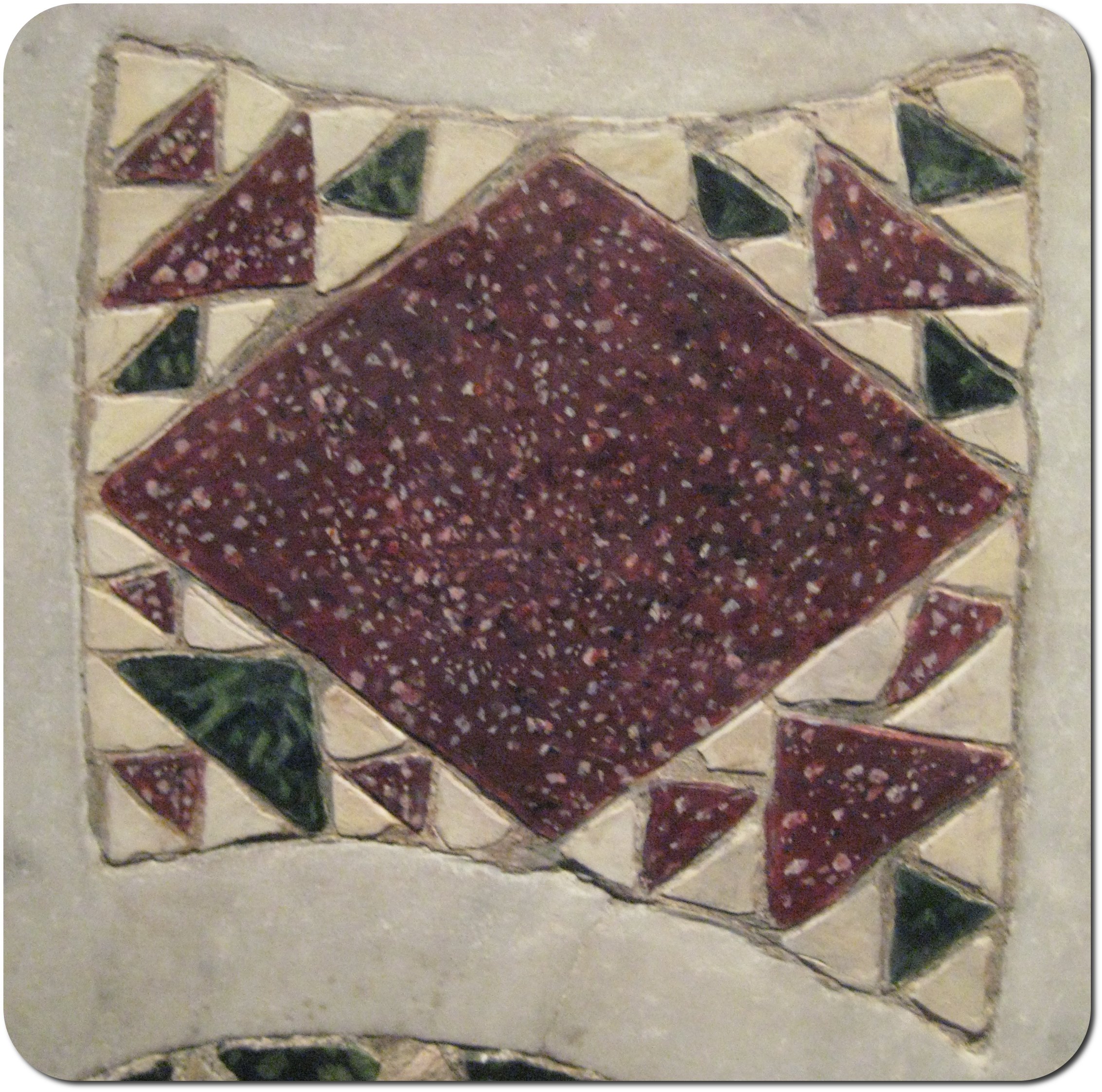}
  \caption{%
    \small
    (left) Image of the Sierpinski gasket in the triangle $T$ with 
    vertices $(0,0)$, $(1,0)$, $(0,1)$.
    In the picture it is shown, in green, the set $T_{7,\boS_3}$. 
    (right) Detail of a cosmatesque~\cite{PA02} mosaic dated about 11th--12th 
    century (photo taken by the author at the Phillips Museum in Washington DC).
  }
  \label{fig:Sie}
\end{figure}

One generating it in $\CPo$ is, for example, the one induced by the matrices
$$
S^\bC_1 = \frac{1}{\sqrt{2}}\begin{pmatrix}
  1&i\cr
  0&2\cr
  \end{pmatrix},\,
S^\bC_2 = \frac{1}{\sqrt{2}}\begin{pmatrix}
  1&1\cr
  0&2\cr
  \end{pmatrix},\,
S^\bC_3 = \frac{1}{\sqrt{2}}\begin{pmatrix}
  1&-1\cr
  0&\phantom{-}2\cr
  \end{pmatrix}.
$$
The regularity of the matrices $S^\bR_i$ and $S^\bC_i$ makes possible to
perform simple direct calculations that illustrate the main points of this paper.

Consider first the real version. Let $\|S\|_\infty$ be the norm given by
the maximum absolute row sum of $S$. Since all lines of the $S^\bR_i$
sum to 2 then $\|S^{\bR}_{i_1}\cdots S^{\bR}_{i_p}\|=(2^{-1/3})^p\cdot 2^p=2^{2p/3}$
for every $p\geq1$. Hence in the sphere of radius $k$ lie 
$$
N(k)=\sum_{p=0}^{\lfloor\frac{3}{2}\log_2k\rfloor}3^p=\frac{3^{\lfloor\frac{3}{2}\log_2k+1\rfloor}-1}{2}
$$
products of the $S^\bR_i$, where $\lfloor3\log_2k/2\rfloor$ is the integer
part of $3\log_2k/2$, and therefore
$$
s_{\bR}=\lim_{k\to\infty}\frac{\log_2N(k)}{\log_2k}=\frac{3}{2}\log_23.
$$
Note that $s_\bR$ is also the exponent that separates the values for which
the series $\sum_{S\in\langle S^\bR_i\rangle}\|S\|^{-s}$ diverges from those for
which it diverges, where the sum is extended to all
elements of the semigroup freely generated by the $S^\bR_i$.
Finally, note that the following relation holds between the Hausdorff
dimension of the Sierpinski gasket and the rate growth: $3\dim_H\boS_3=2s_\bR$.

Consider now the complex version. Endow $M_2(\bC)$ with the norm
$\|S\|$ given by the largest modulus of the entries of $S$.
Since the last row of each $S^\bC_i$ is $(0,2)$, then 
$\|S^{\bC}_{i_1}\cdots S^{\bC}_{i_p}\|=(2^{-1/2})^p\cdot 2^p=2^{p/2}$
for every $p\geq1$. Hence in this case
$$
N(k)=\sum_{p=0}^{\lfloor2\log_2k\rfloor}3^p=\frac{3^{\lfloor2\log_2k+1\rfloor}-1}{2}
$$
and therefore
$$
s_{\bC}=\lim_{k\to\infty}\frac{\log_2N(k)}{\log_2k}=2\log_23.
$$
Similarly to what happens in the real case, 
$s_\bC$ is also the exponent that separates the values for which
the series $\sum_{S\in\langle S^\bC_i\rangle}\|S\|^{-s}$ diverges from those for
which it diverges, where the sum is extended to all
elements of the semigroup freely generated by the $S^\bC_i$.
Note that in this case the relation between the Hausdorff
dimension of the Sierpinski gasket and the norms' growth
rate is the following: $2\dim_H\boS_3=s_\bC$.
%
\vskip .35cm
For thorough surveys on the Sierpinski gasket and especially on the
more challenging Apollonian gasket we refer the reader to the book by 
A.A. Kirillov~\cite{Kir07}, the series of papers by Lagarias, Mallows, 
Wilks and Yan~\cite{G03,G05,G06} and the recent article by 
Sarnak~\cite{Sar11}.

The present paper is structured in the following way.

In Section~\ref{sec:norms} we generalize Boyd's arguments on the asymptotics
of the sequence of radii of Soddy's circles in a Apollonian gasket and use 
them to obtain similar results on the asymptotics of norms of matrices 
in subsemigroups $S\subset M_n(K)$, $K=\bR,\bC$,
by introducing a sufficient condition for the existence of 
$\lim_{k\to\infty}\log N_S(k)/\log k$, where $N_S(k)$ is the number of
matrices in $S$ whose norm is not larger than $k$, and
relating this limit to the critical exponent of a natural zeta-function
defined on $S$.

In Section~\ref{sec:fract} we consider the action induced by these 
semigroups on the corresponding real or complex projective
spaces and study, in particular but significant cases, the relation, 
observed above in case of the Sierpinski gasket, between the critical 
exponent of the semigroup and the Hausdorff (for $n=2$) 
or box dimension (for $n\geq3$)
of the limit set of a point under its action.

We use Section~\ref{sec:defs} below to 
introduce the main concepts, notations and definition used 
throughout the paper and to state the main results of the paper.
%
\section{Notations, definitions and main results.}
\label{sec:defs}
{\bf Matrices and Norms.} We endow the vector space $M_n(K)$ of all $n\times n$ matrices
with coefficients in $K$ with the max norm, namely, given a matrix $M=(M^i_j)$,
$$
\|M\|=\max_{i,j=1,\dots,n}\left\{\left|M^i_j\right|\right\}.
$$
We denote by $B^n_r\subset M_n(K)$ the closed ball of radius 
$r>0$ in this norm.
Note that this norm is not sub-multiplicative but rather
%
\begin{equation}
  \label{eq:PQ}
  \sup_{P,Q\in M_n(K)}\frac{\|PQ\|}{\|P\|\|Q\|}=n\,.
\end{equation}
%
%
%
%
%
Since in finite dimension all norms are equivalent, the 
main results of the paper will not depend on this particular choice.

\vskip .35cm\noindent
{\bf The multi-indices semigroup.} We denote by $\cI^m$ the infinite $m$-ary tree of multi-indices 
of integers ranging from 1 to $m$ defined as follow.
The root of the tree is the number 0. The $m$ children ({\em 1-indices})
of 0 are the integers from 1 to $m$. Their children ({\em 2-indices}) 
are the ordered pairs $1i,\dots,mi$ and so on recursively for the 
{\em k-indices}, $k>2$.
We denote by $\cI^m_k$ the set of all $k$-indices of $\cI^m$. 
Since we will use them often, we denote by $\cD^m_\ell$, $\ell\geq0$,
the set of all {\em diagonal} multi-indices $I=i_1\dots i_k\in\cI^m$, 
$k\leq\ell$, i.e.
such that $i_1=\dots=i_k$, and set $\cD^m=\cup_{\ell\geq0}\cD^m_\ell$.
Similarly, we denote by $\cJ^m_\ell$, $\ell\geq2$, the set of all 
{\em next-to-diagonal} multi-indices $I=i_1i_2\dots i_k\in\cI^m$, $k\leq\ell$,
i.e. those such that $i_1\neq i_2=\dots=i_k$, and set
$\cJ^m=\cup_{\ell\geq2}\cJ^m_\ell$. 

We endow $\cI^m$ with the canonical structure of semigroup given by
$i_1\dots i_k\cdot i'_1\dots i'_{k'}=i_1\dots i_ki'_1\dots i'_{k'}$
with $0$ as identity element. We also endow $\cI^m$ with a partial order by
saying that $I\geq J$ if $I$ can be factorized as $I=LJ$ for some
multi-index $L\neq0$. 
Finally, we denote by 
$I'=i_1\dots i_k$ the $k$-index obtained from the $(k+1)$-index 
$I=i_0i_1\dots i_k$ by dropping the first index on the left.
%
%

\vskip .35cm\noindent
{\bf Gaskets of matrices.} Given $m$ matrices $A_1,\dots,A_m\in M_n(K)$,
$K=\bR$ or $\bC$, we denote by $\bA=\langle A_1,\dots,A_m\rangle$ 
the semigroup they generate, given by the intersection of all subsemigroups 
of $M_n(K)$ containing all the $A_i$ and the unit matrix.
  
In this paper we are mainly interested in the 
asymptotic growth of norms of matrices in free semigroups but, since our 
results hold for the more general case when there are relations between 
the generators, we often formulate theorems using 
the more general concept
of semigroup homomorphisms $\cI^m\to M_n(K)$. We often denote such objects
with the letter $\cA$ and use the notation
$$
A_I\bydef\cA({i_1\dots i_k})=A_{i_1}\cdot\dots\cdot A_{i_k},
$$ 
where $A_i\bydef\cA(i)$. We say that the matrices $A_1,\dots,A_m$ 
{\em generate} $\cA$.

Notice that, when there is no relation between the $A_i$, then 
there is a bijection between $\cA(\cI^m)$ and $\{\one_n\}\cup\bA$,
so when the $A_i$ are free generators it is essentially equivalent referring to 
either the semigroup homomorphism $\cA$ or the semigroup $\bA$.

%
Given any $M\in\GL_n(K)$ we denote by $\cAM$ the ``right coset'' 
map defined by $\cA_M(I)\bydef A_IM$. If $\boldsymbol{A}$ is
free then there is a bijection between $\cA_M(\cI^m)$ and 
$\{M\}\cup\bA M$, where $\bA M$ is a right coset of $\bA$. 
Clearly $\cA_{\one_n}=\cA$.
\begin{definition}
  We denote by $N_\cAM(r)$ the cardinality of the set $B^n_r\cap\cAM(\cI^m)$
  and say that $\cAM$ is a {\em $m$-gasket} (or simply a {\em gasket}) 
  if $N_\cAM(r)<\infty$ for every $r>0$. 
  We say that the gasket $\cAM$ is {\em hyperbolic} if 
  the sequence $a_k=\min_{I\in\cI^m_k}\|A_IM\|$ diverges exponentially,
  namely if there exists $\alpha>1$ such that $a_k\asymp \alpha^k$,
  where $\asymp$ means that the ratio of the terms on either side
  is bounded away from 0 and $\infty$ for all $k$. 
  When $a_k$ is slower than exponential we say that $\cAM$ is 
  {\em parabolic}\footnote{Note that $a_k$ cannot be faster than 
  exponential so this covers all possible cases.}.
\end{definition}
\begin{example}
  Every semigroup $\cA:\cI^m\to GL_n(\bC)$ whose generators have all 
  their spectrum outside the unit circle is a hyperbolic gasket.
  Consider for example the simple case of $\cA:\cI^2\to GL_2(\bC)$ with
  $$
  A_1=\begin{pmatrix}\lambda&0\cr 0&1\cr\end{pmatrix},
  A_2=\begin{pmatrix}1&0\cr 0&\lambda\cr\end{pmatrix},
  $$
  where $|\lambda|>1$. Then $\min_{|I|=2k}{\|A_I\|}=\|A_1^kA_2^k\|=\|\lambda\|^k$.
\end{example}
\begin{example}
  Every semigroup $\cA:\cI^m\to GL_n(\bC)$ whose generators have all 
  non-zero coefficients strictly larger than 1 
  is a hyperbolic gasket,
  since the norm of $k$ of such matrices will be not smaller
  than the $k$-th power of their smallest non-zero entry.
\end{example}
\begin{example}
  \label{ex:dyn}
  The (free) semigroup $\boC_2\subset SL_2(\bN)$ generated by the two 
  {\em parabolic} (i.e. with trace equal to $\pm2$) matrices
  $$
  C_1=\begin{pmatrix}1&0\cr 1&1\cr\end{pmatrix},
  C_2=\begin{pmatrix}1&1\cr 0&1\cr\end{pmatrix}
  $$
  is a parabolic gasket.
  It is a gasket because, if $M\in SL_2(\bN)$ is distinct
  from the identity, then $\|C_iM\|\geq\|M\|+1$ since $M$ has at least
  a column with two entries different from 0.
  It is parabolic because 
  $$
  \min_{I\in\cI^m_k}\{\|C_I\|\}\leq\|C^k_{1}\|=k.
  $$
  Note that $\boC_2$ contains also {\em hyperbolic} elements 
  (namely matrices $C_I$ with $|\tr C_I|>2$), e.g. 
  $$
  C_1C_2=\begin{pmatrix}2&1\cr 1&1\cr\end{pmatrix},
  $$ so that in the 
  sets $\boC_{2,k}=\{C_I,|I|=k\}$ there are some elements whose norm grows
  polynomially and some others whose norm grows exponentially with $k$,
  similarly to what happens for the cubic and Apollonian gaskets
  (see Figs.~\ref{fig:dc} and~\ref{fig:apollonian}).

  This elementary but still non-trivial example was suggested to the author by 
  I.A.~Dynnikov and was the starting point for the author's study of the 
  asymptotics of norm's growth in semigroups of linear maps in full generality.
\end{example}
\begin{example}
  Suppose that $A_1,\dots,A_m\in M_n(Z)$, where $Z=\bZ$ or $\bZ[i]$,   
  generate {\em freely} $\boldsymbol{A}$. 
  Then $\boldsymbol{A}$
  is a gasket, since
  in $M_n(Z)$ there are only finitely 
  many matrices whose norm is smaller than any fixed $r>0$
  and by the freedom hypothesis the products of any number of
  $A_i$ are all distinct.
\end{example}  
%
%
%
Note that clearly if $\bA$ is a gasket then so is also every semigroup 
conjugated to it,
as well as every semigroup obtained from it
by multiplying all elements by a constant $\lambda$
such that $|\lambda|>1$.
%
%
%
%
%
\begin{definition}
  Let $\cA$ be a $n$-gasket 
  and $M\in\GL_n(K)$.
  We call {\em zeta function} 
  of $\cAM$ the series 
  $$
  \zeta_{\cAM}(s) = \sum_{I\in\cI^m}\frac{1\phantom{^s}}{\|A_IM\|^s}.
  $$
  We call {\em exponent} of $\cAM$ the number $s_{\cAM}$ defined as follows:
  $$
  s_\cAM = \sup_{s\geq0}\{\,s\,|\,\zeta_\cAM(s)=\infty\}.
  $$
\end{definition}
Note that, if $s_\cAM<\infty$, we also have that 
$s_\cAM=\inf_{s\geq0}\{s\,|\,\zeta_\cAM(s)<\infty\}$.
\begin{example}
  \label{ex:AB}
  Let $\cA$ be the (parabolic) gasket generated by 
  $$
  A_1=\begin{pmatrix}1&1\cr 0&1\cr\end{pmatrix},\;
  A_2=A_1\in\SL_2(\bN).
  $$
  In this case $\|A_{I_k}\|=k$, so that 
  $\zeta_\cA(s) = \sum_{k\in\bN} 2^kk^{-s}$
  diverges for all $s$, i.e. $s_\cA=\infty$.

  On the contrary, let $\cB$ be the (hyperbolic) gasket generated by 
  $$
  B_1=\begin{pmatrix}2&1\cr 1&1\cr\end{pmatrix},\;
  B_2=B_1\in\GL_2(\bN).
  $$
  Then $\|B_{I_k}\|=F_{2k+2}$ where $F=(0,1,1,2,3,5,\dots)$
  is the Fibonacci sequence. Hence asymptotically 
  $\|B_{I_k}\|\simeq g^{2k}$, where $g=\frac{1+\sqrt{5}}{2}$ 
  is the golden ratio, and so 
  $\zeta_\cB(s)$ diverges or converges with 
  $\sum_{k\in\bN} 2^k g^{-2s k},$
  i.e. $s_\cB=\frac{1}{2\log_2g}$.
\end{example}
Next proposition shows that norms of matrices in $\cAM$ have 
the same asymptotic properties of those in $\cA$ for every
$M\in\GL_n(K)$:
%
\begin{proposition}
  Let $\cA$ be a $m$-gasket and $M\in\GL_n(K)$. 
  Then $\cA_M$ is a $m$-gasket and $s_\cAM=s_\cA$.
%
\end{proposition}
\begin{proof}
 It is a direct consequence of 
 $n\|P\|\|M\|\geq\|PM\|\geq\frac{\|P\|}{n\|M^{-1}\|}$.
\end{proof}
Showing that hyperbolic gaskets have a finite exponent does not require
any effort:
%
\begin{proposition}
  Let $\cA:\cI^m\to M_n(K)$ be a hyperbolic gasket. Then $s_\cA$ is finite.
\end{proposition}
\begin{proof}
  Since $\cA$ is hyperbolic then $\|A_I\|\geq A\alpha^{|I|}$ for some $A>0$ and
  $\alpha>1$. Hence 
  $$
  \zeta_\cA(s)\leq \sum_{k=0}^\infty m^kA^s\alpha^{ks}
  =
  \sum_{k=0}^\infty A^s m^{k(1-s\log_m\alpha)},
  $$ 
  so that $\zeta_\cA(s)\leq\infty$ for $s\geq\log_\alpha m$, 
  namely $s_\cA\leq\log_\alpha m$. 
\end{proof}
As Example~\ref{ex:AB} shows, 
proving a similar statement for the parabolic case we must
require some growth condition on the norms of products.
\begin{definition}
  We say that a gasket $\cA:\cI^m\to M_n(K)$ is {\em fast} if 
  %
  there is a constant $c>0$
  such that
  $$
  \|A_{IJ}\| \geq c \|A_I\| \|A_{J}\|.
  $$ 
  for every multi-indices $I\in\cI^m$ and $J\in\cJ^m\cdot\cI^m$.
  We call 
  $$
  c_\bA=\inf_{\substack{I\in\cI^m\\ J\in\cJ^m\cdot\cI^m}}\frac{\|A_{IJ}\|}{\|A_I\| \|A_{J}\|}
  $$ 
  the {\em coefficient} of the gasket.
\end{definition}
%
%
\begin{example}
  Consider the parabolic and hyperbolic gaskets of Example~\ref{ex:AB}.
  Since 
  $$
  \|A_1^{k'} A_2 A_1^k\| = \|A_1^{k+k'+1}\|=1+k+k'
  $$ 
  and $\inf_{k,k'\geq1}\{(1+k+k')/(kk')\}=0$, $\cA$ is not fast.

  On the contrary, since any product of $N=k+k'$ copies of $B_{1,2}$ 
  is equal to $B^N_{1}$ and
  $$
  \|B_1^{k'}B_1^k\|=\|B_1^{k+k'}\|=F_{2k+2k'} > F_{2k'} F_{2k-2} = \|B_1^{k'}\|\|B_1^{k-1}\|,
  $$
  then $\cB$ is fast with $c_\cB\geq1$.
\end{example}
\begin{example}
  \label{ex:dynfast}
  The (parabolic) cubic gasket $\boC_2$ 
  of Example~\ref{ex:dyn} is fast.
  Indeed consider first $J=21L$, 
  with $C_L=\begin{pmatrix}a&b\cr c&d\cr\end{pmatrix}$, so that
  $$
  C_J = \begin{pmatrix}1&0\cr 1&1\cr\end{pmatrix}
        \begin{pmatrix}1&1\cr 0&1\cr\end{pmatrix}
        \begin{pmatrix}a&b\cr c&d\cr\end{pmatrix}
      =
        \begin{pmatrix}a+c&b+d\cr a+2c&b+2d\cr\end{pmatrix}.
  $$
  Clearly $\|C_J\|=\max\{a+2c,b+2d\}\leq 2\max\{a+c,b+d\}$ and therefore
  $$
  \|MC_J\|\geq\frac{1}{2}\|M\|\|C_J\| 
  $$
  for every $M\in SL_2(\bN)$. The same argument applies to $J=12L$.
  Since $\|C_1^{k'}C_1C_2C_1^k\|=k'(k+1)$, $\|C_1^{k'}\|=k'$ and
  $\|C_1C_2C_1^k\|=2k+1$ it follows at once that in fact $c_{\boC_2}=1/2$.
\end{example}
\begin{example}
  A hyperbolic gasket is not necessarily fast. Consider for instance
  $\cA:\cI^2\to GL_2(\bZ[i])$ with
  $$
  A_1 = \begin{pmatrix}2&0\cr 0&\alpha\cr\end{pmatrix},
  A_2 = \begin{pmatrix}\beta&0\cr 0&2\cr\end{pmatrix},
  $$
  where $|\alpha|,|\beta|=1$.
  $\cA$ is hyperbolic since every $A_I$, $|I|=2k$, contains an entry 
  with modulus $2^{k'}$ and $k'\geq k$.
  On the other side $\cA$ is not fast. Indeed  
  $\|A_1^kA_1A_2^{k'}\|=2^{-k}\|A_1^k\|\|A_1A_2^{k'}\|$ for all $k\leq k'$
  and so 
  $$\inf_{I\in\cI^2,J\in\cJ^2}\frac{\|A_{IJ}\|}{\|A_I\|\|A_J\|}=0.$$
\end{example}
\begin{proposition}
  If $\cA$ is a fast gasket with coefficient $c$, the gasket $\lambda\cA$ 
  is a fast gasket for every $\lambda\neq0$ with coefficient $c/|\lambda|$.
\end{proposition}
\begin{proof}
  $\|\lambda A_{IJ}\|=|\lambda|\|A_{IJ}\|\geq c|\lambda|\|A_I\|\|A_{J'}\|=
  \frac{c}{|\lambda|}\|\lambda A_I\|\|\lambda A_J\|$
\end{proof}
Our main {\em algebraic} results on the exponent of a gasket are the following:
\begin{theorem1}[Exponent of a fast gasket]
  If $\cA:\cI^m\to M_n(K)$ is a fast gasket 
  then $0<s_\cA<\infty$.
\end{theorem1}
In the more detailed discussion in Section~\ref{sec:mu}
we also show how to build explicitly two sequences of monotonically 
decreasing functions 
$f_{\cA,k}(s)$ and $g_{\cA,k}(s)$ such that, for all $k$ from some $\bar k$ on,
$s_\cA\in[g_{\cA,k}^{-1}(1),f_{\cA,k}^{-1}(1)]$ and 
$|f_{\cA,k}^{-1}(1)-g_{\cA,k}^{-1}(1)|\leq a\log k$ for some $a>0$.
These sequences will be used in Section~\ref{sec:fract} to evaluate
analytical bounds for the exponents of a few semigroups.
%
\begin{theorem2}[Alternate characterization of the exponent of a gasket]
  Let $\cA:\cI^m\to M_n(K)$ be a gasket with $s_\cA<\infty$. Then the function
  $$
  \xi_\cA(s) = \lim_{k\to\infty}\left[\sum_{|I|=k}\|A_I\|^{-s}\right]^\frac{1}{k}
  $$
  is a well defined log-convex positive strictly decreasing function
  and satisfies the following properties:
  \begin{enumerate}
    \item $\xi_\cA(0)=m$;
    \item $\xi_\cA(s)>1$ for $s<s_\cA$;
    \item $\xi_\cA(s_\cA)=1$;
    \item $\xi_\cA(s)<1$ for $s>s_\cA$ if $\cA$ is a hyperbolic gasket;
    \item $\xi_\cA(s)=1$ for $s>s_\cA$ if $\cA$ is a fast parabolic gasket.
  \end{enumerate}
\end{theorem2}
\begin{theorem3}[Exponential growth of norms]
  Let $\cA:\cI^m\to M_n(K)$ be a gasket with $s_\cA<\infty$. Then
  $$
  \lim_{k\to\infty}\frac{\log N_\cAM(k)}{\log k} = s_\cA
  $$
  for every $M\in\GL_n(K)$.
\end{theorem3}
It is natural to ask whether this result can be made stronger, namely
whether $N_\cA(k)\asymp k^{s_\cA}$. This leads to the
following:
\begin{question}
  Based on numerical investigations, Boyd~\cite{Boy82} conjectured that, 
  in case of the Apollonian gasket, the number $N(k)$ of circles whose curvature 
  is not larger than $k$ (which, in our setting, corresponds to the number
  of Hirst matrices whose norm is not larger than $k$) were only {\em weakly}
  asymptotic to $k^{s}$, where $s$ is the dimension of the Apollonian gasket,
  namely that
  $$
  N(k)\simeq Ak^{s}\log^t(k/B)
  $$ 
  for some $A,B,t>0$.
  Recently Kontorovich and Oh~\cite{KH11} disproved this conjecture by 
  showing that actually $N(k)$ is {\em strongly} asymptotic to $k^{s}$, 
  namely 
  $N(k)\asymp k^{s}$
  for all $k\in\bN$.
  
  We pose the following question: is $N_\cA(k)$ {\em strongly}
  asymptotic to $k^{s_\cA}$ for every gasket $\cA$ with a finite
  exponent? Or is there some extra condition that must be put
  on $\cA$ to ensure this behaviour?
\end{question}

\noindent
{\bf Hausdorff dimension of limit sets of discrete subsemigroups of real and complex
projective automorphisms.} 
Recall that the projective space $K P^{n-1}$ is defined as the set of
all 1-dimensional linear subspaces of $K^n$. We denote by $[v]$ the 1-dimensional
subspace of $K$ containing $v$, so that $[\cdot]:K^n\to KP^{n-1}$ is a projection
with kernel $K\setminus\{0\}$. 
Since linear maps preserve linear subspaces, every automorphism $f:K^n\to K^n$ 
induces a projective automorphism $\psi_f:P K^{n-1}\to P K^{n-1}$ defined by
$\psi_f([v])=[f(v)]$. 

We endow $P K^{n-1}$ with the {\em round} metric tensor, 
namely the metric with constant curvature equal to 1, and with the distance
and measure induced by it. In every compact of every chart of $K P^{n-1}$ this 
distance is Lipschitz equivalent to the Euclidean distance in the chart and the 
measure is in the same measure class of the Lebesgue measure in the chart,
so that we can forget about the definition and work directly with the
Euclidean distance and Lebesgue measure in any chart.

In this paper by {\em iterated function system} (IFS) on a metric space $(X,d)$
we mean a homomorphism $F:\cI^m\to C(X)$, where $C(X)$ is the set of continuous
maps from $X$ into itself, such that the generators $f_i\bydef F(i)$, 
$i=1,\dots,m$, are non-expansive maps, namely there exist $K_i\in(0,1]$ such that
$d(f_i(x),f_i(y))\leq K_i d(x,y)$ for all $x,y\in X$. If $K_i<1$ for all $i$
then we say that the IFS is {\em hyperbolic}, otherwise that is {\em parabolic}.
Recall that to a hyperbolic IFS $F$ on $X$ it corresponds a unique compact 
set $R_F\subset X$, which is also equal to the set of limit points of the orbit
of almost every point $x\in X$ under the action of $F$,
whose Hausdorff dimension, which we denote by $\dim_H R_F$,
is often non-integer and can be evaluated through some property of the $f_i$, 
e.g. their coefficients $K_i$ or, if regular enough, their derivatives 
(e.g. see Chapter 9 of~\cite{Fal90} for more details and examples).
Finally, recall that the $f_i$ satisfy the {\em open set} condition
if there exists a relatively compact subset $V\subset X$ such that 
$V\supset\sqcup_{i=1}^m f_i(V)$. Note that this condition in particular
implies that there is no relation between the $f_i$.

In a seminal paper D. Sullivan~\cite{Sul84} 
related the Hausdorff dimension of the limit set $R_\Gamma\subset\bC P^1$ 
of a (geometrically finite) Kleinian (i.e. discrete) subgroup 
$\Gamma\subset PSL_2(\bC)$ to a critical exponent of $\Gamma$ defined
in the context of hyperbolic geometry (the recent results by 
Kontorovich and Oh quoted in Open Question~1 are in fact based on this
result). In the same spirit here we relate (partly in form of conjecture) 
the exponent $s_\bA$ of a 
sub{\em semi}{\hskip1pt}group of $SL^\pm_n(\bR)$ or $SL_n(\bC)$ 
to the Hausdorff or Box dimensions of the limit set $R_\bA\subset KP^{n-1}$, 
$K=\bR$ or $\bC$, of the subsemigroup of $PSL^\pm_n(\bR)$ or $PSL_n(\bC)$ 
naturally induced by $\bA$. This somehow complements recent results
on the geometry of residual sets of real projective IFS by Barnsley 
and Vince~\cite{BV10} and on complex projective IFS by Vince~\cite{Vin12}.

We denote by $f_1,\dots,f_m$ linear volume-preserving automorphisms of 
$K^n$, by $A_i\in SL^\pm_n(K)$ the matrices representing the $f_i$ 
in some coordinate system, by $\bA$ the semigroup generated by the $A_i$
and by $\psi_i\in PSL^\pm_n(K)$ the projective automorphism naturally
induced by $f_i$.

Our main {\em geometric} results on the exponent of a gasket are the following:
\begin{theorem4}[Semigroups of $PSL^\pm_2(\bR)$]
  Let $f_1,\dots,f_m:\Rt\to\Rt$ be such that 
  the induced maps $\psi_i\in PSL^\pm_2(\bR)$ satisfy the open 
  set condition with respect to some proper subset $V\subset\bR P^1$
  and that, in some affine chart $\varphi:\bR P^1\to\bR$, 
  the $\psi_i$ are contractions on $\varphi(\overline{V})$
  with respect to the Euclidean distance.
  Let $R_\bA=\cap_{k=1}^\infty\left(\cup_{|I|=k}\psi_{I}(V)\right)$ 
  be the corresponding residual set. Then $2\dim_HR_\bA=s_\bA$.
\end{theorem4}
The simplicity of the geometry of $\bR P^1$ suggests that the
theorem above can be extended to more general families of $f_i$, 
leading to the following:
\begin{question}
  Let $\bA\subset SL_2^\pm(\bR)$ be a free finitely generated semigroup 
  with finite exponent $s_\bA$ and $\Psi_\cA$ the corresponding free discrete 
  subsemigroup of $PSL_2^\pm(\bR)$. Under which general assumptions is 
  the Hausdorff dimension of the limit set of $\Psi_\bA$ related to $s_\bA$?
\end{question}
An identical theorem holds in the complex case:
\begin{theorem5}[Semigroups of $PSL_2(\bC)$]
  Let $f_1,\dots,f_m:\bC^2\to\bC^2$ be such that 
  the induced maps $\psi_i\in PSL_2(\bC)$ satisfy the open 
  set condition with respect to some proper subset $V\subset\bC P^1$
  and that, in some affine chart $\varphi:\bC P^1\to\bR$, 
  the $\psi_i$ are contractions on $\varphi(\overline{V})$
  with respect to the Euclidean distance.
  Let $R_\bA=\cap_{k=1}^\infty\left(\cup_{|I|=k}\psi_{I}(V)\right)$ 
  be the corresponding residual set. Then $2\dim_HR_\bA=s_\bA$.
\end{theorem5}
Even in this case there is evidence that the claim of the theorem
can be generalized further, at least to some parabolic case, namely
when the $\psi_i$ are not contractive but just non-expanding. This
is, for example, the case for the Apollonian gasket as shown in the
introduction.

For $n>2$ things get geometrically much more complicated for both $\bR$ 
and $\bC$ and we cannot claim any general result. 
In order to provide motivations for further studies of this case
and as a source of examples of non-trivial fast gaskets,
in Section~\ref{sec:fract} we introduce for free semigroups 
$\bA\subset SL^\pm_n(\bR)$ 
the notion of {\em real} and {\em complex self-projective Sierpinski 
gaskets}, the name being due to the fact that to such $\bA$'s 
it correspond the same ``cut-out'' construction of the standard Sierpinski
gasket 
and of its multi-dimensional generalizations.
In particular to each such gaskets it corresponds a curvilinear simplex
$T_\bA\subset\bR P^{n-1}$ invariant by the action induced on it by $\bA$ 
and a compact invariant set $R_\bA\subset T_\bA$ obtained by subtracting 
recursively curvilinear polyhedra from $T_\bA$. 

In particular in Section~\ref{sec:rSG}
we show that in $PSL^\pm_n(\bR)$, $n\geq3$, there exist a 
smooth 1-parameter family of fast free semigroups $\bA_t$, $t\in[1,\infty)$, 
such that: 1) $\bA_t$ is a hyperbolic IFS for $t\in(1,4)$; 2) the extremal
gasket $\bA_1$ is the (parabolic, multi-dimensional generalization of the) 
cubic gasket; 3) $\bA_2$ is the (multi-dimensional real self-projective 
generalization of the) standard Sierpinski gasket. 
Siilarly, in Section~\ref{sec:cSG} we correspondingly show that 
in $PSL_2(\bC)$ there exist a smooth 
1-parameter family of semigroups $\bA_t$, $t\in[1/5,\alpha]$, where 
$\alpha\simeq0.651$, such that: 1) $\bA_t$ is a hyperbolic IFS for 
$t\in(1/5,\alpha)$; 2) the extremal semigroup $\bA_{1/5}$ is the (parabolic) 
Apollonian gasket and $\bA_{1/2}$ is the (complex self-projective 
generalization of the) standard Sierpinski gasket.
This shows that the cubic and Apollonian gaskets are both natural parabolic 
deformations of the standard Sierpinski gasket, the first in the real
setting and the second in the complex setting.

Finally, based on numerical and analytical results in Section~\ref{sec:rSG},
we claim the following about the box dimension of the residual set: 
\begin{conjecture}
  \label{thm:conj}
  Let $\bA\subset SL^\pm_n(\bR)$ be a real projective Sierpinski gasket
  and $R_\bA$ its residual set. 
  Then, under suitable natural assumptions, 
  $$(n+1)\dim_B R_\bA\geq n s_\bA.$$
\end{conjecture}
%
\section{Asymptotic growth of norms}
\label{sec:norms}
\subsection{The case $m=1$.}
\label{sec:m=1}
Let us briefly discuss the exponent of a gasket in the trivial case $m=1$, 
since these results will be useful later in this section. 

Here $\cI=\bN$, $\cA$ is generated by a single matrix $A\in M_n(K)$, $K=\bR$ or
$\bC$, 
and $\cA(k)=A^k$. 
In order for $\cA$ to be a gasket then it is necessary and sufficient that
$A$ has an eigenvalue with modulus larger than 1 (in which case it 
will be a hyperbolic gasket) or that has a single eigenvalue
of modulus 1 and its spectral norm larger than 1 (in which case it 
will be a parabolic gasket).

If $\cA$ is hyperbolic then $\|A^k\|$ grows exponentially with $k$
and so $s=0$ is the only exponent that can make the series
$\zeta_\cA(s)=\sum_{k\in\bN}\|A^k\|^{-s}$ divergent, i.e. $s_\cA=0$.

If $\cA$ is parabolic then its generator $A$ has only eigenvalues
of modulus 1 and therefore its norm grows as some polynomial of degree
$d<n$.
Hence 
$\zeta_\cA(s)$ diverges for $s\leq 1/d$
and is finite for $s>1/d$, i.e. $s_\cA=1/d$. 

Now consider the number $N_\cA(r)$ of powers of $A$ 
whose norm is not larger than $r$. When $\|A^k\|$ is a polynomial of order $d$ 
their number grows as $r^{1/d}$, so that the limit
$\lim_{k\to\infty}\log N_\cA(r)/\log r=1/d$ equals $s_\cA$.
When $\|A^k\|$ grows exponentially then $N_\cA(r)$ grows logarithmically
so that, again, $\lim_{k\to\infty}\log N_\cA(r)/\log r=0$
equals $s_\cA$. Incidentally, this proves Theorems 1 and 3 for $m=1$.

\subsection{The case $m>1$.}
When there is more than one generator things are qualitatively different: 
the number of terms of order $k$ (i.e. products of $k$ generators)
increases exponentially and there can be coexistence of polynomial and 
exponential growths of norms for terms of the same order. E.g. 
we already pointed out this behaviour for the $\boC_2$ semigroup
in Example~\ref{ex:dynfast} and the same happens for all semigroups
$\boC_n$ (see Section~\ref{sec:cubic}).
%
%

Following Boyd's arguments in~\cite{Boy73b} we are able to find upper and
lower bounds for $\zeta_\cAM(s)$ using only the series of the norms of the
``diagonal'' terms $A_DM$, $D\in\cD^m$, and of those ``next-to-diagonal''
ones $A_JM$, $J\in\cJ^m$.
The key point of the next arguments is the following elementary recursive
re-writing of the zeta function: 
\begin{equation}
  \label{eq:SAM}
  \zeta_{\cAM}(s) = 
  \sum_{D\in\cD^m} \frac{1\phantom{^s}}{\|A_DM\|^s} + \sum_{J\in\cJ^m} \zeta_{\cA_{ A_JM}}(s).
\end{equation}
\subsubsection{A fundamental set of inequalities}
Using (\ref{eq:SAM}) we can now write the following fundamental inequality:
%
\begin{proposition}
  Let $\cA:\cI^m\to M_n(K)$ be a fast gasket with coefficient
  $c_\cA$ and let $J\in\cJ^m\cdot\cI^m$
  and $I\in\cI^m$.
  Then the following inequalities hold:
  \begin{gather}
    \label{eq:SAAA}
    n^{-s}\|A_J\|^{-s} \zeta_\cA(s)
    \leq 
    \zeta_{\cA_{A_J}}(s) 
    \leq 
    c^{-s}_\cA\|A_J\|^{-s}\zeta_\cA(s) \\
    \label{eq:SAMinLeft}
    \zeta_{\cA_{A_I}}(s)
    \geq
    \nu_{\cA_{A_I}}(s) + n^{-s} \mu_{\cA_{A_I}}(s)\zeta_\cA(s)\\
    \label{eq:SAMinRight}
    \zeta_{\cA_{A_I}}(s)
    \leq
    \nu_{\cA_{A_I}}(s) + c^{-s}_\cA \mu_{\cA_{A_I}}(s)\zeta_\cA(s),
  \end{gather}
  %
  where
  $\displaystyle
  \nu_{\cAM}(s)=\sum_{D\in\cD^m}\frac{1}{\|A_DM\|^{s}}$
  and 
  $\displaystyle
  \mu_{\cAM}(s)=\sum_{J\in\cJ^m}\frac{1}{\|A_JM\|^{s}}$.
\end{proposition}
\begin{proof}
  The left side of (\ref{eq:SAAA}) is a direct consequence of  
  (\ref{eq:PQ}), its right side of the definition of fast gasket.
  The starting point to prove inequalities 
  (\ref{eq:SAMinLeft},\ref{eq:SAMinRight}) is (\ref{eq:SAM}), from which
  we get 
  $$
  \zeta_{\cA_{A_I}}(s) = 
  \nu_{\cA_{A_I}}(s) + \sum_{J\in\cJ^m} \zeta_{\cA_{A_{JI}}}(s).
  $$
  Applying the right side of (\ref{eq:SAAA}) to the summation above we get that
  $$
  \sum_{J\in\cJ^m} \zeta_{\cA_{A_{JI}}}(s)
  \leq
  \sum_{J\in\cJ^m}\frac{1}{c^s_\cA\|A_{JI}\|^s}\zeta_\cA(s) 
  =
  \frac{1}{c^s_\cA}\mu_{\cA_{A_I}}(s)\zeta_\cA(s). 
  $$
  Hence (\ref{eq:SAMinRight}) follows and analogously it is proven (\ref{eq:SAMinLeft}).
\end{proof}
\begin{remark}
  \label{rmk:nu}
  The function $\nu_\cAM$ has the same complexity of the zeta functions of $1$-gaskets.
  Indeed let $\cA_i=\langle A_i\rangle$ be the 1-gaskets generated by the $m$ 
  generators of $\cA$. Then $\nu_\cAM(s)=\sum_{1\leq i\leq m}\zeta_{\cA_i}(s)$.
  The ultimate idea of this section is to exploit this fact to find upper
  and lower bounds for $\zeta_\cA$ using the much simpler $\nu_\cA$.
\end{remark}
\begin{remark}
  \label{rmk:munu}
  For finite $s$, $\mu_{\cA}$ and $\nu_{\cA}$, converge or diverge
  together. Indeed if $J=iD\in\cJ^m$ then $J'=D\in\cD^m$ and 
  $$
  \frac{\|A_D\|}{n\|A_i^{-1}\|}\leq\|A_{iD}\|\leq n\|A_i\|\|A_D\|,
  $$
  so that
  $$
  (m-1)n^{-s}\min_{1\leq i\leq m}\|A_i\|^{-s}\nu_\cA(s)
  \leq 
  \mu_\cA(s)
  \leq 
  (m-1)n^s\max_{1\leq i\leq m}\|A_i^{-1}\|^s\nu_\cA(s).
  $$
  The same is true for $s\to\infty$ after some finite number of 
  terms (those of ``too small'' norm) is removed from the two series. 
  For similar reasons $\nu_\cAM$ and $\mu_\cAM$ converge
  or diverge together with $\nu_\cA$ for every $M\in\GL_n(K)$.
\end{remark}
Let us now examine closely the two inequalities 
(\ref{eq:SAMinLeft},\ref{eq:SAMinRight}) 
for $I=0$. The first one becomes 
\begin{equation}
  \label{eq:left}
  \zeta_\cA(s) \geq \nu_\cA(s) + n^{-s} \mu_\cA(s) \zeta_\cA(s)
\end{equation}
%
Since we are going to use this inequality to get lower bounds for 
$\zeta_\cA(s)$, we can proceed without loss of generality 
by assuming that $\zeta_\cA(s)<\infty$.
Then we get that, for $n^{-s} \mu_\cA(s)<1$,
\begin{equation}
  \label{eq:lb}
  \zeta_\cA(s)\geq \frac{\nu_\cA(s)}{1-n^{-s} \mu_\cA(s)}.
\end{equation}
Analogously the right one becomes 
\begin{equation}
  \label{eq:right}
\zeta_\cA(s) \leq \nu_\cA(s) + \c^{-s}_\cA \mu_\cA(s) \zeta_\cA(s).
\end{equation}
This time though, since we are going to use this inequality to provide 
upper bounds to the zeta function, we need to truncate the infinite series
to a finite sum to ensure we are dealing with finite numbers.
A natural recursive definition, inspired by the structure of (\ref{eq:SAM}), is
$$
\begin{tabular}{lcl}
  $\zeta^{0}_\cAM(s)$&$=$&$\nu^0_\cAM(s)$\cr
  $\zeta^{\ell}_\cAM(s)$&$=$&$\displaystyle \nu^{\ell}_\cAM(s)
  + \sum_{J\in\cJ^m_{\ell+1}} \zeta^{\ell+1-|J|}_{\cA_{A_JM}}(s),\;\ell\geq1$\cr
\end{tabular}
,
$$
where $\nu^{\ell}_\cAM$ is the necklace sum truncated at the order $\ell$.
\begin{proposition}
  Consider the sets $\cP^m_\ell\subset\cI^m$ recursively defined as
  $$
  \begin{tabular}{lcl}
    $\cP^m_0$&$=$&$\cD^m_0,$\cr
    $\cP^m_{\ell}$&$=$&$\displaystyle\cD^m_\ell
    \bigcup\left[\bigcup_{J\in\cJ^m_{\ell+1}}\cP^m_{\ell+1-|J|}\cdot J\right],
    \;\ell\geq1.$
  \end{tabular}
  $$
  Then the following properties hold:
  \begin{enumerate}
    \item $\cP^m_{\ell}\subset\cP^m_{\ell+1}$ for all $\ell\geq0$;
    \item $\cup_{\ell\geq0}\cP^m_\ell=\cI^m$;
    \item $\zeta^\ell_\cAM(s)=\sum_{I\in\cP^m_\ell}\|A_IM\|^{-s}.$
  \end{enumerate}
\end{proposition}
\begin{proof}
  1. We prove this by induction assuming that $\cP^m_{k}\subset\cP^m_{k+1}$
  for all $k\leq\ell-1$. Now, if $I\in\cP^m_{\ell}$ then either $I\in\cD^m_\ell$ 
  or $I=PJ$ with $J\in\cJ^m_{\ell+1}$ and $P\in\cP^m_{\ell+1-|J|}$. 
  In the first case $I\in\cP^m_{\ell+1}$ since $\cD^m_\ell\subset\cD^m_{\ell+1}$.
  In the second case we have that $\ell+1-|J|\leq \ell-1$ since every element
  of $\cJ^m$ has at least rank two and therefore, by the inductive hypothesis, 
  $\cP^m_{\ell+1-|J|}\subset\cP^m_{\ell+1-|J|+1}$.
  Hence, by definition, $I\in\cP^m_{\ell+2-|J|}\cdot J\subset\cP^m_{\ell+1}$.

  2. Notice first of all that every index $I\in\cI^m$ either belongs to
  $\cD^m$ (and so to some $\cP^m_\ell$)
  or it can be factored out as a product $I=J_1\cdots J_k$ with 
  $J_i\in\cJ^m$. This factorization simply consists in singling out the 
  patterns of the form $i_1\neq i_2=i_3=\dots=i_p$ inside $I$ and is 
  clearly unique. Let $K=\max_{1\leq i\leq k}|J_i|$. By construction
  $J_i\in\cP^m_K$ for every $1\leq i\leq k$. Then 
  $\cP^m_{2K+1}\supset\bigcup_{J\in\cJ^m_{2K+2}}\cP^m_K\cdot J$ 
  contains all products $J_i J_j$, $1\leq i,j\leq k$ and similarly
  $\cP^m_{kK+1}$ contains all possible products of $k$ of the $J_i$,
  namely $I\in\cP^m_{kK+1}$.

  3. Let us write $\zeta^{\ell}_\cAM(s)=\sum_{I\in\cG^\ell}\|A_IM\|^{-s}$.
  Then if $I\in\cG^\ell$ either $\|A_IM\|^{-s}$ appears in $\nu^\ell_\cAM(s)$, 
  in which case $I\in\cD^m_\ell$ by definition, or in $\zeta^{\ell+1-|J|}_{\cA_{A_JM}}(s)$, 
  in which case $I=KJ$ with $K\in\cG^{\ell+1-|J|}$. This is exactly
  the rule that defines recursively the $\cP^m_\ell$. Since 
  we also have by definition of $\zeta^0_\cAM(s)$ that $\cG^0=\cD^m_0$ 
  it follows that $\cG^\ell=\cP^m_\ell$.
\end{proof}
\begin{corollary}
  Let $\cA$ be a fast gasket. Then the 
  $\zeta^{\ell}_\cAM(s)$ satisfy the following properties:
  %
  \begin{align}
    &\zeta^{\ell}_\cAM(s)\leq\zeta^{\ell+1}_\cAM(s)\hbox{ for all }\ell\geq0;
    \label{eq:ZAm} \\
    &\lim_{\ell\to\infty}\zeta^{\ell}_\cAM(s)=\zeta_\cAM(s);\label{eq:ZAmLim}\\
    &n^{-s}\|A_J\|^{-s} \zeta^\ell_\cA(s)
    \leq 
    \zeta^\ell_{\cA_{A_J}}(s) 
    \leq 
    c^{-s}_\cA\|A_J\|^{-s}\zeta^\ell_\cA(s);\label{eq:SAMinRightEll} \\
    &\zeta^{\ell}_{\cA_{A_I}}(s)\leq \nu_{\cA_{A_I}}(s) 
      + \c^{-s}_\cA \mu_{\cA_{A_I}}(s) \zeta^{\ell}_\cA(s).\label{eq:SAAAEll}
  \end{align}
  %
\end{corollary}
\begin{proof}
  (\ref{eq:ZAm},\ref{eq:ZAmLim}) are a direct consequence of points 1. and 2. 
  of the previous proposition.
  (\ref{eq:SAMinRightEll}) is a direct consequence of (\ref{eq:PQ}) (left) 
  and of the definition of fast gasket (right). 
  Using the rhs of (\ref{eq:SAMinRightEll}) and then (\ref{eq:ZAm}) we get that
  $$
  \zeta^{\ell}_\cAM(s)
  \leq
  \nu^{\ell}_\cAM(s)+c^{-s}\sum_{J\in\cJ^m_{\ell+1}}\|A_JM\|^{-s}\zeta^{\ell+1-|J|}_{\cA}(s)
  \leq 
  \nu^{\ell}_\cAM(s)+c^{-s}\mu^{\ell+1}_\cAM(s)\zeta^{\ell}_{\cA}(s)
  $$
  from which (\ref{eq:SAAAEll}) follows.
\end{proof}
%
%
%
Using the monotonicity in $\ell$ of 
$\nu^{\ell}_\cAM(s)$ and $\mu^{\ell}_\cAM(s)$ and setting $M=\one_n$ we get that,
in particular,
$$
\zeta^{\ell}_\cA(s) \leq \nu_\cA(s) + \c^{-s}_\cA \mu_\cA(s) \zeta^\ell_\cA(s).
$$
From this we deduce that, when $\mu_\cA(s)<c^s_\cA$,
$\zeta^{\ell}_\cA(s)\leq \frac{\nu_\cA(s)}{1-\c^{-s}_\cA \mu_\cA(s)}$
for all $\ell$ and therefore, finally,
\begin{equation}
  \label{eq:ub}
  \zeta_\cA(s) \leq \frac{\nu_\cA(s)}{1-c^{-s}_\cA \mu_\cA(s)}.
\end{equation}
In the next section we will use these bounds to build bounds
for $s_\cA$.
\subsubsection{The exponent of a fast gasket}
\label{sec:mu}
%
%
The idea of next theorem comes from the following observation, valid in
the particular case when all generators $A_i$ of $\cA$ have non-negative 
coefficients and norms larger than the maximum between 1 and the inverse 
of the coefficient $c_\cA$ of the gasket.

Consider (\ref{eq:lb}). The function $g_\cA(s) = n^{-s}\mu_\cA(s)$
is clearly monotonically decreasing with $s$, defined on 
$(0,\infty)$ and $g(0,\infty)=(0,\infty)$. 
In particular the set $g_\cA(s)<1$ is not empty and
therefore (\ref{eq:lb}) holds for $s>s_g$, with $s_g=g_\cA^{-1}(1)$.
When we let $s\to s_g^+$ the right hand side of (\ref{eq:lb}) goes to 
infinity so that $\zeta_\cA(s_g)$ diverges, namely $s_\cA\geq s_g$. 

Now consider (\ref{eq:ub}). The function $f_\cA(s)=c^{-s}\mu_\cA(s)$ 
satisfies the same properties listed above for $g_\cA$. 
Let $s_f=f_\cA^{-1}(1)$. Then $f_\cA(s)<1$ for $s>s_f$ 
so that (\ref{eq:lb}) holds.
This proves that $\zeta_\cA(s)<\infty$ for all $s>s_f$,
namely $s_f\geq s_\cA$.

This simple argument not only grants us that $s_\cA$ is 
finite but also provides for it non-trivial lower and upper 
bounds.
In order to obtain a similar result in full generality  
we need to refine (\ref{eq:ub}). 
This will lead us to refine also (\ref{eq:lb}) and to generate 
a pair of sequences converging to $s_\cA$ from the left and 
from the right at logarithmic speed. 
The idea is to apply over and over recursively first the inequalities 
(\ref{eq:SAMinLeft},\ref{eq:SAMinRight}) and then the inequality 
(\ref{eq:SAAA}) to the truncated zeta function.

The starting point is the sequence of sets of multi-indices
$\cQ_{\cA,k}$ built as follows. We define $\cQ^0_{\cA,k}=\cJ^m$. 
Then we consider the sets recursively defined as
$$
\cQ^\ell_{\cA,k} = 
  \left\{J\in\cQ^{\ell-1}_{\cA,k}\,|\,\|A_J\|> k\right\}
  \bigcup
  \cJ^m\cdot\left\{J\in\cJ^m\,|\,J\in\cQ^{\ell-1}_{\cA,k}, \|A_J\|\leq k\right\}
$$
with $\ell\geq1$.
%
\begin{proposition}
  For every gasket $\cA$ and every $k>0$ there exists a $\bar\ell$ such that
  $\cQ^{\bar\ell}_{\cA,k}=\cQ^{\bar\ell+1}_{\cA,k}$.
\end{proposition}
\begin{proof}
  The only thing that the recursive algorithm does is replacing the indices 
  of $\cQ^{\ell}_{\cA,k}$ corresponding to matrices with norm not larger 
  than $k$ with indices of higher order. The set $\cQ^{\ell+1}_{\cA,k}$
  thus obtained can still contain indices corresponding to matrices 
  with norm not larger than $k$ but in a finite number of steps all
  such indices will disappear because, by definition, there is only
  a finite amount of them. Hence there exists a finite $\bar\ell$
  such that $\|A_I\|>k$ for all $I\in\cQ^{\bar\ell}_{\cA,k}$.
  Then the algorithm leaves $\cQ^{\bar\ell}_{\cA,k}$ unchanged.
\end{proof}
\begin{definition}
  We use the notation
  $\cQ_{\cA,k}=\cQ^{\bar\ell}_{\cA,k}$ and, correspondingly, 
  \begin{align}
    f^\ell_{\cA,k}(s) &= c^{-s}_\cA\sum_{J\in\cQ^\ell_{\cA,k}}\|A_J\|^{-s},&
    f_{\cA,k}(s) &= c^{-s}_\cA\sum_{J\in\cQ_{\cA,k}}\|A_J\|^{-s}\\
    g^\ell_{\cA,k}(s) &= n^{-s}\sum_{J\in\cQ^\ell_{\cA,k}}\|A_J\|^{-s},&
    g_{\cA,k}(s) &= n^{-s}\sum_{J\in\cQ_{\cA,k}}\|A_J\|^{-s}.
  \end{align}
\end{definition}
Note that $f_{\cA,0}(s)=f_\cA(s)$ and $g_{\cA,0}(s)=g_\cA(s)$.
Next proposition though shows that, 
for $k$ big enough, $f_{\cA,k}$ and $g_{\cA,k}$ have a nicer 
behaviour than $f_\cA$ and $g_\cA$ do for a general $\cA$.
\begin{proposition}
  \label{thm:bigK}
  For every fast gasket $\cA$ there exists a 
  $\bar k$ and a $\gamma_{\cA,k}<\infty$ such that $f_{\cA,k}$ and $g_{\cA,k}$ 
  are strictly decreasing continuous functions of $s$ defined in 
  $(\gamma_{\cA,k},\infty)$ and with image $(0,\infty)$ for all $k>\bar k$.
\end{proposition}
\begin{proof}
  By construction $f_{\cA,k}(s)$ and $g_{\cA,k}(s)$ are proportional to
  each other and, respectively with constants $c^{-s}_\cA$ and $n^{-s}$,  
  to the sum of a finite number of functions $\mu^{(k)}_{\cA_{A_I}}$
  defined as the series $\mu_{\cA_{A_I}}$ from which all terms 
  with norm smaller than $k$ have been subtracted. 

  By Remarks~\ref{rmk:nu} and~\ref{rmk:munu} then
  $$
  n^{-s}\min_{1\leq i\leq m}\|A_i\|^{-s}\nu^{(k)}_{\cA_{A_I}}(s)
  \leq 
  \mu^{(k)}_{\cA_{A_I}}(s)
  \leq
  (m-1)n^s\max_{1\leq i\leq m}\|A_i^{-1}\|^s\sum_{I\in\cG}\nu^{(k)}_{\cA_{A_I}}(s),
  $$
  where $\cG\subset\cI^m$ is some finite set of indices and the series 
  $\nu^{(k)}_{\cA_{A_I}}$ is equal to $\nu_{\cA_{A_I}}$ minus those terms with 
  indices $D\in\cD^m$ such that  $\|A_{JI}\|\leq k$ for all $J\in\cJ^m$ 
  with $J'=D$.
  In particular then also $f_{\cA,k}$ and $g_{\cA,k}(s)$ are bounded by
  a finite sum of series $\nu^{(k)}_{\cA_{A_I}}(s)$ and so,
  by the case $m=1$ of Theorem 1 discussed in Section~\ref{sec:m=1}, 
  they are finite on some connected non-empty 
  interval $(\gamma_{\cA,k},\infty)$. 

  Now let $\bar k = n^2\max_{1\leq i\leq m}\|A_i^{-1}\|\max_{1\leq i\leq m}\|A_i\|$.  
  For every $k>\bar k$ we have that, if the index $DI$ appears 
  in the series $\nu^{(k)}_{\cA_{A_I}}$ and $J$ is one of the indices
  such that $\|A_{JI}\|>k$,
  $$
  \|A_{DI}\|
  \geq
  \frac{\|A_{JI}\|}{n\|A_i\|}
  >
  \frac{k}{n\|A_i\|}
  >
  n\max_{1\leq i\leq m}\|A_i^{-1}\|,
  $$
  namely 
  $$
  \frac{n\displaystyle\max_{1\leq i\leq m}\|A_i^{-1}\|}{\|A_{DI}\|}\leq
  \frac{n^2\displaystyle\max_{1\leq i\leq m}\|A_i^{-1}\|\max_{1\leq i\leq m}\|A_i\|}{k}<1.
  $$
  This means that all summands in $f_{\cA,k}$ and $g_{\cA,k}$ are $s$-powers 
  of numbers uniformly bounded from above by some number strictly smaller 
  than 1 and therefore $f_{\cA,k}$ and $g_{\cA,k}$ are 
  strictly decreasing with $s$ and their image equals $(0,\infty)$.
\end{proof}
%
%
\begin{theorem1}
  Let $\cA$ be a fast gasket of $m$ matrices. Then $0<s_\cA<\infty$
  and both $s_{g,k}=g_{\cA,k}^{-1}(1)$ and $s_{f,k}=f_{\cA,k}^{-1}(1)$ 
  are uniquely defined and converge, respectively from left and right, 
  to $s_\cA$ as $k\to\infty$ with speed at least logarithmic. 
\end{theorem1}
\begin{proof}
  
  We start by showing that
  \begin{equation}
    \label{eq:zone}
    \zeta_\cA(s)\geq h_{\cA,k}(s) + g_{\cA,k}(s)\zeta_\cA(s)
  \end{equation}
  and
  \begin{equation}
    \label{eq:ztwo}
    \zeta_\cA(s)\leq h_{\cA,k}(s) + f_{\cA,k}(s)\zeta_\cA(s)
  \end{equation}
  for every $k$, where $h_{\cA,k}$ is some positive continuous function 
  that plays no role here.
  We prove it in detail for the most complicated case, namely the second one,
  when we need to pass through the partial sums $\zeta^{\ell}_\cA$.
  First we notice that (\ref{eq:SAAAEll}) writes as
  \begin{equation}
    \label{eq:T1.1}
  \zeta^\ell_{\cA_{A_I}}(s)\leq h^0_{\cA_{A_I},k} + f^0_{\cA_{A_I},k}(s)\zeta^\ell_\cA(s)
  \end{equation}
  after putting $h^0_{\cA_{A_I},k}(s)=\nu_{\cA_{A_I}}(s)$. 
  Now we write the definition of $\zeta^\ell_{\cA_{A_I}}(s)$ splitting
  the last term in two
  %
  \begin{equation}
    \label{eq:split}
    \zeta^{\ell}_{\cA_{A_I}}(s)
    = \nu^{\ell}_{\cA_{A_I}}(s) 
    + \sum_{\substack{J\in\cJ^m_{\ell+1}\\ \|A_J\|\leq k}} \zeta^{\ell+1-|J|}_{\cA_{A_{JI}}}(s) 
    + \sum_{\substack{J\in\cJ^m_{\ell+1}\\ \|A_J\|>k}} \zeta^{\ell+1-|J|}_{\cA_{A_{JI}}}(s)
  \end{equation}
  and apply (\ref{eq:SAAAEll}) to the first term in $\zeta$ and 
  (\ref{eq:T1.1}) to the second, getting
  $$
  \zeta^{\ell}_{\cA_{A_I}}(s) \leq \nu_{\cA_{A_I}}(s) 
  + \sum_{\substack{J\in\cJ^m\\ \|A_J\|\leq k}}\left[\nu^{\ell}_{\cA_{A_{JI}}}(s)
  +c^{-s}_\cA\mu_{\cA_{A_{JI}}}(s)\zeta^\ell_\cA(s)\right] 
  +c^{-s}_\cA\sum_{\substack{J\in\cJ^m\\ \|A_J\|>k}}\|A_{JI}\|^{-s}\zeta^\ell_\cA(s)
  $$
%
  $$
  =h^1_{\cA_{A_I},k}(s)+c^{-s}_\cA\sum_{L\in\cQ^1_{\cA,k}}\|A_L\|^{-s}\zeta^\ell_\cA(s)
  =h^1_{\cA_{A_I},k}(s)+f^1_{\cA_{A_I},k}(s)\zeta^\ell_\cA(s),
  $$
  where we set
  $h^1_{\cA_{A_I},k}(s)=h^0_{\cA_{A_I},k}(s)+\sum_{\substack{J\in\cJ^m\\ \|A_J\|\leq k}}\nu_{\cA_{A_{JI}}}(s)$.
%
 
  By repeating recursively this procedure we get that 
  $$
  \zeta^{\ell}_\cA(s) \leq h^r_{\cA,k}(s) + f^r_{\cA,k}(s) \zeta_\cA(s)
  $$
  for all $r\geq0$. Since the $\cQ^r_{\cA,k}$ 
  stabilize into the $\cQ_{\cA,k}$, we proved that
  $$
  \zeta^{\ell}_\cA(s) \leq h_{\cA,k}(s) + f_{\cA,k}(s) \zeta_\cA(s).
  $$
  for every $\ell$ and therefore (\ref{eq:ztwo}) follows;
  (\ref{eq:zone}) 
  is proved analogously.

  By Proposition~\ref{thm:bigK}, for $k$ big enough the points $s_{f,k}=f_{\cA,k}^{-1}(1)$
  and $s_{g,k}=g_{\cA,k}^{-1}(1)$ are uniquely defined and strictly between
  $0$ and $\infty$.
  An argument analog to the one at the beginning of the section
  shows that these inequalities imply that $s_{g,k}\leq s_\cA\leq s_{f,k}$ 
  for every such $k$.
  In particular $0<s_\cA<\infty$ since every $s_{f,k}$ is finite
  and every $s_{g,k}$ is larger than 0.
  
  To prove the last part of the theorem we point out that
  $g_{\cA,k}(s)=c_\cA^sn^{-s}f_{\cA,k}(s)$ and that, for any $s'>0$,
  $$
  f_{\cA,k}(s)=\frac{m-1}{c^s_\cA}\sum_{J\in\cQ_{\cA,k}}\|A_J\|^{-s}
  >
  (c_\cA k)^{s'}\frac{m-1}{c^{-s-s'}_\cA}\sum_{J\in\cQ_{\cA,k}}\|A_J\|^{-s-s'}
  = (c_\cA k)^{s'}f_{\cA,k}(s').
  $$
  Hence
  $$
  1
  =
  g_{\cA,k}(s_{g,k}) 
  = 
  c^{s_{g,k}}n^{-s_{g,k}}f_{\cA,k}(s_{g,k})
  >
  $$
  $$
  >
  c^{s_{g,k}}_\cA n^{-s_{g,k}}(c_\cA k)^{s_{f,k}-s_{g,k}}f_{\cA,k}(s_{f,k})
  =
  c^{s_{f,k}}_\cA n^{-s_{g,k}}k^{s_{f,k}-s_{g,k}}
  $$
  so that
  $$
  0>s_{f,k}(\log k+\log c_\cA)-s_{g,k}(\log k+\log n)
  $$
  and finally
  $$
  0 \leq s_{f,k}-s_{g,k} < s_{g,k}\left(\frac{\log k+\log n\phantom{_\cA}}{\log k+\log c_\cA}-1\right)
  \leq s_\cA\frac{\log n-\log c_\cA}{\log k+\log c_\cA}.
  $$
\end{proof}
\subsubsection{An alternate characterization of $s_\cA$}
The exponent $s_\cA$ of a $m$-gasket $\cA$ can also be extracted
from the asymptotics of the partial sums of the $\|A_I\|^{-s}$ over 
same-rank multi-indices, namely from the sequence of functions
$\zeta_{\cA,k}(s)=\sum_{I\in\cI^m_k}\|A_I\|^{-s}$.
\begin{lemma}
  The sequence of analytical log-convex monotonically decreasing functions 
  $\zeta^{1/k}_{\cA,k}(s)$ converges pointwise for every $s\in[0,\infty)$ to a 
  bounded continuous log-convex monotonically decreasing function $\xi_\cA(s)$.
\end{lemma}
\begin{proof}
  The $\zeta^{1/k}_{\cA,k}(s)$ are analytical log-convex monotonically 
  decreasing functions because every summand $\|A_I\|^{-s}$ of the 
  $\zeta_{\cA,k}$  satisfies those properties and so does every finite 
  or infinite (converging) sum of them and positive power. 

  In order to prove the convergence of the sequence we can replace 
  without loss of generality in the expression of the $\zeta_{\cA,k}$ 
  (by abuse of notation we will denote the new functions still by 
  $\zeta_{\cA,k}$) the max norm with any submultiplicative norm 
  $\|\cdot\|'$ and notice 
  that, since $\|A_{IJ}\|'\leq\|A_I\|'\|A_J\|'$, then
  $\zeta_{\cA,k+k'}(s)\geq\zeta_{\cA,k}(s)\zeta_{\cA,k'}(s)$.
  From this follows that, for every $s$, the sequence $\zeta^{1/k}_{\cA,k}(s)$ 
  can have only one accumulation point and this point must be equal to 
  $\sup_{k\in\bN}\zeta^{1/k}_{\cA,k}(s)$. The main point is that for every element
  $\zeta^{1/k_0}_{\cA,k_0}(s)$ almost all other elements of the sequence are
  not smaller than $\zeta^{1/k}_{\cA,k_0}(s)-\varepsilon$ for every 
  $\varepsilon>0$. Indeed if $k=Nk_0$ then immediately
  $\zeta^{1/k}_{\cA,k}(s)\geq\left[\zeta^{N}_{\cA,k_0}(s)\right]^{1/k_0}
  =\zeta^{1/k_0}_{\cA,k_0}(s)$, 
  while if $k=Nk_0+\ell$, with $1\leq\ell\leq k_0-1$, 
  then $\zeta_{\cA,k}(s)\geq\zeta^N_{\cA,k_0}(s)\zeta_{\cA,\ell}(s)$, so that
  $$
  \zeta^{1/k}_{\cA,k}(s)
  \geq
  \zeta^{\frac{N}{Nk_0+\ell}}_{\cA,k_0}(s)\zeta^{\frac{1}{Nk_0+\ell}}_{\cA,\ell}(s)
  =
  \left[\zeta^{1/k_0}_{\cA,k_0}(s)\right]^{\frac{1}{1+\ell/(Nk_0)}}\zeta^{\frac{1}{Nk_0+\ell}}_{\cA,\ell}(s)  
  $$
  Since there are only a finite number of possible values of $\ell$,
  for every $\varepsilon'>0$ we can find a $N$ big enough s.t. both 
  $\left|\left[\zeta^{1/k_0}_{\cA,k_0}(s)\right]^{\frac{1}{1+\ell/(Nk_0)}}
  -\zeta^{1/k_0}_{\cA,k_0}(s)\right|<\varepsilon'$ and 
  $\left|\zeta^{\frac{1}{Nk_0+\ell}}_{\cA,\ell}(s)-1\right|<\varepsilon'$
  hold for all $\ell$. Hence 
  $$
  \zeta^{1/k}_{\cA,k}(s)
  \geq 
  \zeta^{1/k_0}_{\cA,k_0}(s)-\varepsilon'\left(\zeta^{1/k_0}_{\cA,k_0}(s)+1-\varepsilon'\right)
  \geq\zeta^{1/k_0}_{\cA,k_0}(s)-\varepsilon
  $$
  for small enough $\varepsilon'$.

  That $\xi_\cA(s)=\sup_{k\in\bN}\zeta^{1/k}_{\cA,k}(s)$ be finite for all $s$
  is clear from the fact that all $\zeta^{1/k}_{\cA,k}(s)$ are positive
  decreasing functions bounded by $\zeta^{1/k}_{\cA,k}(0)=m$.
\end{proof}
\begin{theorem2}
  The function $\xi_\cA$ satisfies the following properties:
  \begin{enumerate}
    \item $\xi_\cA(s)>1$ for $s<s_\cA$;
    \item $\xi_\cA(s_\cA)=1$;
    \item $\xi_\cA(s)<1$ for $s>s_\cA$ if $\cA$ is a hyperbolic gasket;
    \item $\xi_\cA(s)=1$ for $s>s_\cA$ if $\cA$ is a parabolic fast gasket.
  \end{enumerate}
\end{theorem2}
\begin{proof}
  Directly from the $n$-th root test we get that $\xi_\cA(s)\geq1$ for
  $s<s_\cA$ and $\xi_\cA(s)\leq1$ for $s>s_\cA$, so that in particular
  $\xi_\cA(s_\cA)=1$. 

  Assume first that $\cA$ is hyperbolic, so that there exist
  constants $\alpha>1$ and $K>0$ s.t. $\|A_I\|\geq K\alpha^{|I|}$ 
  for every $I\in\cI^m$. Hence
  $$
  \frac{d\phantom{s}}{ds}\ln\zeta^{1/k}_k(s)=
  -\frac{1}{k}\frac{\sum_{|I|=k}\left(\|A_I\|^{-s}\ln\|A_I\|\right)}
  {\sum_{|I|=k}\|A_I\|^{-s}}\leq-\ln\alpha-\frac{\ln K}{k},
  $$
  namely for every $\varepsilon>0$ we can find a $\alpha'>1$, with
  $|\alpha-\alpha'|\leq\varepsilon$, and a $\bar k>0$ such that
  $(\ln\zeta^{1/k}_k(s))'\leq -\ln\alpha'$ for all $k\geq\bar k$.
  Since $\log_\alpha\zeta^{1/k}_k(s_\cA)=0$ this means that, for every $k\in\bN$ and 
  $s>0$, 
  $\ln\zeta^{1/k}_k(s_\cA+s)\leq -s\ln\alpha'$ and 
  $\ln\zeta^{1/k}_k(s_\cA-s)\geq s\ln\alpha'$,
  namely $\zeta^{1/k}_k(s_\cA)\geq(\alpha')^s>1$ at the left of $s_\cA$ and 
  $\zeta^{1/k}_k(s_\cA)\leq(\alpha')^{-s}<1$ at its right. Since this is true for 
  almost all $k$, the same properties hold for $\xi_\cA$.

  Assume now that $\cA$ is fast parabolic and that $s_\cA<\infty$ (e.g. in case
  that $\cA$ is fast). In this case the sequence
  $a_k=\min_{|I|=k}\{\|A_I\|\}$ grows polynomially and therefore, for $s>s_\cA$, 
  $$
  1\geq\zeta^{1/k}_k(s)\geq a_k^{-s/k}=\left(a_k^{-\frac{1}{k}}\right)^{-s}\buildrel k\to\infty \over 
  \longrightarrow 1,
  \hbox{ i.e. }\xi_\cA(s)=\lim_{k\to\infty}\zeta^{1/k}_k(s)=1.
  $$
  Proving that $\xi_\cA(s)>1$ for $s<s_\cA$ requires much more work.
  We start by observing that, analogously to (\ref{eq:SAAA}) and 
  (\ref{eq:SAM}), we have the inequality  
  \begin{equation}
    \label{eq:SkAAA}
    \zeta_{\cA_{A_I},k}(s)\geq\frac{1}{n^s\|A_I\|^s}\zeta_{\cA,k}(s)
  \end{equation}
  and we can re-write $\zeta_{\cA_{A_I},k}$ as follows:
  \begin{equation}
    \label{eq:SkAM}
    \zeta_{\cA_{{A_I},k}}(s) = \sum_{\substack{D\in\cD^m_k\\ |D|=k}}\zeta_{\cA_{A_{DI}},0}(s)
                     +\sum_{J\in\cJ^m_{k}}\zeta_{\cA_{A_{JI}},k-|J|}(s).
  \end{equation}
  %
  Applying (\ref{eq:SkAAA}) to (\ref{eq:SkAM}) we get that
  $$
  \zeta_{\cA_{{A_I},k}}(s) \geq \sum_{\substack{D\in\cD^m_k\\ |D|=k}}\zeta_{\cA_{A_{DI}},0}(s)
  +n^{-s}\sum_{J\in\cJ^m_{k}}\|A_{JI}\|^{-s}\zeta_{{\cA},k-|J|}(s).
  $$
  We set $w^{0,j}_{\cA_{A_I},\kappa}(s)=\sum_{\substack{J\in\cJ^m_{k}\\ |J|=j}}n^{-s}\|A_{JI}\|^{-s}$,
  so that 
  \begin{equation}
    \label{eq:uff}
    \zeta_{\cA_{{A_I},k}}(s) \geq \sum_{2\leq j\leq k}w^{0,j}_{\cA_{A_I},\kappa}(s)\zeta_{{\cA},k-j}(s).
  \end{equation}
  Now we split the second summation in (\ref{eq:SkAM}) as we did in (\ref{eq:split})
  and obtain
  $$
  \zeta_{\cA_{{A_I},k}}(s) = \sum_{\substack{D\in\cD^m_k\\ |D|=k}}\zeta_{\cA_{A_{DI}},0}(s)
                     +\sum_{\substack{J\in\cJ^m_{k}\\ \|A_J\|\leq\kappa}}\zeta_{\cA_{A_{JI}},k-|J|}(s)
                     +\sum_{\substack{J\in\cJ^m_{k}\\ \|A_J\|>\kappa}}\zeta_{\cA_{A_{JI}},k-|J|}(s).
  $$
  Now we apply recursively (\ref{eq:uff}) to the first term in the splitting
  and (\ref{eq:SkAAA}) to the second, obtaining
  $$
  \zeta_{\cA_{{A_I},k}}(s) 
  \geq
  \sum_{\substack{J\in\cJ^m_{k}\\ \|A_J\|\leq\kappa}}\zeta_{\cA_{A_{JI}},k-|J|}(s)
  +
  \sum_{\substack{J\in\cJ^m_{k}\\ \|A_J\|>\kappa}}\zeta_{\cA_{A_{JI}},k-|J|}(s)\geq
  $$
  $$
  \geq \sum_{2\leq j\leq k} \sum_{\substack{J\in\cJ^m_{k}\\ \|A_J\|\leq\kappa}}
  w^{0,j}_{\cA_{A_{JI}},\kappa}(s)\zeta_{{\cA},k-j}(s) 
  + n^{-s}\sum_{\substack{J\in\cJ^m_{k}\\ \|A_J\|>\kappa}}\|A_{JI}\|^{-s}\zeta_{{\cA},k-|J|}(s)=
  $$
  $$
  =\sum_{2\leq j\leq k}w^{1,j}_{\cA_{A_I},\kappa}(s)\zeta_{{\cA},k-j}(s),
  $$
  where 
  $$
  w^{1,j}_{\cA_{A_I},\kappa}(s) = 
  \sum_{\substack{J\in\cJ^m_{k}\\ |J|=j\\ \|A_J\|\leq\kappa}}w^{0,j}_{\cA_{A_{JI}},\kappa}(s)
  +
  n^{-s}\sum_{\substack{J\in\cJ^m_{k}\\ |J|=j\\ \|A_J\|>\kappa}}\|A_{JI}\|^{-s}.
  $$ 
  Since there are only finitely many terms in $\cA$ with norm not larger than
  $\kappa$, by repeating recursively this procedure, the $w^{i,j}_{\cA_{A_I},\kappa}$
  stabilize to some functions $w^{j}_{\cA_{A_I},\kappa}$ such that
  $\zeta_{\cA_{{A_I},k}}(s) 
  \geq
  \sum_{2\leq j\leq k}w^{j}_{\cA_{A_I},\kappa}(s)\zeta_{{\cA},k-j}(s)$.
  For $I=0$ we get
  $$
  \zeta_{\cA_,k}(s) 
  \geq
  \sum_{2\leq j\leq k}w^j_{\cA,\kappa}(s)\zeta_{{\cA},k-j}(s).
  $$
  Now consider the polynomials 
  $p_k(x)=x^k-\sum_{2\leq j\leq k}w^j_{\cA,\kappa}(s)x^{k-j}$.
  By Descartes' rule of signs they all have a single positive root.
  Moreover for $k$ big enough this root is larger than 1. Indeed
  by comparing the recursive algorithm that generate the $w^j_{\cA,\kappa}$ 
  with the one generating $g_{\cA,\kappa}$ it is clear that the function
  $W^k_{\cA_{A_I},\kappa}(s)=\sum_{2\leq j\leq k}w^{j}_{\cA_{A_I},\kappa}(s)$ is equal 
  to the truncation of
  the series $g_{\cA_{A_I},\kappa}$ to the matrices $A_I$ with $|I|\leq k$.
  In particular this means that when $s<s_\cA$ there is some $\bar k$ 
  large enough that $W^{\bar k}_{\cA_{A_I},\kappa}(s)>1$. Then $p_{\bar k}(1)<0$
  and therefore the positive root $\sigma$ of $p_{\bar k}$ is larger than 1.
  Hence 
  $$
  \mu\bydef
  \min_{0\leq j\leq \bar k}\{\zeta_{\cA_,j}(s)\sigma^{-j}\}=
  \inf_{0\leq j\leq\infty}\{\zeta_{\cA_,j}(s)\sigma^{-j}\},
  $$
  which follows by induction as a consequence of the following observation:
  $$
  \zeta_{\cA_,\bar k+1}(s)
  \geq
  \sum_{2\leq j\leq \bar k}w^{j}_{\cA_{A_I},\kappa}(s)\zeta_{{\cA},\bar k+1-j}(s)
  \geq
  \mu\sum_{2\leq j\leq \bar k}w^{j}_{\cA_{A_I},\kappa}(s)\sigma^{\bar k+1-j}
  \geq
  \mu\sigma^{\bar k+1}.
  $$
  Hence $\lim_{j\to\infty}\zeta^{1/j}_{\cA_,j}(s)\geq\sigma>1$ for 
  every $s\geq s_\cA$.
\end{proof}
\begin{example}
  \label{ex:cool}
  Let $M_1,\dots,M_m$ be upper triangular matrices of the form
  $$
  M_i=\begin{pmatrix}\alpha_i&\beta_i\cr 0&1\cr\end{pmatrix}
  $$
  and assume that 
  $$
  \max_{1\leq i\leq m}\{\beta_i\}\leq\frac{1}{1-\max_{1\leq i\leq m}\{1-\alpha_i\}}.
  $$
  It is easy to prove by induction that under this assumption the non-zero 
  off-diagonal term never gets larger than 1, so that $\|M_I\|=1$ 
  for every $I\in\cI^m$. 

  Now consider the gasket $\cA$ generated by $A_i=\rho_i M_i$, $\rho_i>1$. 
  By the observation above for every $I=i_1\dots i_k$ we have that
  $\|A_I\|=\rho_{i_1}\cdots\rho_{i_k}$ and therefore
  $$
  \zeta_{\cA,k} = \sum_{I\in\cI^m_k}\|A_I\|^{-s}=
  \sum_{I\in\cI^m_k}\rho_{i_1}^{-s}\cdots\rho_{i_k}^{-s}=
  \left(\rho_1^{-s}+\dots+\rho_m^{-s}\right)^k.
  $$
  Since $\cA$ is clearly a hyperbolic gasket, by Theorem~2 its
  exponent $s_\cA$ is the unique solution of the equation 
  $\rho_1^{-s}+\dots+\rho_m^{-s}=1$.
  Similar but more complicated conditions can be found for 
  upper triangular matrices in higher dimension (see Section~\ref{sec:aSG}
  for a similar case with $3\times3$ matrices).
\end{example}
\subsubsection{Norm asymptotics of fast gaskets}
We can use the previous section's results to prove Theorem 2.
\begin{lemma}
  Let $\cA:\cI^m\to M_n(K)$ be a semigroup homomorphism
  and $M\in\GL_n(K)$. Then 
  \begin{equation}
    \label{eq:N(r)}
    N_\cAM(r)>N_\cA\left(\frac{r}{n\|M\|}\right).
  \end{equation}
\end{lemma}
\begin{proof}
  Since $\|AM\|\leq n\|A\|\|M\|$, we have that 
  $\|A_I\|\leq\frac{r}{n\|M\|}\implies\|A_IM\|\leq r$,
  namely $\{A_I|\|A_I\|\leq\frac{r}{n\|M\|}\}\subset\{A_IM|\|A_IM\|\leq r\}$
\end{proof}
\begin{theorem3}
  Let $\cA$ be a hyperbolic or fast parabolic gasket. Then
  $$
  \lim_{r\to\infty}\frac{\log N_\cAM(r)}{\log r} = s_\cA.
  $$
  for every $M\in\GL_n(K)$.
\end{theorem3}
\begin{proof}
  Since $\|A\|/(n\|M^{-1}\|)\leq\|AM\|\leq n\|A\|\|M\|$ we can prove the theorem 
  without loss of generality for $M=\one_n$.

  $\displaystyle\boldsymbol{\limsup_{r\to\infty}\frac{\log N_\cA(r)}{\log r} \leq s_\cA}$. 
  Let $s>s_\cA$. Then
  $$
  \infty>\zeta_\cA(s)>\sum_{\|A_I\|\leq r}\|A_I\|^{-s}\geq\sum_{\|A_I\|\leq r}r^{-s}= N_\cA(r)r^{-s},
  $$
  so that 
  $$
  s + \frac{\log \zeta_\cA(s)}{\log r} > \frac{\log N_\cA(r)}{\log r} 
  $$
  and therefore $\limsup_{r\to\infty}\frac{\log N_\cA(r)}{\log r}\leq s$.
  Since this is true for every $s>s_\cA$ it follows at once that 
  $\limsup_{r\to\infty}\frac{\log N_\cA(r)}{\log r} \leq s_\cA$.

  $\displaystyle\boldsymbol{\liminf_{r\to\infty}\frac{\log N_\cA(r)}{\log r} \geq s_\cA}$.
  From the elementary observation that
  $$
  \left\{A_I|\|A_I\|\leq r,I\in\cI^m\right\}\supset
  \bigcup_{J\in\cI^m_k}\{A_{IJ}|\|A_{IJ}\|\leq r,I\in\cI^m\}
  $$
  and using (\ref{eq:N(r)}) we get that, for every $k\in\bN$,
  $$
  N_\cA(r)\geq\sum_{J\in\cI^m_k}N_{\cA_{A_J}}(r)
  \geq
  \sum_{J\in\cI^m_k}N_{\cA}\left(\frac{r}{n\|A_J\|}\right).
  $$

  Assume now $s>s_\cA$. Since $\cA$ is, by hypothesis, either a hyperbolic
  or a fast parabolic gasket, by the definition of gasket and Theorem 2
  we can always choose a $k_0$ such that $\|A_I\|>1/n$ for $|I|\geq k_0$ and 
  $\sum_{I\in\cI^m_{k_0}}\|A_I\|^{-s}>n^s$.

  Now set $a_m=n\min_{I\in\cI^m_{k_0}}\|A_I\|$ and $a_M=n\max_{I\in\cI^m_{k_0}}\|A_I\|$,
  let $r_0>0$ be s.t. $N_\cA(r_0)>0$ and set $r_1=a_M r_0$ and
  $r_i=a_m^{i-1}r_1$, $i\geq2$.
  Similarly to the proof of Theorem~2 we have that, by induction,
  $$
  M\bydef\min_{r\in[r_0,r_1]}N_\cA(r_0)r^{-s}=\inf_{r\in[r_0,\infty]}N_\cA(r_0)r^{-s}.
  $$
  Indeed
%
  note first of all that $\lim_{i\to\infty}r_i=\infty$ since we chose 
  $k_0$ so that $a_m>1$. 
  Assume now 
  that $N_\cA(r)\geq Mr^s$ in $[r_0,r_i]$ and let $r\in[r_i,r_{i+1}]$.
  Then for every $J\in\cI^m_{k_0}$ we have that $r/(n\|A_J\|)\in[r_0,r_i]$
  since 
  $$
  r_i=\frac{r_{i+1}}{a_m}\geq\frac{r}{n\|A_J\|}\geq\frac{r_i}{a_M}=a_m^{i-1}r_0\geq r_0.
  $$
  and therefore
  $$
  N_\cA(r)
  \geq 
  \sum_{J\in\cI^m_k}N_{\cA}\left(\frac{r}{n\|A_J\|}\right)
  \geq
  \sum_{J\in\cI^m_k}M\left[\frac{r}{n\|A_J\|}\right]^{s}
  \geq 
  Mr^s.
  $$
  Hence it follows at once that
  $$
  \frac{\log N_\cA(r)}{\log r} \geq \frac{\log M}{\log r} + s
  $$
  and therefore $\liminf_{r\to\infty}\frac{log N_\cA(r)}{\log r}\geq s$.
  Since this is true for all $s<s_\cA$ it follows that
  $\liminf_{r\to\infty}\frac{log N_\cA(r)}{\log r}\geq s_\cA$.
\end{proof}
%
\section{Hausdorff dimension of limit sets of discrete subsemigroups of real and complex
projective automorphisms.}
\label{sec:fract}
In this section we show how the exponent of a free finitely generated
semigroup $\bA\subset SL_n(K)$ (resp. $\bA\subset SL^\pm_n(K)$ if $n=2n'$ 
and $K=\bR$) is sometimes related to the Hausdorff dimension of the set of 
limit points of a generic orbit in $KP^{n-1}$ of the subsemigroup 
$\Psi(\bA)\subset PSL_n(K)$ (resp. $\Psi(\bA)\subset PSL^\pm_{2n'}(\bR)$) 
naturally induced by $\bA$ (equivalently, to the residual set of the 
IFS corresponding to $\Psi(\bA)$). 

Throughout this section we will make analytical and numerical evaluations 
of the exponent of several semigroups. The analytical bounds are obtained
via the functions $\mu_{\bA,k}$ defined in Section~\ref{sec:mu} thanks to
Theorem 1. The numerical ones are obtained by evaluating numerically the
function $N_\bA(k)$ and interpolating the curve $\log N_\bA(k)$ with respect
to $\log k$ thanks to Theorem 3. 
Since $\bA$ is a gasket and $N_\bA(k)$ is integer-valued, 
a computer program can evaluate {\em exactly} its values, the only 
constraint coming from the running time which increases exponentially
with $k$. All calculations were done with Xeon 2MHz CPUs under Linux.
\subsection{$n=2, K=\bR$}
Let $\{f_1,\dots,f_m\}$ be a free set of linear automorphisms of $\bR^2$
preserving a volume 2-form {\em modulo sign}. 
With respect to any frame $\cE=\{e_1,e_2\}$, these automorphisms are 
represented by matrices $A_i\in SL_2^\pm(\bR)$. 
We denote by $\bA$ the semigroup generated by the $A_i$ and by
$\psi_I\in PSL^\pm_2(\bR)$ the automorphism of $\bR P^1\simeq\bS^1$ naturally 
induced by $f_I$, $I\in\cI^2$. The similarity between the characterization
of the exponent $s_\cA$ given in Theorem 2 and the formula for the Hausdorff
dimension of a 1-dimensional IFS given in~\cite{Fal90} (Theorem~9.9, p.~126)
suggests the following claim:
%
\begin{theorem4}
  \label{thm:SL(2,R)}
  Assume that the $f_i$ have all real distinct eigenvalues and that there exists 
  some proper open set $V\subset\bR P^1$ invariant by the $\psi_i$ and such that, 
  for some affine chart $\varphi:\bR P^1\to\bR$, the $\psi_i$ are contractions 
  on $\varphi(\overline{V})$ with respect to the Euclidean distance and satisfy 
  the ``open set condition'' $\psi_1(V)\cap\psi_2(V)=\emptyset$. 
  Let $R_\bA=\cap_{k=1}^\infty\left(\cup_{|I|=k}\psi_{I}(V)\right)$ 
  be the corresponding residual set. Then $2\dim_HR_\bA=s_\bA$.
\end{theorem4}
\begin{proof}
  Let $(x,y)$ and $[x:y]$ the affine and homogeneous coordinates associated
  to $\varphi$ and assume, for the argument's sake, that $\varphi([x:y])=x/y$.
  Under the hypotheses every $\psi_I$ has two fixed points, an attractive one 
  $a_I$ and a repulsive one $r_I$. 
  Since we are assuming $V$ to be invariant under the $\psi_I$, then it must 
  happen that $a_I\in\overline{V}$ for all $I\in\cI^m$. 
  Let $A_I=\begin{pmatrix}\alpha&\beta\cr \mu&\nu\cr\end{pmatrix}$ 
  be the matrix representing the $f_I$ in the 
  affine coordinates relative to $\varphi$. Let 
  $\Psi_I=\varphi\circ\psi_I\circ\varphi^{-1}(x)$ be the coordinate expression 
  of $\psi_I$ in the chart $\varphi$. Then a direct calculation shows that
  $$
  \Psi_I(\varphi)=\frac{\alpha\varphi+\beta}{\mu\varphi+\nu},\;
  \Psi'_I(\varphi)=\frac{\det A_I}{(\mu\varphi+\nu)^2},\;
  \Big|\Psi'_I(\varphi(a_I))\Big|=\frac{1}{\|A_I\|^2}
  $$
  (recall that $\det A_I=\pm1$).
  Now let $\varphi(\overline{V})=[\varphi_1,\varphi_2]$ and set 
  $\varphi_m=(\varphi_1+\varphi_2)/2$. 
  Then $\|A_I\|/2\leq\mu\varphi_m+\nu\leq\|A_I\|$ and therefore
  $$
  \frac{1}{\|A_I\|^2}\leq\Big|\Psi'_I(\varphi_m)\Big|\leq\frac{4}{\|A_I\|^2}.
  $$
  Hence the limit 
  $$
  \lim_{k\to\infty}\left[\sum_{|I|=k}\Big|\Psi'_I(\varphi_m)\Big|^s\right]^\frac{1}{k}
  $$
  converges iff it converges the limit
  $$
  \lim_{k\to\infty}\left[\sum_{|I|=k}\|A_I\|^{-2s}\right]^\frac{1}{k}.
  $$
  By the result on 1-dimensional contractions quoted above,
  the exponent separating the values of $s$ for which the first limit diverge 
  from those for which it converges is exactly $\dim_H R_\bA$.
  By Theorem 2 this means exactly that $2\dim_H R_\bA=s_\bA$.
\end{proof}
An interesting consequence of the previous theorem is the following
constraint posed by geometry to the (algebraic) exponent of the semigroups 
satisfying its conditions:
\begin{corollary}
  Let $\bA\subset SL_2^\pm(\bR)$ be a semigroup satisfying the conditions
  of the theorem above. Then $s_\bA\leq2$.
\end{corollary}
\begin{proof}
  This is a direct consequence of the trivial fact that the Hausdorff
  dimension of a subset of $\bR$ cannot be bigger than 1.
\end{proof}
\subsubsection{Matrices with non-negative entries}
The semigroup $SL_2^\pm(\bR^+)$ is a source of several interesting 
semigroups that satisfy the hypothesis of Theorem 4. Indeed every
linear automorphism \hbox{$f=A^i_je_i\otimes\varepsilon^j$} preserves
the positive cone $C(\cE)$ over $\cE$ and therefore the induced
projective automorphism $\psi_f$ preserves the closed segment 
$S(\cE)\subset\bR P^1$.

In this simple setting there is an easy sufficient condition
to determine whether a gasket is fast:
\begin{proposition}
  Let $\cA:\cI^m\to SL^\pm_2(\bR^+)$ be a gasket and suppose that all
  products $A_I$, $|I|=2$, have no entry equal to zero. Then $\cA$ 
  is fast.
\end{proposition}
\begin{proof}
  Let $A_{12}=\begin{pmatrix}p&q\cr r&s\cr\end{pmatrix}$ and 
  $A_K=\begin{pmatrix}a&b\cr c&d\cr\end{pmatrix}$, so that
  $$
  A_{12K}=\begin{pmatrix}pa+qc&pb+qd\cr ra+sc&rb+bd\cr\end{pmatrix}.
  $$
  Assume, for the argument's sake, that $\|A_{12K}\|=ra+sb$.
  Then, since 
  $$
  pa+qc\geq\frac{\min\{p,q\}}{\max\{r,s\}}(ra+sb),$$
  we have that, for every $M\in M_2(\bR^+)$,
  $$
  \|MA_{12K}\|\geq\|M\|(pa+qc)\geq\frac{\min\{p,q\}}{\max\{r,s\}}\|M\|\|A_{12K}\|.
  $$
  By repeating this argument for every index of order 2 and denoting
  by $c$ the smallest of these coefficients, we have that 
  $\|MA_{12K}\|\geq c\|M\|\|A_{12K}\|$ for every $M\in M_2(\bR^+)$. 
  In particular then $\cA$ is a fast gasket with coefficient not smaller than $c$.
 \end{proof}
\begin{example}
  \label{ex:Fal}
  Let $\cE=\{e_1,e_2\}$ be a frame on $\Rt$ and $f_{1,2}$ defined by
  $$
  f_1(e_1)=e_1+e_2, f_1(e_2)=e_2;\; f_2(e_1)=2e_1+e_2, f_2(e_2)=e_2.
  $$
  With respect to $\cE$ the $f_i$ are represented by the matrices
  $$
  F_1=\begin{pmatrix}1&1\cr 1&0\cr\end{pmatrix},
  F_2=\begin{pmatrix}2&1\cr 1&0\cr\end{pmatrix}. 
  $$
  The semigroup $\bF=\langle F_1,F_2\rangle\subset SL^-(\bN)$ is free because 
  if $F_I\in\bA$, $I\in\cI^2$, then the entries in $F_I$'s 
  lower row are equal to the entries in the upper row of the matrix $F_{I'}$
  and according to whether the upper left entry of $F_I$ is larger or smaller
  than its lower left entry we get whether $I=2I'$ or $I=1I'$. Proceeding 
  recursively this way we see that there is no other index $J\neq I$ s.t.
  $F_J=F_I$. In particular then $\bF$ is a gasket. Moreover $\bF$ is 
  hyperbolic: indeed clearly $\|F_I\|\geq\|F_1^{|I|}\|$, since $F_2$ has no 
  entry smaller than the corresponding entry of $F_1$, and 
  $\|F_1^k\|\simeq g^k$, where $g$ is the golden ratio, because clearly 
  $F_1^k=\begin{pmatrix}f_{k+2}&f_{k+1}\cr f_{k+1}&f_{k}\cr\end{pmatrix}$,
  where $f_k$, $k\geq1$, is the Fibonacci sequence $0,1,1,2,3,5,\dots$ 
  Finally, $\bF$ is fast (with coefficient not smaller than $1/3$)
  by the previous proposition.

  A direct calculation shows that $\psi_{1,2}$ are not contractions over
  $S(\cE)$ but they are so over
  the smaller set $S(\cE')$, with 
  $$\cE'=\{e'_1=(1+\sqrt{3},2),e'_2=(1+\sqrt{3},1)\}.$$
  Let $[x:y]$ be homogeneous coordinates on $\bR P^1$ corresponding to $\cE'$.
  In the canonical chart $\varphi=x/y$ the maps $\psi_i$ induced by $f_i$ 
  write as
  $$
  \psi_1(\varphi) = \frac{\varphi+1}{\varphi},\;  
  \psi_2(\varphi) = \frac{2\varphi+1}{\varphi},
  $$
  which reveals that this example coincides with Example 9.8 of~\cite{Fal90}, 
  coming from the theory of continued fractions.

  To obtain analytical bounds for $s_\bF$ we can use Theorem 2.
  Since both generators have an eigenvalue larger than 1 the norms
  of the terms $F_1F_2^k$ and $F_2F_1^k$ grow exponentially, so that
  we can get a good approximation of $\mu_{\bF,\ell}$ by truncating the
  sums after just a few terms. By considering only the terms with
  $k\leq10$ in $\mu_{\bF,0}$ and solving the equation $\mu_{\bF,0}(s)=2^s$
  in this approximation we get $s_\bF\geq.51$, with a relative error of about
  $6\%$ on the more precise estimate $s_\bF\geq.54$ obtained by considering
  $k\geq20$. Since $c=1/3$, the first $\mu_{\bF,\ell}$ we can get 
  upper bounds is $\mu_{\bF,3}$. Here we just mention that from
  $\mu_{\bF,8}$, considering the first 30 summands of all series that
  appear in its expression, we get $0.95\leq s_\bF\leq 1.76$. 
  In terms of the dimension of $R_\bF$ this translates in 
  $0.474\leq\dim_H R_\bF\leq0.877$. By evaluating $N_\bF(k)$ for
  $k=2^p$, $1\leq p\leq28$, (taking about 20 minutes of CPU time) we get 
  the estimate
  $s_\bF\simeq1.062$ (see Table~\ref{tab:N} for the corresponding values 
  of $N_\bF$), with a (heuristic) error of 2 on the last digit. 
  This corresponds to the well-known fact $\dim_H R_\bF\simeq 0.531$.
\end{example}
\begin{example} 
  Consider now 
  $$
  f_1(e_1)=e_1, f_1(e_2)=e_1+e_2;\; 
  f_2(e_1)=e_1+e_2, f_2(e_2)=e_2.
  $$
  The corresponding matrices (with respect to $\cE$)
  $$
  C_1=\begin{pmatrix}1&0\cr 1&1\cr\end{pmatrix},
  C_2=\begin{pmatrix}1&1\cr 0&1\cr\end{pmatrix} 
  $$
  generate the semigroup $\boC_2\subset SL_2(\bN)$ we already met 
  in Examples~\ref{ex:dyn} and~\ref{ex:dynfast}. 
  In particular we already know that $\boC_2$ is a parabolic fast
  gasket with coefficient $c=1/2$. 
  A direct check shows that the slowest and fastest growths 
  with respect to the order $k$ of the multi-index $I$ of $C_I\in\cA$ 
  correspond respectively to the pure powers $C_i^k$, for which $\|C_i^k\|=k$,
  and to the ``cyclic'' products $C_i C_{i+1} C_{i+2}\cdots C_{i+k-1}$, for
  which $\|C_i C_{i+1} C_{i+2}\cdots C_{i+k-1}\|\simeq g^k$, where the sums
  in the indices are meant ``modulo 2'' in the sense that 3 means 1, 4 means 2 
  and so on. The reason why the golden ration $g$ appears is that, similarly 
  to the previous case,
  $$
  C_1 C_2 C_1\cdots C_{i+k-1} = 
  \begin{pmatrix}f_{k+2}&f_{k+1}\cr f_{k+1}&f_{k}\cr\end{pmatrix}
  $$
  for $k$ odd while if $k$ is even the two rows get exchanged
  and analogously for the cyclic products beginning by $C_2$.

  In the affine chart $\varphi:[x:y]\to x/(x+y)$ the maps $\psi_i$
  induced by the $f_i$ write as 
  $$
  \psi_1(\varphi) = \frac{\varphi}{1+\varphi},\;
  \psi_2(\varphi) = \frac{1}{2-\varphi}
  $$
  and the segment $S(\cE)$ maps into $[0,1]$. Note that this choice
  of chart corresponds to writing $e_1=e'_1$ and $e_2=e'_1+e'_2$,
  expressing the $f_i$ with respect to $\cE'=\{e'_1,e'_2\}$ and
  using the canonical chart $y'=1$ for the corresponding homogeneous 
  coordinates $[x':y']$. In terms of the semigroup, this corresponds
  to the adjunction via the matrix 
  $M=\begin{pmatrix}1&1\cr 0&1\cr\end{pmatrix}$.
  A direct calculation
  shows that the $|\psi_i'(\varphi)|\leq1$ on $[0,1]$, with the equal sign 
  holding at $0$ for $\psi_1$ and at 1 for $\psi_2$, namely the IFS
  $\{\psi_1,\psi_2\}$ is parabolic. 

  Evaluating the Hausdorff dimension of the limit set $R_{\cC_2}$ of 
  a point $w\in(0,1)$ under the action of $\cC_2$ is nevertheless 
  an easy task. 
  Indeed, since $\psi_1((0,1))=(0,1/2)$ and $\psi_2((0,1))=(1/2,1)$,
  the $\psi_I$, $|I|=k$, subdivide $(0,1)$ into $2^k$ disjoint segments 
  $d_I=\psi_I(0,1)$ in such a way that $\cup_{|I|=k}\overline{d_I}=[0,1]$.
  Moreover the length of these segments goes to zero for $k\to\infty$.
  Indeed if $C_I=\begin{pmatrix}a&b\cr c&d\cr\end{pmatrix}\in\cC_2$ 
  then $\psi_I(\varphi)=\frac{(a-c)\varphi+c}{(a+b-c-d)\varphi+c+d}$
  and therefore
  $$
  |d_I| = \Big|\frac{c}{c+d}-\frac{a}{a+b}\Big|=\frac{1}{(a+b)(c+d)},
  $$
  from which we get 
  $$
  \frac{1}{4\|C_I\|^2}\leq|d_I|\leq\frac{1}{\|C_I\|}.
  $$
  Hence the orbit under the $\psi_I$ of every $w\in(0,1)$ is dense in
  $(0,1)$ and therefore $\dim_H R_{\boC_2}=1$. Note that, since this IFS is 
  not hyperbolic, Theorem 4 does not apply to it.

  Let us get analytical bounds for $s_\cA$ via $\mu_{\boC_2,\ell}$.
  Unlike the previous example, the presence of parabolic elements 
  in the semigroup does not allow to truncate the series to just 
  a few terms because of its very slow convergence.
  Since $\|C_iC_{i+1}^k\|=k+1$ we get easily that
  $$
  \mu_{\boC_2,0}(s)=\sum_{J\in\cJ^2}\|C_J\|^{-s}=2\sum_{k=1}^\infty\|C_1C_2^k\|^{-s}
  =2\sum_{k=2}^\infty k^{-s}=2(\zeta(s)-1),
  $$
  where $\zeta(s)$ is the Riemann's zeta function.
  The solution of $\mu_{\boC_2,0}(s)=2^s$ gives us the bound $s_{\boC_2}\geq1.54$.
  The first upper bound can be gotten from $\mu_{\boC_2,2}$,
  obtained by replacing the two terms of norm 2 in 
  $\mu_{\boC_2,0}(s)$ , namely $\|C_{12}\|^{-s}$ and $\|C_{21}\|^{-s}$, 
  with, respectively, $\mu_{\boC_2{A_{12}}}(s)$ and $\mu_{\boC_2{A_{21}}}(s)$.
  A direct calculation shows that 
  $$
  \mu_{\boC_2{A_{12}}}(s)=\sum_{k=2}^\infty(2k+1)^{-s}+\sum_{k=4}^\infty k^{-s}
  =2^{-s}\zeta(s,\frac{5}{2})+\zeta(s)-1-2^{-s}-3^{-s},
  $$
  where $\zeta(s,t)$ is the Hurwitz zeta function, and by symmetry
  we know that $\mu_{\boC_2{A_{12}}}=\mu_{\boC_2{A_{21}}}$.
  Hence
  $$
  \mu_{\boC_2,2}(s)=2\left(2\zeta(s)+2^{-s}\zeta(s,\frac{5}{2})-2-2^{1-s}-3^{-s}\right)
  $$
  which gives the bounds $1.7\leq s_{\boC_2}\leq3.93$ as solutions of
  $\mu_{\boC_2,2}(s)=2^s$ and $\mu_{\boC_2,2}(s)=2^{-s}$.
  A numerical evaluation of $N_{\boC_2}(k)$ for $k=2^p$, $1\leq p\leq17$,
  (taking about 1 hour of CPU time, see Table~\ref{tab:N} for the
  evaluated values) gives $s_{\boC_2}=2.0001$ with a 
  (heuristic) error of 1 on the last digit.
  This and the evaluation of $\dim_H R_{\boC_2}$ above strongly suggest
  that $s_{\boC_2}=2$.  
\end{example}
It is interesting to consider the following generalization of the 
previous example, namely the free semigroups $\boC_{2,\alpha}\subset SL_2(\bN)$
generated by
$$
C_{1,\alpha}=\begin{pmatrix}\alpha&0\cr 1/\alpha&1/\alpha\cr\end{pmatrix},
C_{2,\alpha}=\begin{pmatrix}1/\alpha&1/\alpha\cr 0&\alpha\cr\end{pmatrix}. 
$$
In this case, in the same framework used above, 
$$
\psi_{1,\alpha}(\varphi) = \frac{\varphi}{\alpha^2+(2-\alpha^2)\varphi},\;
\psi_{2,\alpha}(\varphi) = \frac{1+(\alpha^2-1)\varphi}{2+(\alpha^2-2)\varphi}.
$$
A direct check shows that, for every fixed $\alpha\in(1,2)$, the $\psi_{i,\alpha}$ 
are contractions on the invariant interval $[0,1]$ and that they satisfy 
the open set condition with respect to it. Let $R_{\boC_{2,\alpha}}$ be the limit set
of the orbit of any point $w\in(0,1)$ under the action induced by $\boC_{2,\alpha}$.
The very same argument used in the example above shows that 
$\dim_H R_{\boC_{2,\alpha}}=1$. As a corollary of Theorem 4 we get the following:
\begin{proposition}
$s_{\boC_{2,\alpha}}=2$ for every $\alpha\in(1,2)$.
\end{proposition}
\begin{remark}
  The restriction on the possible values of $\alpha$ looks more like
  an artificial effect of a poor choice for the distance function 
  rather than a true property of the semigroups. 
  We believe that by choosing a ad-hoc metric and maybe
  slightly modifying the argument the proposition above can be extended
  to the half-line $[1,\infty)$.
\end{remark}
%
%
\subsection{$n=2, K=\bC$}
\label{sec:cSG}
Now we consider the case of $2\times2$ {\em complex} matrices.
The corresponding projective space is the {\em Riemann sphere}
$\bC P^1$, namely the complex plane with the addition of a point 
at infinity.

The geometry of Kleinian groups, namely of discrete subgroups of
the M\"obius group 
$PSL_2(\bC)$, is known to be extremely rich and is presumably even 
richer in case of Kleinian sub{\em semi\hskip1pt}groups. A study of such 
semigroups in a general setting is therefore way beyond the scope 
of the present paper. Here we rather state first a somehow general 
property analogous to the real case above and then, as a source
of examples, focus our attention on a particular but interesting 
case that we refer to as {\em complex projective Sierpinski gaskets}.
\begin{theorem5}
  \label{thm:SL(2,C)}
  Let $f_i$ be $m$ linear automorphisms of $C^2$ represented in coordinates
  by the matrices $A_i\in SL_2(\bC)$ and let $\psi_i$ be the induced
  elements in $PSL_2(\bC)$. Assume that there exists 
  some proper open set $V\subset\bC P^1$ invariant by the $\psi_i$ and such that, 
  for some affine chart $\varphi:\bC P^1\to\bC$, the $\psi_i$ are contractions 
  on $\varphi(\overline{V})$ with respect to the Euclidean distance and satisfy 
  the ``open set condition'' $\psi_1(V)\cap\psi_2(V)=\emptyset$. 
  Let $R_\bA=\cap_{k=1}^\infty\left(\cup_{|I|=k}\psi_{I}(V)\right)$ 
  be the corresponding residual set. Then $2\dim_HR_\bA=s_\bA$.
\end{theorem5}
The proof of this theorem is almost verbatim the same of Theorem 4
and, correspondingly, we have the following corollary:
\begin{corollary}
 Under the hypotheses of the previous theorem, $2\leq s_\bA\leq4$. 
\end{corollary}
%
\begin{definition}
  Let $f_1,f_2,f_3$ be volume-preserving linear automorphisms of $\bC^2$ 
  with real spectrum and denote by $A_1,A_2,A_3\in\SL_2(\bC)$ the corresponding 
  matrices with respect to some coordinate system and by 
  $\psi_1,\psi_2,\psi_3\in PSL_2(\bC)$ their corresponding projective 
  automorphisms acting on the Riemann sphere. Let $[e_i]\in\CPo$ be
  a fixed point for $\psi_i$ corresponding to the largest eigenvalue 
  $\lambda\geq1$ of $f_i$.
  We say that the semigroup $\boldsymbol{F}$ generated by the $f_i$ 
  (or, equivalently, the semigroup $\bA$ generated by the $A_i$) is a 
  {\em complex projective Sierpinski gasket} if 
  the following conditions are satisfied:
  \begin{enumerate}
  \item $[f_i(e_j)]=[f_j(e_i)]$ for every $i,j$ with $i\neq j$;
  \item the circle $\Gamma_k$ passing through $[e_i]$, $[e_j]$ 
    (where $i,j,k$ is a permutation of $1,2,3$) and $[f_i(e_j)]$ 
    is invariant under both $f_i$ and $f_j$, $i\neq j$;
  \item $[f_k(e_i)]$ and $[f_k(e_j)]$ belong to the same connected
    component of $\CPo\setminus\Gamma_k$ for every permutation $(i,j,k)$
    of $(1,2,3)$.
  \item the circles $\Gamma_k$ do not intersect in the interior 
    of the curvilinear triangle $T_\bA\subset\bC P^1$ having 
    as vertices the $[e_i]$ and as sides the segments of the $\Gamma_k$ 
    with vertices the points $[e_i]$ and $[e_j]$ containing 
    the point $[f_i(e_j)]$.
  \end{enumerate}
  %
\end{definition}
By construction every such gasket $\bA$ is free and satisfies the 
open set condition with respect to the interior of $T_\bA$.
Since the M\"obius group $PSL_2(\bC)$ is transitive on triples of distinct
points, we assume without loss of generality in the rest of this section
that $T_\bA$ has vertices $[e_1]=[1:1]$, $[e_2]=[i:1]$, $[e_3]=[-1:1]$ 
with respect to homogeneous coordinates $[z:w]$ and use the affine
chart $w=1$ with complex coordinate $z=x+iy$ for all calculations.
%
\begin{proposition}
  Let $f_1,f_2,f_3$ be volume-preserving linear automorphisms with 
  real spectrum having respectively
  $e_1=(1,1)$, $e_2=(i,1)$, $e_3=(-1,1)$ as eigenvectors corresponding
  to the largest eigenvalue and assume that
  $$
  \begin{cases}
    \psi_1([e_3])=\psi_3([e_1])=u+iv,\\
    \psi_2([e_1])=\psi_1([e_2])=is,\\
    \psi_3([e_2])=\psi_2([e_3])=-u+iv.\\
  \end{cases}
  $$
  A necessary condition for $f_1,f_2,f_3$ to generate 
  a Sierpinski gasket
  {\em symmetric with respect to the imaginary axes}, namely
  such that $f_1(z)=\overline{f_2(-\overline{z})}$ and 
  $f_3(z)=\overline{f_3(-\overline{z})}$, is that 
  $\psi_1([e_3])\in\Gamma$, where $\Gamma$ is the circle
  \begin{equation}
    \label{eq:symm}
    x^2+y^2-x(1-s^2)-s^2=0.
  \end{equation}
  For $s=0$ the condition is sufficient for $u\in[1/5,\alpha]$, 
  where $\alpha\simeq0.651$.
\end{proposition}
\begin{proof}
  A long but straightforward direct calculation shows that condition 
  (\ref{eq:symm}) is the only
  one coming from imposing that each one of the tetruples $[e_1]$, $[e_3]$, 
  $\psi_1([e_3])$, $\psi_{11}([e_3])$ and $[e_1]$, $[e_2]$, $\psi_1([e_2])$, 
  $\psi_{11}([e_2])$ identifies a single circumference.
  No further condition comes from $\psi_3$ and by symmetry we obtain an 
  equivalent condition with respect to $\psi_2$.
  
  When $s=0$ another direct calculation shows that if $u<1/5$ then $e_1$
  is not anymore the eigenvector of $f_1$ corresponding to its
  largest eigenvalue. When $u=\alpha$ the circles $\Gamma_{13}$ and 
  $\Gamma_{12}$ are tangent to each other and for $u>\alpha$ they intersect inside 
  $T_\bA$.
\end{proof}
%
%
\begin{example}
  Let us give a short survey of the kind of geometry we meet in case
  of complex projective Sierpinski gaskets symmetric with respect to 
  the imaginary axes. For $u=16/25\simeq\alpha$ we get the gasket
  $$
  A_1=\frac{1}{\sqrt{544}}\begin{pmatrix}20&12i\cr -3i&29\cr\end{pmatrix},
  A_2=\frac{1}{\sqrt{24}}\begin{pmatrix}4&4\cr 1&7\cr\end{pmatrix},
  A_3=\frac{1}{\sqrt{24}}\begin{pmatrix}4&-4\cr -1&7\cr\end{pmatrix}.
  $$
  \begin{figure}
    \centering
    \begin{tabular}{cc}
      \includegraphics[width=6.4cm]{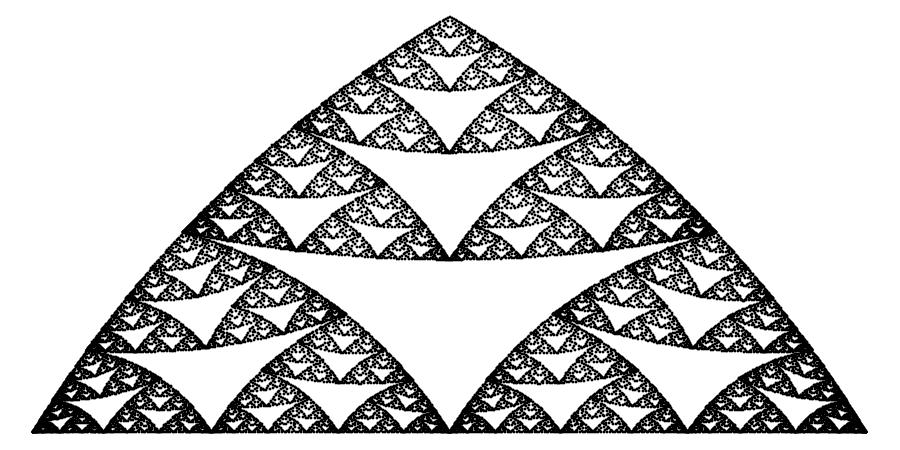}&\includegraphics[width=6.4cm]{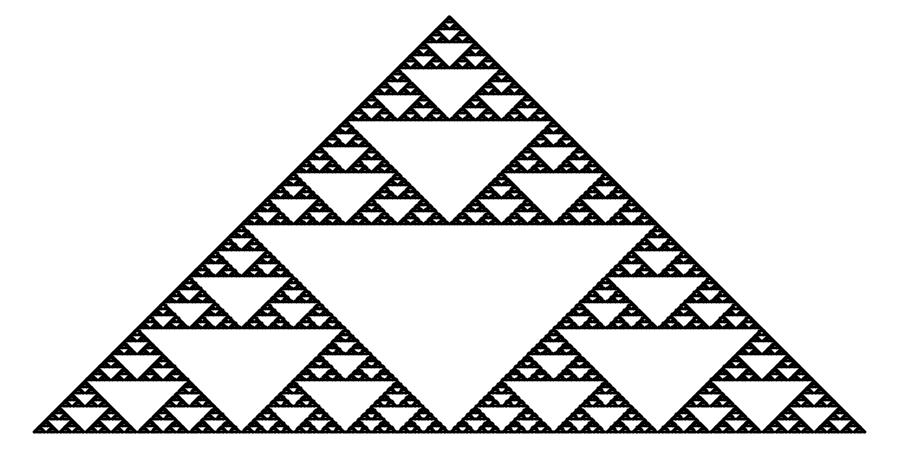}\cr
      a&b\cr
      \noalign{\medskip}
      \includegraphics[width=6.4cm]{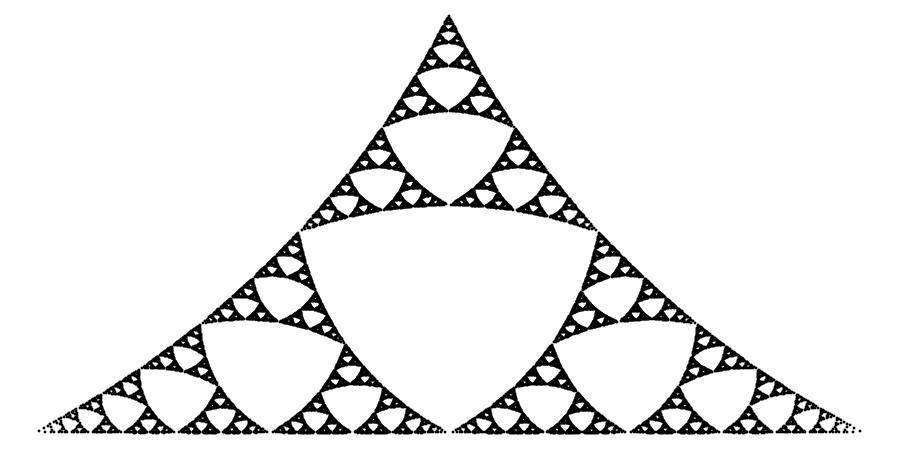}&\includegraphics[width=6.4cm]{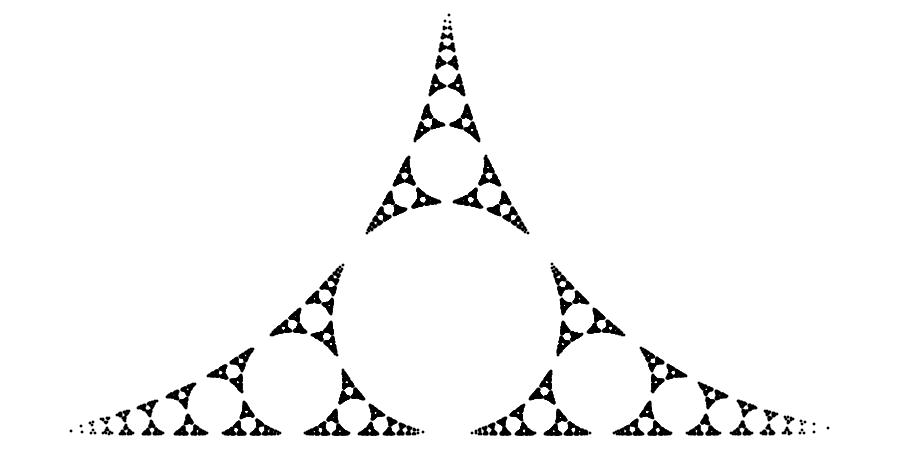}\cr
      c&d\cr
    \end{tabular}
    \caption{%
      \footnotesize
      Limit sets of complex self-projective Sierpinski gaskets: 
      (a) $u=16/25$, $s_\bA\simeq2.88$, $\dim R_\bA\simeq1.44$;
      (b) $u=1/2$, $s_\bA=2\log_23$, $\dim R_\bA=\log_23$;
      (c) $u=9/25$, $s_\bA\simeq2.88$, $\dim R_\bA\simeq1.44$;
      (d) $u=1/5$, $s_\bA\simeq2.61$, $\dim R_\bA\simeq1.305$.
      In figure for each case we show the 19683 points of the
      orbit of a random point under the action of all matrices 
      $A_I$ of the gasket with $|I|=9$.
    }
    \label{fig:cool}
  \end{figure}
  In Fig.~\ref{fig:cool} we show the orbit of a point under the action
  of the semigroup $\bA$ generated by the $A_i$. The triangle $T_\bA$
  is convex and, correspondingly, the triangle 
  $Z_\bA=T_\bA\setminus(\cup_{i=1}^3T_{\bA A_i})$ is concave.
  Each angle is almost zero because the sides of the triangle are almost
  tangent to each other, which corresponds to the fact that the limit
  value $\alpha$ is close to $16/25$.
  The restriction to $T_\bA$ of corresponding maps $\psi_i$ are contractive 
  so Theorem 5 applies. A rough numerical evaluation of the exponent
  of $\bA$ gives $s_\bA\simeq2.88$, so that $\dim R_\bA\simeq1.44$.

  By increasing $u$ the curvature of the sides increases (we consider negative
  the curvature of concave sides) until it gets zero for $u=1/2$. The 
  semigroup is now generated by
  $$
  A_1=\frac{1}{\sqrt{2}}\begin{pmatrix}1&i\cr 0&2\cr\end{pmatrix},
  A_2=\frac{1}{\sqrt{2}}\begin{pmatrix}1&1\cr 0&2\cr\end{pmatrix},
  A_3=\frac{1}{\sqrt{2}}\begin{pmatrix}1&-1\cr 0&2\cr\end{pmatrix}.
  $$
  In this case
  all sides are segments of straight lines and the gasket is diffeomorphic
  to the standard Sierpinski gasket in $\bR^2$. It is easy to prove that
  $s_\bA=2\log_23$ and, correspondingly, we get the well-known
  result $\dim R_\bA=\log_23$.

  By increasing $u$ further the curvature of the sides keeps 
  increasing and therefore $T_\cA$ becomes convex. For $u=9/25$
  the semigroup is generated by
  $$
  A_1=\frac{1}{45}\begin{pmatrix}3&6i\cr 2i&11\cr\end{pmatrix},
  A_2=\frac{1}{\sqrt{24}}\begin{pmatrix}3&3\cr -1&7\cr\end{pmatrix},
  A_3=\frac{1}{\sqrt{24}}\begin{pmatrix}3&-3\cr 1&7\cr\end{pmatrix}.
  $$
  The corresponding $\psi_i$ are contractive over $T_\cA$ so that
  Theorem 5 still applies. A rough numerical evaluation of the exponent
  gives $s_\bA\simeq2.88$ so that $\dim R_\bC\simeq1.44$.

  At the extremal value $u=1/5$ every angle of the triangle 
  is equal to $\pi$, namely every triangle $Z_\bA$ is actually
  a circle. Indeed this Sierpinski gasket is actually
  the Apollonian gasket $\bA_3$, introduced in the Motivational
  Example 2 and generated by
  $$
  A_1=\begin{pmatrix}0&i\cr i&2\cr\end{pmatrix},
  A_2=\frac{1}{2}\begin{pmatrix}\phantom{-}1&1\cr -1&3\cr\end{pmatrix},
  A_3=\frac{1}{2}\begin{pmatrix}1&-1\cr 1&\phantom{-}3\cr\end{pmatrix}.
  $$
  This time the corresponding $\psi_i$ are only non-expansive,
  which corresponds to the fact that $\bA_3$ are parabolic.

  Finally we point out that all these gaskets are fast. Here we
  outline the argument in case of the Apollonian gasket $\bA_3$ but 
  the same argument holds for all complex projective gaskets symmetric 
  with respect to the imaginary axes. 
  Note first of all that it is straightforward proving by induction that 
  $\|A_I\|=|(A_I)_{22}|$ for every matrix $A_I\in\bA_3$. 
  Now consider the case
  $$
  A_I=\begin{pmatrix}\alpha&\beta\cr \gamma&\delta\cr\end{pmatrix}, 
  A_{23}=\frac{1}{2}\begin{pmatrix}1&1\cr 1&5\cr\end{pmatrix}, 
  A_L=\begin{pmatrix}a&b\cr c&d\cr\end{pmatrix}.
  $$
  Then 
  $$
  A_{23L}=\frac{1}{2}\begin{pmatrix}a+c&b+d\cr a+5c&b+5d\cr\end{pmatrix},
  $$
  so that 
  $\|A_{I23L}\|=\frac{1}{2}|\gamma(b+d)+\delta(b+5d)|
  \geq2|\delta||d|\geq\frac{1}{3}\|A_I\|\|A_{23L}\|$.
  The case of $A_{32}$ is completely analogous to this. The remaining four
  combinations are instead analogous to the case of 
  $$
  A_{12}=\frac{1}{2}\begin{pmatrix}-i&3i\cr -2+i&6+i\cr\end{pmatrix}.
  $$
  This time
  $$
  A_{12L}=\frac{1}{2}
  \begin{pmatrix}-ia+3ic&-ib+3id\cr (-2+i)a+(6+i)c&(-2+i)b+(6+i)d\cr\end{pmatrix}
  $$
  and
  $\|A_{I12L}\|=\frac{1}{2}|\gamma(3id-ib)+\delta((i-2)b+(6+i)d)|
  \geq2|\delta||d|\geq\frac{1}{5}\|A_I\|\|A_{12L}\|$.
  Hence $\bA_3$ is a fast gasket with coefficient not smaller than $1/5$.

%
\end{example}
\subsection{$n\geq3, K=\bR$}
\label{sec:rSG}
%
%
%
%
In the real case, projective maps induced by at least $3\times3$ matrices
are not conformal
and we could find any simple way to prove analogues of Theorems 4
and 5. The non-triviality of the matter is granted by the 
well-known non-triviality of the theory of real self-affine 
sets (e.g. see~\cite{Fal88,FL98,ABVW10,FM11}). Indeed,
since $PSL^\pm_n(\bR)$ contains a subgroup
homeomorphic to the $(n-1)$-dimensional affine group, self-projective
sets are at least as non-trivial as the self-affine ones
(see Section~\ref{sec:aSG} for more details).

Because of this and in order to provide motivation for the interest
of real self-projective sets we restrict our attention to the
following particular case:
\begin{definition}
  Let $\bF=\langle f_1,\dots,f_n\rangle$ be a free semigroup of 
  volume-preserving linear automorphisms of $\bR^n$ and 
  $\psi_1,\dots,\psi_n\in PSL_n(\bR)$ the induced projective automorphisms 
  of $\bR P^{n-1}$. 
  Let $\cE=\{e_1,\dots,e_n\}$ be a $n$-frame of $\bR^n$ and
  $\cE^*=\{\varepsilon_1,\dots,\varepsilon_n\}$ its dual frame.
  We say that $\bF$
  is a {\em real projective Sierpinski gasket} over $\cE$ if
  the following conditions are satisfied:
  \begin{enumerate}
  \item $f_i=A_{ij}^k e_k\otimes\varepsilon^j$ with $A_{ij}^k\geq0$;
    \item $f(e_i)=\lambda_i e_i$, with $\lambda_i=\max_{1\leq j\leq n}\{A_{ij}^j\}$;
    \item $f_i(e_j)=\alpha e_i+\beta e_j$ with $\alpha,\beta>0$;
    \item $\psi_i([e_j])=\psi_j([e_i])$, $i\neq j$.
  \end{enumerate}
  We say that $\cE$ is a {\em proper frame} for $\bF$.
  More generally given $m<n$ of the $f_i$ we say that they are
  a Sierpinski gasket if there exist automorphisms $f_{m+1},\dots,f_n$ such
  that $\langle f_1,\dots,f_n\rangle$ is a Sierpinski gasket.
\end{definition}
Note that conditions 1--3 above imply that $\sp\{e_i\}$ is the only 
eigenspace of $f_i$ corresponding to its largest eigenvector, so that
every proper frame for $\bF$ identifies the same $n$ points $[e_i]$ on $\RPnm$.

Denote by $C(\cE)$ the {\em positive cone} over $\c$, namely the convex hull
of the set $\cup_{i=1}^n\{\lambda e_i,\lambda>0\}$. 
Then its projection on $\RPnm$ is the same for every proper frame
of $\bF$ and we denote it by $T_\bF$. This set is a $(n-1)$-simplex 
with the $n$ points $[e_i]$ as vertices. By points 2 and 3 of the
definition above, $[e_i]$
is a fixed point for $\psi_i$ and each set $T_{\bF f_i}\bydef\psi_i(T_\bF)$ is a 
$(n-1)$-simplex having in common with every other $T_{\bF f_j}$, $i\neq j$,
the vertex $\psi_i([e_j])$. Like in case of the $(n-1)$-dimensional 
standard Sierpinski gasket in $\bR^{n-1}$, the difference between
$T_\bF$ and $\cup_{i=1}^n T_{\bF f_i}$ is the interior of a convex 
polyhedron with $n(n-1)/2$ vertices that we denote by $Z_\bF$.

By repeating this procedure recursively we see that, at every step $k>0$,
$$
T_{k,\bF} \bydef \bigcup_{|I|=k}T_{\bF f_I} = T_\bF\setminus\left[\bigcup_{|I|<k}Z_{\bF f_I}\right]
$$
It is standard to call $R_\bF=\cap_{k\geq0}T_{k,\bF}$ the {\em residual set} of $\bF$.

For sake of simplicity and conciseness we limit our discussion
to the following subclass of Sierpinski gaskets:
\begin{definition}
  We say that a Sierpinski gasket $\boldsymbol{F}=\langle f_1,\dots,f_m\rangle$ 
  is {\em simple} when each $f_i$ either has only one eigenvalue 
  (first kind) or has exactly two eigenvalues and the eigenspace 
  corresponding to the larger one is 1-dimensional (second kind).
\end{definition}
\begin{example}
  \label{ex:sg}
  The most important 1-parameter family of real projective
  Sierpinski gaskets we discuss in this paper is
  $F_n^\alpha=\{f_1^\alpha,\dots,f^\alpha_n\}$,
  $\alpha\geq1$, 
  $$
  f_i^\alpha(e_j)=\alpha^{-\frac{1}{3}}
  \begin{cases}
    \alpha e_i,&i=j,\cr
    e_i+e_j,\!\!\!\!&i\neq j.\cr
  \end{cases}
  $$

  For $n=3$ the $f_i^\alpha$ are represented, with respect to any proper
  frame, by the matrices
  $$
  A_1^{\alpha,3}=\alpha^{-\frac{1}{3}}\begin{pmatrix}\alpha&1&1\cr 0&1&0\cr 0&0&1\cr\end{pmatrix},
  A_2^{\alpha,3}=\alpha^{-\frac{1}{3}}\begin{pmatrix}1&0&0\cr 1&\alpha&1\cr 0&0&1\cr\end{pmatrix},
  A_3^{\alpha,3}=\alpha^{-\frac{1}{3}}\begin{pmatrix}1&0&0\cr 0&1&0\cr 1&1&\alpha\cr\end{pmatrix}.
  $$
  As already shown in the introduction, for $\alpha=1$ we get the cubic gasket 
  $\boC_3$ and for $\alpha=2$ the (real projective generalization of the) 
  standard Sierpinski gasket $\boS_3$.
\end{example}
Consider now the dual semigroup $\boldsymbol{F}^*\subset Aut((\bR^n)^*)$
of a Sierpinski gasket.
\begin{proposition}
  Let $\bF$ be a simple Sierpinski gasket over a $n$-frame 
  $\cE=\{e_1,\dots,e_n\}$ of $\bR^n$ generated by maps 
  $$
  f_i(e_i)=\alpha_i e_i,\;f_i(e_j)=\beta_{ij}e_i+\gamma_i e_i.
  $$ 
  Then $\bF^*$ is a simple Sierpinski gasket over the frame
  $\cH=\{\eta^1,\dots,\eta^n\}$ of $(\bR^n)^*$ defined by
  \begin{equation}
    \label{eq:H}
    \eta^i=\sum_{j\neq i}\beta_{ij}\varepsilon^j+(\alpha_i-\gamma_i)\varepsilon^i.
  \end{equation}
\end{proposition}
\begin{proof}
  %
  %
  By direct calculation we see that
  $$
  f_i^*(\eta^i)=
  \sum_{j\neq i}\beta_{ij}f_i^*(\varepsilon^j)
  +
  (\alpha_i-\gamma_i)f_i^*(\varepsilon^i)=
  $$
  $$
  =
  \sum_{j\neq i}\beta_{ij}\varepsilon^j
  +
  (\alpha_i-\gamma_i)(\alpha_{i}\varepsilon^i+\sum_{j\neq i}\beta_{ij}\varepsilon^j)=
  $$
  $$
  =
  \alpha_{i}\left(\sum_{j\neq i}\beta_{ij}\varepsilon^j
  +
  (\alpha_i-\gamma_i)\varepsilon^i\right)=\alpha_i\eta^i
  $$
  and 
  $$
  f_i^*(\eta^k)=
  \sum_{j\neq k}\beta_{kj}f_i^*(\varepsilon^j)
  +
  (\alpha_k-\gamma_k)f_i^*(\varepsilon^k)=
  $$
  $$
  =\sum_{j\neq k,i}\beta_{kj}\varepsilon^j
  +\beta_{ki}\left(
  \sum_{j\neq i}\beta_{ij}\varepsilon^j
  +
  \alpha_i\varepsilon^i\right)
  +
  (\alpha_k-\gamma_k)\varepsilon^k = \beta_{ki}\eta^i+\gamma_i\eta^k.
  $$
\end{proof}
%
%
\begin{lemma}
  \label{thm:induction}
  Let $\boldsymbol{F}=\langle f_i\rangle\subset Aut(V^n)$
  be a simple Sierpinski gasket over $\cE=\{e_i\}$.
  Then the following inequalities hold:
  \begin{gather}
    \label{eq:ineqAsym}
    \|f_I\|_{\ell^1}\leq C
    \min_{\substack{1\leq k,k'\leq n\\k\neq k'}}\{\|f_I(e_k)\|_{\ell^1}+\|f_I(e_{k'})\|_{\ell^1}\},&
    \hbox{ if $\boldsymbol{F}$ is of the first kind.}\cr
    \|f_I\|_{\ell^1}\leq C\min_{1\leq k\leq n}\|f_I(e_k)\|_{\ell^1},&
    \hbox{ if $\boldsymbol{F}$ is of the second kind,}\cr
  \end{gather}
  for some $C>0$, where $\|\omega\|_{\ell^1}=\sum_{1\leq j\leq n}|\omega_j|$.
\end{lemma}
\begin{proof}
  Let $\cH=\{\eta^i=\sum\hat\beta^i_j\varepsilon^j\}$ be the proper frame 
  for $\boldsymbol{F}^*$ introduced in (\ref{eq:H}),
  where $\hat\beta^i_j=\beta_{ij}>0$, $j\neq i$, and 
  $\hat\beta^i_i=\alpha_i-\gamma_i\geq0$. 
  Clearly $\omega=\sum_{1\leq i\leq n}\eta^i\in C(\cH)$. By the previous
  proposition, $\omega_I=f^*_I(\omega)\in C(\cH)$ for all $I\in\cI^n$.
  This means that 
  $\omega_I=\sum_{1\leq i\leq n}(\omega_I)_i\varepsilon^i=\sum_{1\leq i\leq n}\lambda_i\eta^i$
  (with $\lambda_i\geq0$ and $\sum_{1\leq i\leq n}\lambda_i>0$) so that 
  $(\omega_I)_i=\sum_{1\leq j\leq n}\lambda_j\hat\beta^j_i$ and therefore
  $$
  \|\omega_I\|_{\ell^1}=\sum_{1\leq i\leq n}(\omega_I)_i
  =\sum_{1\leq i,j\leq n}\lambda_j\hat\beta^j_i\leq
  n\max_{1\leq i,j\leq n}\{\hat\beta^i_j\}\sum_{1\leq i\leq n}\lambda_i.
  $$
  
  Now note that 
  $(\omega_I)_k\geq(\min_{\hat\beta^i_k>0}\hat\beta^i_k)\sum_{\hat\beta^i_k>0}\lambda_i$
  is always a non-empty condition.

  If $\boldsymbol{F}$ is of the second kind then $\hat\beta^i_j>0$ for all $i,j$,
  so that 
  $$
  \|\omega_I\|\leq n\frac{\displaystyle\max_{1\leq i,j\leq n}\{\hat\beta^i_j\}}
            {\displaystyle\min_{1\leq i\leq n}\{\hat\beta^i_k\}}(\omega_I)_k
  $$
  for all $1\leq k\leq n$.

  If $\boldsymbol{F}$ is of the first kind then $\hat\beta^i_k=0$ iff $i=k$,
  namely every $\omega_k$ misses $\lambda_k$ in its expression, which we can
  recover by adding any other $\omega_{k'}$, $k'\neq k$. Hence in
  this case
  $$
  \|\omega_I\|\leq(n-1)\frac{\displaystyle\max_{1\leq i,j\leq n}\{\hat\beta^i_j\}}
  {\displaystyle\min_{\substack{1\leq i,j\leq n\\i\neq j}}\{\hat\beta^i_j\}}
  ((\omega_I)_k+(\omega_I)_{k'})
  $$
  for all $1\leq k,k'\leq n$, $k\neq k'$.

  Finally note that $f_I=\sum_{1\leq i,j\leq n}A^i_{Ij}e_i\otimes\varepsilon^j$,
  so that 
  $$
  \omega_I=f^*_I(\omega)=\sum_{1\leq i,j,k\leq n}\hat\beta^i_jA^j_{Ik}\varepsilon^k
  $$
  and therefore
  $$
  (\omega_I)_k\leq \max_{1\leq i,j\leq n}\{\hat\beta^i_j\}\sum_{1\leq j\leq n}A^j_{Ik}=
   \max_{1\leq i,j\leq n}\{\hat\beta^i_j\}\|f^*_I(e_k)\|_{\ell^1}
  $$
  and 
  $$
  (\min_{\hat\beta^i_j>0}\hat\beta^i_j)\|f_I\|_{\ell^1}=
  (\min_{\hat\beta^i_j>0}\hat\beta^i_j)\sum_{j,k}A^j_{Ik}\leq\|\omega_I\|_{\ell^1}
  $$
  from which follows the claim of this lemma.
\end{proof}
\begin{example}
  Consider the Sierpinski gaskets $F^{\alpha,n}$ introduced in Example~\ref{ex:sg}.
  A proper frame for $(F^{\alpha,n})^*$ is given by
  $\eta^i=(\alpha-1)\varepsilon^i+\sum_{j\neq i}\varepsilon^j$,
  so that $\omega=\sum_{1\leq i\leq n}\eta^i=(\alpha+n-2)\sum_{1\leq i\leq n}\varepsilon^i$
  and therefore 
  $$
  \omega_I=f^*_I(\omega)=(\alpha+n-2)\sum_{1\leq i\leq n}f^*_I(\varepsilon^i)=
  (\alpha+n-2)\sum_{1\leq i,k\leq n}A^i_{Ik}\varepsilon^k=
  $$
  $$
  =(\alpha+n-2)\sum_{1\leq k\leq n}\|f^*_I(e_k)\|_{\ell^1}\varepsilon^k.
  $$
  If $\alpha=1$ then
  $$
  \|\omega_I\|_{\ell^1}=\sum_{i,j}\lambda_j\hat\beta^j_i=(n-1)\sum_{j}\lambda_j
  $$
  and 
  $$
  (\omega_I)_i=\sum_{j\neq i}\lambda_j
  $$
  so that
  $$
  \|\omega_I\|_{\ell^1}\leq(n-1)\left((\omega_I)_k+(\omega_I)_{k'}\right)
  $$
  or, equivalently, 
  $$
  \|f_I\|_{\ell^1}\leq(n-1)\left(\|f_I(e_k)\|_{\ell^1}+\|f(e_{k'})\|_{\ell^1}\right)
  $$
  for every $k\neq k'$

  If $\alpha>1$ then 
  $$
  \|\omega_I\|_{\ell^1}=\sum_{i,j}\lambda_j\hat\beta^j_i=(n-1)\max\{\alpha-1,1\}
  \sum_{j}\lambda_j
  $$
  and 
  $$
  (\omega_I)_i\geq\min\{\alpha-1,1\}\sum_{i}\lambda_j
  $$
  so that 
  $$
  \|\omega_I\|_{\ell^1}\leq(n-1)\min\{\alpha-1,1\}(\omega_I),
  $$
  or, equivalently, 
  $$
  \|f_I\|_{\ell^1}\leq(n-1)\min\{\alpha-1,1\}\|f_I(e_k)\|_{\ell^1},
  $$
  for all $k$.
\end{example}
\begin{proposition}
  \label{thm:fastSG}
  Every simple Sierpinski gasket is a fast gasket.
\end{proposition}
\begin{proof}
  The key fact here is that in every Sierpinski gasket $\boldsymbol{F}$
  with a proper frame $\cE$, 
  for any $i\neq j$ and every $k$, $f_{ij}(e_k)$ is linearly dependent 
  on both $e_i$ and $e_j$. Indeed if $k\neq j$ then 
  $$
  f_{ij}(e_k)=f_i(\beta_{jk}e_j+\alpha_{jk}e_k)
  =\alpha_{jk}f_i(e_k)+\beta_{jk}\left(\beta_{ij}e_i+\alpha_{ij}e_j\right),
  $$
  while if $k=j$ then 
  $$
  f_{ij}(e_j)=\alpha_{jj}(\beta_{ij}e_i+\alpha_{ij}e_j).
  $$ 
  Hence the matrix representing $f_{ij}$ with
  respect to $\cE$ has at least (actually, exactly) two rows with all 
  non-zero coefficients. This means that every column of the matrix
  representing $f_{ijL}$ is a linear combination with strictly positive 
  coefficients of $f_L(e_i)$ and $f_L(e_j)$ with possibly some other 
  positive contribution from the other vectors.

  From this we deduce immediately that 
  $$
  \|f_{IijL}\|_{\ell^1}\geq C\|f_{I}\|_{\ell^1}
  \left(\|f_L(e_i)\|_{\ell^1}+\|f_L(e_j)\|_{\ell^1}\right)
  $$
  for some $C\geq0$ and therefore, by Lemma~\ref{thm:induction}, that
  $$
  \|f_{IijL}\|_{\ell^1}\geq C'\|f_{I}\|_{\ell^1}\|f_{L}\|_{\ell^1}
  $$
  for some $C'\geq0$.
  Since $\|f_{L}\|_{\ell^1}\geq C''\|f_{ijL}\|_{\ell^1}$ for all $i\neq j$
  and some $C''>0$, our claim follows.
\end{proof}
\begin{proposition}
  Let $\boldsymbol{F}\subset Aut(\bR^n)$ be a simple Sierpinski gasket,
  $\cE$ a proper frame of $\boldsymbol{F}$ and $\mu$ any measure of $\RPnm$ 
  in the measure class of the {\em round measure}, namely the measure 
  induced by the metric of sectional curvature identically equal to 1.
  Then there exist constants $A,B,C,D>0$ s.t.
  $$
  \frac{A}{\|f_I\|^{n}}\leq\mu(T_{\bF f_I})\leq\frac{B}{\|f_I\|^{a_n}},\;\;\;
  \frac{C}{\|f_I\|^{n}}\leq\mu(Z_{\bF f_I})\leq\frac{D}{\|f_I\|^{n}},
  $$
  where $a_n=n$ if $\boldsymbol{F}$ is of the second kind and $a_n=n-1$
  if it is of the first kind.
\end{proposition}
\begin{proof}
  It is enough to prove the claim in some chart containing $T_\bF$.
  We fix coordinates $(x^1,\dots,x^n)$ so that the vectors of $\cE$
  are
  $$
  e_1=(1,0,\dots,0,1),\dots,e_{n-1}=(0,\dots,0,1,1),e_{n}=(0,\dots,0,1)
  $$
  and use the chart $x^n=1$. In this chart we pick any smooth measure $\nu$ 
  of finite total volume and with constant density equal to 1 within $T_\bF$.
  
  Note that 
  $f_I=A^i_{Ik}e_i\otimes\varepsilon^k=A^i_{Ik}e^j_i\partial_j\otimes\varepsilon^k$,
  where the last row of the matrix $A^i_{Ik}e^j_i$ contain the $\ell^1$ norms
  of the vectors $f_I(e_i)$.
  A direct calculation shows that 
  $$
  \mu(T_{\bF f_I})
  =
  \frac{1}{n!\displaystyle\prod_{k=1}^{n}A^i_{Ik}e_i^{n}}
  =
  \frac{1}{n!\displaystyle\prod_{k=1}^{n}\|f_I(e_k)\|_{\ell^1}}.
  $$
  Clearly 
  $\displaystyle
  \frac{\displaystyle\prod_{k=1}^{n}\|f_I(e_k)\|_{\ell^1}}{\|f_I\|^n_{\ell^1}}
  \leq
  1.
  $
  By Proposition~\ref{thm:induction}, if $\boldsymbol{F}$ is of the second kind then
  $$
  A
  \leq
  \frac{\displaystyle\prod_{k=1}^{n}\|f_I(e_k)\|_{\ell^1}}{\|f_I\|^n_{\ell^1}}
  $$
  for some $A>0$.
  If it is of the first kind assume, for the argument sake, that
  $\|f_I(e_1)\|_{\ell^1}\leq\dots\|f_I(e_n)\|$. Then
  $$
  \frac{\|f_I\|_{\ell^1}^{n-1}}{\displaystyle\prod_{k=1}^{n}\|f_I(e_k)\|_{\ell^1}}
  =
  \frac{(\displaystyle\sum_{k=1}^{n}\|f_I(e_k)\|_{\ell^1})^{n-1}}{\displaystyle\prod_{k=1}^{n}\|f_I(e_k)\|_{\ell^1}}
  \leq
  $$
  $$
  \leq
  C \prod_{k=1,n-1}\frac{\|f_I(e_k)\|_{\ell^1}+\|f_I(e_{k+1})\|_{\ell^1}}{\|f_I(e_{k+1})\|_{\ell^1}}\leq 2^{n-1}C
  $$
  The geometry of $Z_{\bF f_I}$ is more complex. We divide it in $n$ 
  $(n-1)$-simplices $Z_i$, where $Z_i$'s vertices are the $(n-1)$
  points $[f_i(f_I(e_j))]$, $j\neq i$, plus the point $[\sum_{1\leq i\leq n}f_I(e_i)]$.
  Then
  $$
  \mu(Z_i)=\frac{1}{\displaystyle n!\|f_I(\sum_{1\leq i\leq n}e_i)\|_{\ell^1}
    \prod_{j\neq i}\|\beta_{ij}(e_I)_i+\gamma_i(e_I)_j\|_{\ell^1}}
  = \frac{1}{\displaystyle n!\|f_I\|_{\ell^1}
    \prod_{j\neq i}\|\beta_{ij}(e_I)_i+\gamma_i(e_I)_j\|_{\ell^1}}
  $$
  since all components of all vectors are positive.
  Even in this case then 
  $$
  \frac{1}{\mu(Z_i)\|f_I\|^n_{\ell^1}}
  \leq
  n!
  $$
  and
  $$
  \frac{\|f_I\|_{\ell^1}^{n}}{\|f_I\|_{\ell^1}
    \prod_{j\neq i}\|\beta_{ij}(e_I)_i+\gamma_i(e_I)_j\|_{\ell^1}}
  \leq
  \frac{\|f_I\|_{\ell^1}^{n-1}}{\max\{\beta,\gamma\}\prod_{j\neq i}\|f_I(e_i)+f(e_j)\|_{\ell^1}}
  \leq A
  $$
  for some $A$.
\end{proof}
Next proposition supports the idea that the residual set of a Sierpinski
gasket, or at least of a simple one, have non-integer dimension:
\begin{proposition}
  The residual set of a simple Sierpinski gasket 
  $\bF\subset SL_n^\pm(\bR)$ is a null set with respect to the measure class of 
  the round measure on $\bR P^{n-1}$.
\end{proposition}
\begin{proof}
  From the previous proposition we see that
  $$
  1\leq\frac{\mu(T_{\bF f_I})}{\mu(Z_{\bF f_I})}\leq\frac{C}{\|f_I\|}
  $$
  if $\boldsymbol{F}$ is of the first kind and
  $$
  1\leq\frac{\mu(T_{\bF f_I})}{\mu(Z_{\bF f_I})}\leq C
  $$
  if it is of the second.
  Let us first assume that $\bF$ is of the second kind and 
  let 
  $$S_k=\sum_{|I|=k+1}\mu(T_{\bF f_I})\hbox{ and }P_k=\sum_{|I|=k}\mu(Z_{\bF f_I}).$$
  Then
  $$
  S_{k+1}=S_{k}-P_{k}\leq S_{k}(1-C)
  $$
  and therefore
  $$
  S_k-S_{k+1}\geq C S_k
  $$
  so that, after making a telescopic sum, we get
  $$
  S_1-\lim_{k\to\infty}S_k\geq C\lim_{k\to\infty}k S_k
  $$
  which immediately implies that $\lim_{k\to\infty}S_k=0$.

  If $\boldsymbol{F}$ is of the first kind then 
  $\mu(T_{\bF f_I})/\mu(Z_{\bF f_I})\leq C/\min_{|J|=|I|}\|f_J\|$.
  It is easy to check that $\min_{|J|=|I|}\|f_J\|$ is proportional to $|I|$
  and therefore
  $$
  S_{k+1}=S_{k}-P_{k}\leq S_{k}(1-C/k)
  $$
  and
  $$
  S_k-S_{k+1}\geq C S_k/k
  $$
  so that
  $$
  S_1-\lim_{k\to\infty}S_k\geq C \lim_{k\to\infty}S_k\sum_{1\leq j\leq k}1/j.
  $$
  Since the series $1/j$ diverges we get again that $\lim_{k\to\infty}S_k=0$.
\end{proof}
%
%
\subsubsection{Affine Sierpinski Gaskets}
\label{sec:aSG}
In~\cite{FL98} Falconer and Lammering studied in detail the family 
of affine Sierpinski gaskets $S_{a,b}$, $a,b\in(0,1)$,  
defined by the affine transformations
\begin{align*}
  S_1\begin{pmatrix}x\cr y\cr\end{pmatrix}&=
  (1-a)\begin{pmatrix}x\cr y\cr\end{pmatrix}
  +\begin{pmatrix}0\cr a\cr\end{pmatrix}\cr
  S_2\begin{pmatrix}x\cr y\cr\end{pmatrix}&=
  \begin{pmatrix}b&0\cr 0&a\cr\end{pmatrix}
\begin{pmatrix}x\cr y\cr\end{pmatrix}\cr
  S_3\begin{pmatrix}x\cr y\cr\end{pmatrix}&=
\begin{pmatrix}1-b&1-a-b\cr 0&a\cr\end{pmatrix}
  +\begin{pmatrix}b\cr 0\cr\end{pmatrix}.\cr
\end{align*}
Notice that the case $a=\frac{1}{2},b=\frac{1}{2}$ 
corresponds to the standard Sierpinski gasket.
In particular they proved that the box dimension of the corresponding
residual set $R_{a,b}$ is given by the unique root of the equation
\begin{equation}
  \label{eq:FL1}
  (1-a)^s+ab^{s-1}+a(1-b)^{s-1}=1
\end{equation}
in the triangle $T_1=\{(a,b)\in(0,1)^2,a\geq\max\{b,1-b\}\}$
and of 
\begin{equation}
  \label{eq:FL2}
  (1-a)^s+a^{s-1}=1
\end{equation}
in the opposite triangle $T_2=\{(a,b)\in(0,1)^2,a\leq\min\{b,1-b\}\}$.
This setting provides a convenient source of examples for comparing 
the box dimension of the 
residual set of a real projective Sierpinski gasket with the corresponding
gasket exponent. Indeed the injection sending the affine transformation
$S(x)=Tx+v$, with $T\in M_2(\bR)$ and $x,v\in\Rt$, into the $3\times3$ 
matrices $S\to\begin{pmatrix}T&v\cr 0&1\end{pmatrix}$ applied to the $S_i$
gives the three matrices
$$
M_1=\begin{pmatrix}1-a&0&0\cr 0&1-a&a\cr 0&0&1\cr\end{pmatrix},
M_2=\begin{pmatrix}b&0&0\cr 0&a&0\cr 0&0&1\cr\end{pmatrix},
M_3=\begin{pmatrix}1-b&1-a-b&b\cr 0&a&0\cr 0&0&1\cr\end{pmatrix}
$$
so that the action induced by the $A_i$ on $\bR P^2$ in the 
affine chart $z=1$ coincides with the action of the $S_i$. 
\begin{figure}
  \centering
  \begin{tabular}{cc}
    \includegraphics[width=6.5cm,clip=true,trim=80 250 20 100]{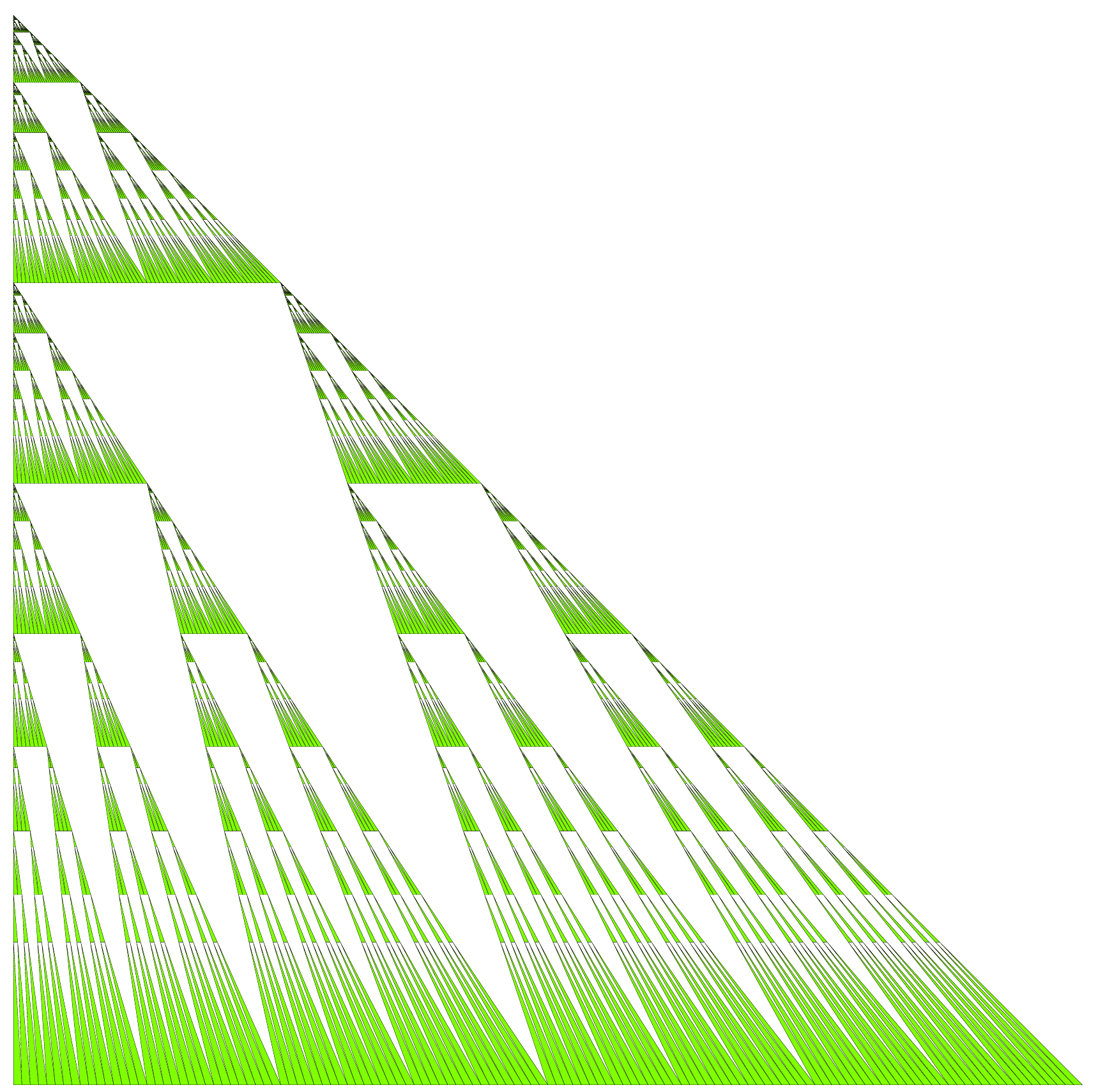}&
    \includegraphics[width=6.5cm,clip=true,trim=80 200 20 100]{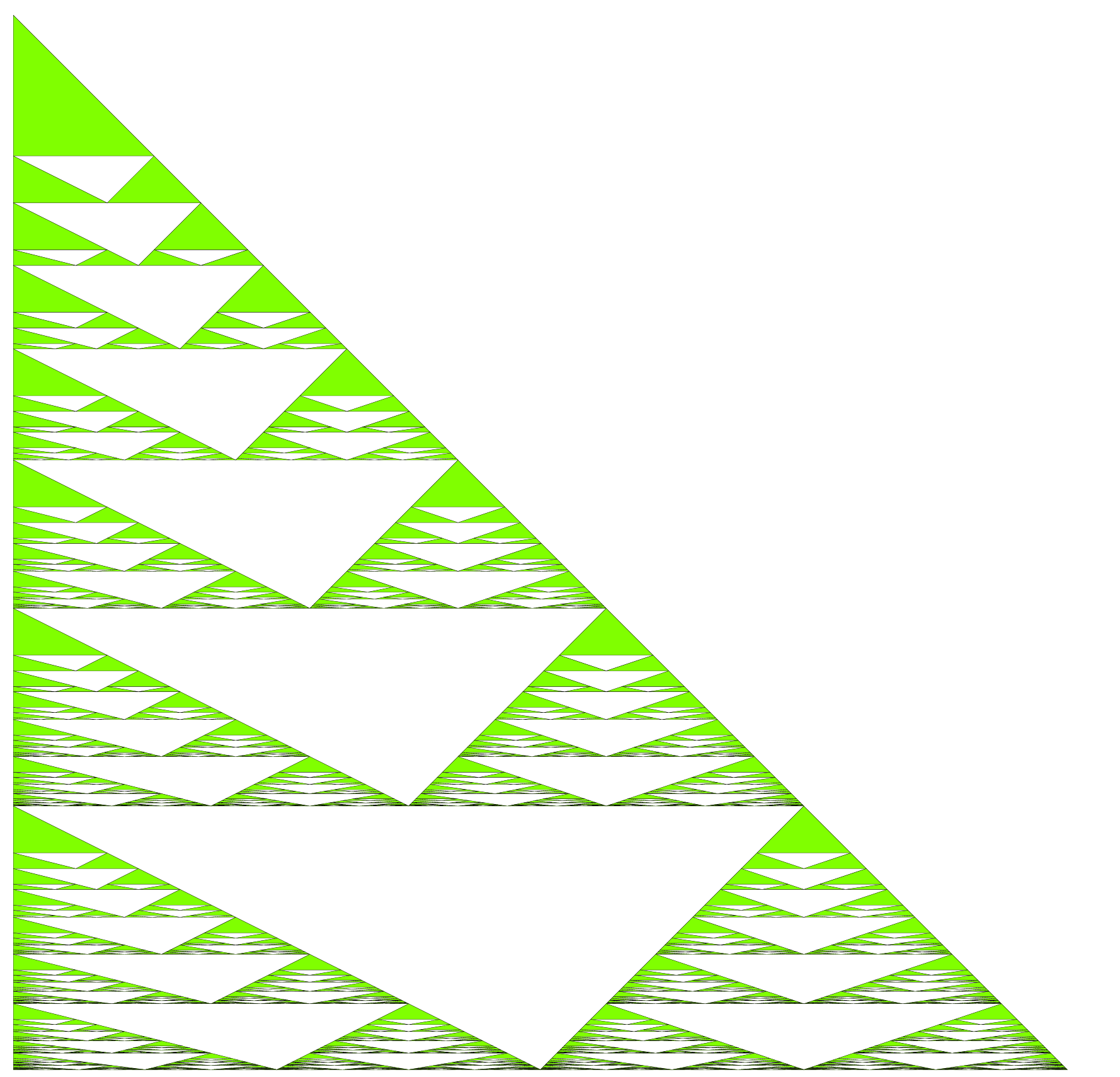}\cr
    $a=\frac{3}{4},b=\frac{1}{2}$&$a=\frac{1}{4},b=\frac{1}{2}$\cr
    \includegraphics[width=6.5cm,clip=true,trim=80 250 20 100]{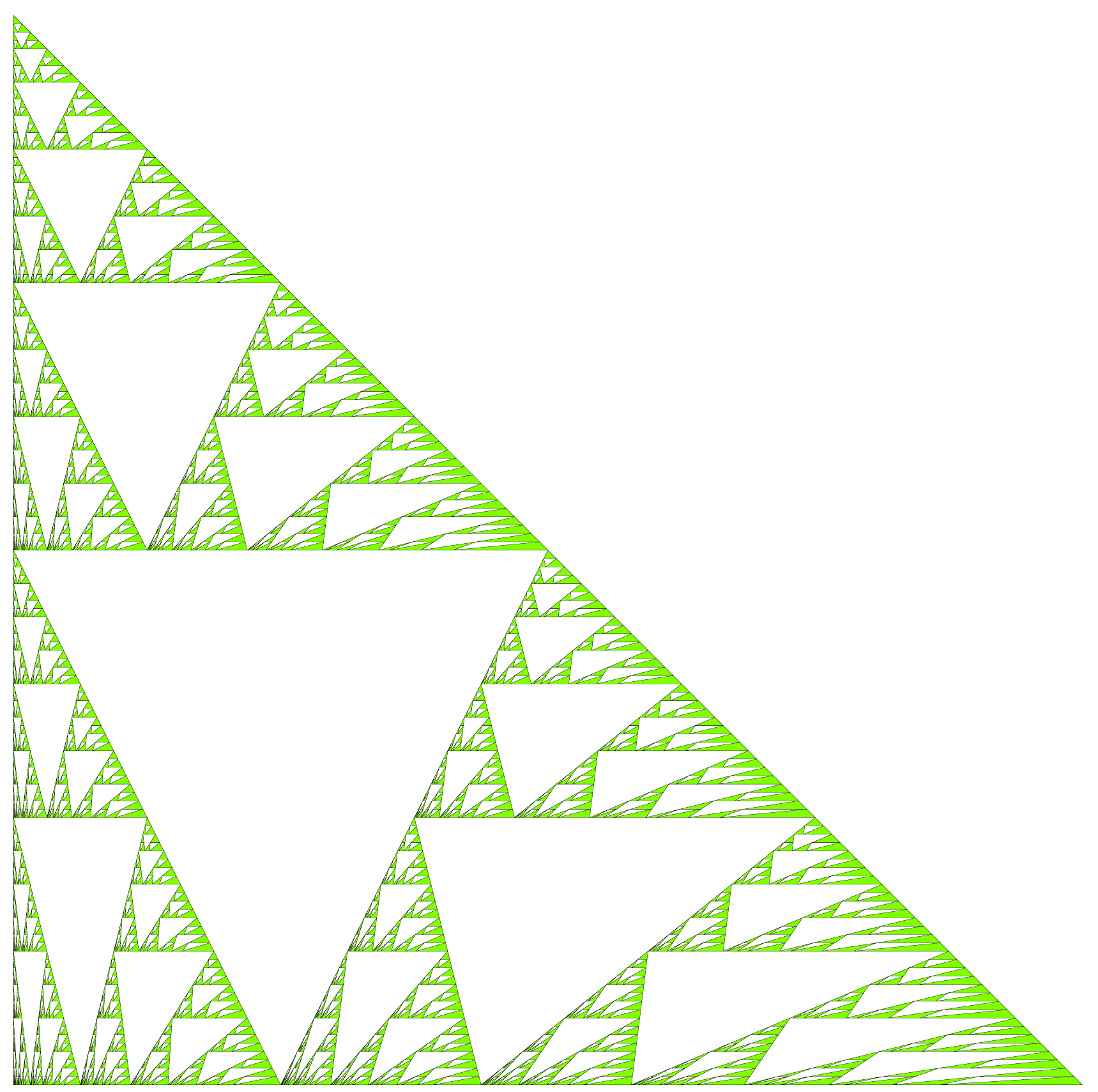}&
    \includegraphics[width=6.5cm,clip=true,trim=80 250 20 100]{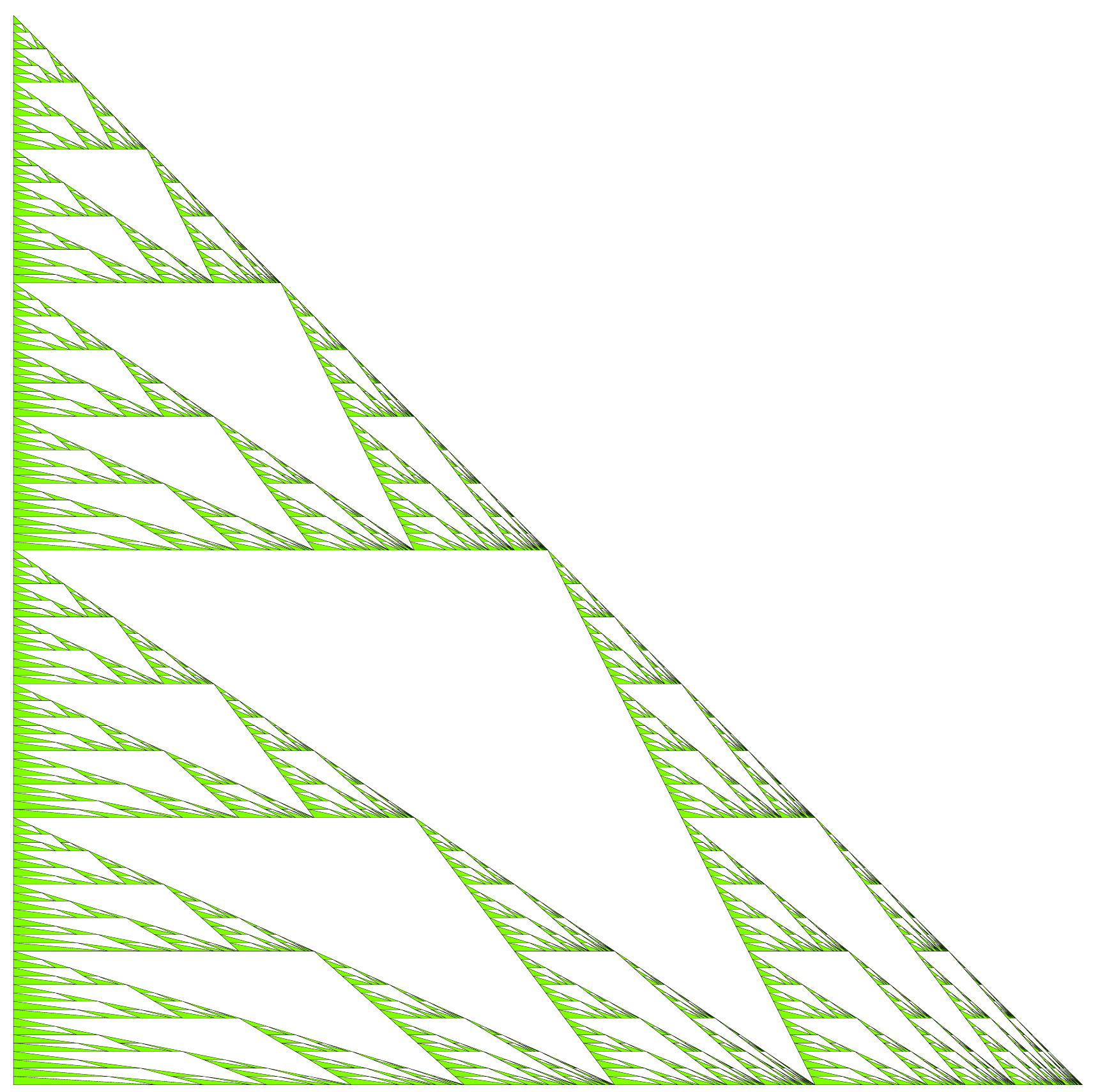}\cr
    $a=\frac{1}{2},b=\frac{1}{4}$&$a=\frac{1}{2},b=\frac{3}{4}$\cr
  \end{tabular}
  \caption{%
    Affine Sierpinski gaskets $\bA_{a,b}$ for four possible pairs $a,b$.
    For each one we plot (in green) the set $T_{7,\bA_{a,b}}$. For the
    upper two the box dimensions can be evaluated analytically and 
    their first six digits are, respectively,  
    $\dim_B R_{\frac{3}{4},\frac{1}{2}}=1.72368$ and 
    $\dim_B R_{\frac{1}{4},\frac{1}{2}}=1.68886$.
    The numerical evaluation of the box dimensions with 
    an elementary box-counting algorithm gives: 
    $\dim_B R_{\frac{3}{4},\frac{1}{2}}\simeq1.71$, 
    $\dim_B R_{\frac{1}{4},\frac{1}{2}}\simeq1.66$, 
    $\dim_B R_{\frac{1}{2},\frac{1}{4}}\simeq1.60$,
    $\dim_B R_{\frac{1}{2},\frac{3}{4}}\simeq1.60$.
    See Table~\ref{tab:aSG} for a comparison of the box dimension
    of these and other affine gaskets with the expoentn of the
    corresponding real projective Sierpinski gaskets.
  }
  \label{fig:aSG}
\end{figure}

Like in Example~\ref{ex:cool}, these matrices are
upper triangular and $\|M_I\|=1$
for all $I\in\cI^3$. To see this first of all we let
$M=\begin{pmatrix}\alpha&\lambda&\mu\cr 0&\beta&\nu\cr 0&0&1\cr\end{pmatrix}$
and notice that $0\leq\lambda+\beta\leq1$ and $0\leq\mu+\nu\leq1$.
Indeed we can limit the discussion to left multiplication by
$M_1$ and $M_3$ and it is immediate to verify by induction that,
assuming the inequalities above for $M_I$, the products 
$$
M_1 M=\begin{pmatrix}
(1-a)\alpha&(1-a)\lambda&(1-a)\mu\cr 0&(1-a)\beta&(1-a)\nu+\alpha\cr 0&0&1\cr
\end{pmatrix},
$$
and 
$$
M_3 M=\begin{pmatrix}
(1-b)\alpha&(1-b)\lambda+(1-a-b)\beta&(1-b)\mu+(1-a-b)\nu+b\cr 
0&a\beta&a\nu\cr 0&0&1\cr
\end{pmatrix},
$$
satisfy the same inequalities.
Now consider the semigroup $\bA_{a,b}$ generated by the 
$A_i=M_i/\det M_i^{1/3}\in SL_3(\bR)$.
Clearly $\|A_I\|=\det M_{i_1}^{-1/3}\cdots\det M_{i_k}^{-1/3}$ for every 
$I=i_1\cdots i_k$ and therefore
$$
\zeta_{\bA_{a,b},k}=\sum_{I\in\cI^3_k}\|A_I\|^{-s}=
(\det M_1^{s/3}+\det M_2^{s/3}+\det M_3^{s/3})^k.
$$
By Theorem~2 then the exponent $s_{\bA_{a,b}}$ is the unique solution 
of the equation 
\begin{equation}
  \label{eq:sab}
(1-a)^{2s/3}+(ab)^{s/3}+(a(1-b))^{s/3}=1.
\end{equation}
%
\begin{proposition}
  $2s_{\bA_{a,b}}\leq3\dim_B R_{a,b}$ within the two triangles $T_{1,2}$,
  with the equal sign holding only in their common vertex.
\end{proposition}
\begin{proof}
  After writing (\ref{eq:sab}) in terms of $t=2s/3$ and renaming $t$ to
  $s$ we are left with the equation $(1-a)^s+(ab)^{s/2}+(a(1-b))^{s/2}=1$.
  Comparing this expression with the left-hand sides of (\ref{eq:FL1}) 
  and (\ref{eq:FL2}) we see that it is enough to prove that 
  $$
  (ab)^{s/2}+(a(1-b))^{s/2}\leq\min\{a^{s-1},ab^{s-1}+a(1-b)^{s-1}\}.
  $$
  Since for obvious geometrical reasons $\dim_B R_{a,b}\leq2$ we can assume 
  in the following $s\geq2$.
  Let us denote respectively by $f_{a,b}(s),g_a(s),h_{a,b}(s)$ the three functions
  above and notice that, since by hypothesis $a,b,1-b\in(0,1)$, they are 
  all strictly monotonically decreasing functions of $s$ converging
  to 0 as $s\to\infty$. Moreover $f_{a,b}(2)=g_a(2)=h_{a,b}(2)=a$ and since
  these functions can have only one intersection it is enough
  to verify their behaviour for $s\to0$. A direct calculation shows 
  that, for every pair $a,b\in(0,1)^2$, we have that
  $\lim_{s\to0}f_{a,b}(s)=2$ while $\lim_{s\to0}g_{a}(s)=\lim_{s\to0}h_{a}(s)=\infty$.
\end{proof}
Numerical experiments (see Fig.~\ref{fig:aSG}) clearly
suggest that, even outside of the triangles $T_{1,2}$, 
$\dim_B R_{a,b}$
always larger than $\frac{2}{3}s_{\bA_{a,b}}$
with the only exception of the case $a=1/2,b=1/2$, when
these two quantities coincide. Moreover it appears that,
roughly, $\frac{2}{3}s_{\bA_{a,b}}\geq\frac{9}{10}\dim_B R_{\bA}$.
%
\subsubsection{The cubic semigroups $\boC^\alpha_n$}
\label{sec:cubic}
Recall that, by definition, $f_i^\alpha(e_i)=\alpha e_i$, 
$f_i^\alpha(e_j)=e_j+e_i$, $j\neq i$. 

\medskip\noindent
$\boldsymbol{n=3}.$ In $\bR^3$ we use coordinates $(x,y,z)$
with respect to the frame $e'_1=e_1+e_3$, $e'_2=e_2+e_3$, $e'_3=e_3$,
so that the $f_i$ are represented by the matrices 
$$
A_1=\begin{pmatrix}\alpha-1&0&1\cr 0&1&0\cr \alpha-2&0&2\cr\end{pmatrix},
A_2=\begin{pmatrix}1&0&0\cr 0&\alpha-1&1\cr 0&\alpha-2&2\cr\end{pmatrix},
A_3=\begin{pmatrix}1&0&0\cr 0&1&0\cr 2-\alpha&2-\alpha&\alpha\cr\end{pmatrix}.
$$
\begin{table}
  \centering
  \begin{tabular}{|c|c|c|c|c|c|c|}
    \hline
    $a$&$b$&
    \begin{tabular}{c}
      $s_\bA$\\ (num.)\\
    \end{tabular}
    &
    \begin{tabular}{c}
      $s_\bA$\\ (anal.)\\
    \end{tabular}
    &
    $2s_\bA/3$&
    \begin{tabular}{c}
      $\dim_B R_\bA$\\ (num.)\\
    \end{tabular}
    &
    \begin{tabular}{c}
      $\dim_B R_\bA$\\ (anal.)\\
    \end{tabular}\cr
    \hline
    $1/4$&$1/2$&2.44&2.42632&1.61755&1.66&1.68886\cr
    \hline
    $3/4$&$1/2$&2.48&2.45425&1.63617&1.71&1.72368\cr
    \hline
    $1/2$&$3/4$&2.35&2.34443&1.56295&1.60&--\cr
    \hline
    $1/5$&$3/10$&2.44&2.43735&1.62490&1.76&1.71262\cr
    \hline
    $4/5$&$3/10$&2.47&2.46960&1.64640&1.77&--\cr
    \hline
    $3/10$&$1/5$&2.43&2.37354&1.58236&1.75&--\cr
    \hline
    $7/10$&$1/5$&2.35&2.35249&1.56833&1.72&1.63373\cr
    \hline
  \end{tabular}
  \caption{%
    Values of the exponents of affine Sierpinski gaskets $\bA_{a,b}$
    for several pairs $a,b$ and of the box dimension of the
    corresponding residual sets $R_{a,b}$. Numerical evaluations
    for $s_\bA$ were done by calculating $N_\bA(k)$ for the values
    $k=2^p$, $p=1,\dots,12$, and are presented to motivate our
    confidence in a relative error not bigger than $1\%$ in the
    other evaluations provided throughout the paper when an
    analytical evaluation is not available. Numerical evaluations
    for the box dimension of $R_{a,b}$ were done via an elementary
    box-counting algorithm and a comparison with the available 
    analytical evaluations suggest that their relative error 
    is about $10\%$.
  }
  \label{tab:aSG}
\end{table}
In the affine chart $[x:y:z]\to(u,v)=(x/z,y/z)$ of $\bR P^2$ we have therefore that
$$
\begin{cases}
  \psi_1(x,y) = \left(\frac{(\alpha-1)u+1}{(\alpha-2)u+2},\frac{v}{(\alpha-2)u+2}\right)\cr 
  \noalign{\medskip}
  \psi_2(x,y) = \left(\frac{u}{(\alpha-2)v+2},\frac{(\alpha-1)v+1}{(\alpha-2)v+2}\right)\cr 
  \noalign{\medskip}
  \psi_3(x,y) = \left(\frac{u}{(2-\alpha)(u+v)+\alpha},\frac{v}{(2-\alpha)(u+v)+\alpha}\right)\cr 
\end{cases}
$$
and the vertices of the invariant triangle $T_{F_3^\alpha}$ are $[e_1]=(1,0)$, 
$[e_2]=(0,1)$ and $[e_3]=(0,0)$. A direct calculation of the eigenvalues 
of the Jacobian matrices $D\psi_i$ shows that, within $T_{F_3^\alpha}$, 
$$
\min\{\frac{1}{\alpha},\frac{\alpha}{4}\}d(x,y)
\leq
d(\psi_i(x),\psi_i(y))
\leq
\max\{\frac{1}{\alpha},\frac{\alpha}{4}\}d(x,y)
$$
for all $i=1,2,3$, namely the semigroup $\langle \psi_1,\psi_2,\psi_3\rangle$, 
{\em as a IFS}, is hyperbolic for $\alpha\in(1,4)$ and parabolic for $\alpha=1,4$. 
The $\psi_i$ are not contractions 
{\em with respect to the Euclidean distance in this chart} for
the other values of $\alpha$ (see Fig.~\ref{fig:rSG} for the plot of 
$T_{7,F_3^\alpha}$ for several values of $\alpha$).

Analytical bounds for the Hausdorff dimension of the residual sets 
$R_{F_3^\alpha}$ can be obtained via Propositions 9.6 and 9.7 in~\cite{Fal90}, 
namely 
$$
\min\{\log_3\frac{4}{\alpha},\frac{1}{\log_3\alpha}\}
\leq
\dim_H R_{F_3^\alpha}
\leq
\max\{\log_3\frac{4}{\alpha},\frac{1}{\log_3\alpha}\}.
$$ 
For $\alpha=2$ we get, as expected, $\dim_H R_{F_3^\alpha}=\log_32$.
\begin{figure}
  \centering
  \begin{tabular}{ccc}
    \includegraphics[width=4.5cm,clip=true,trim=80 250 20 100]{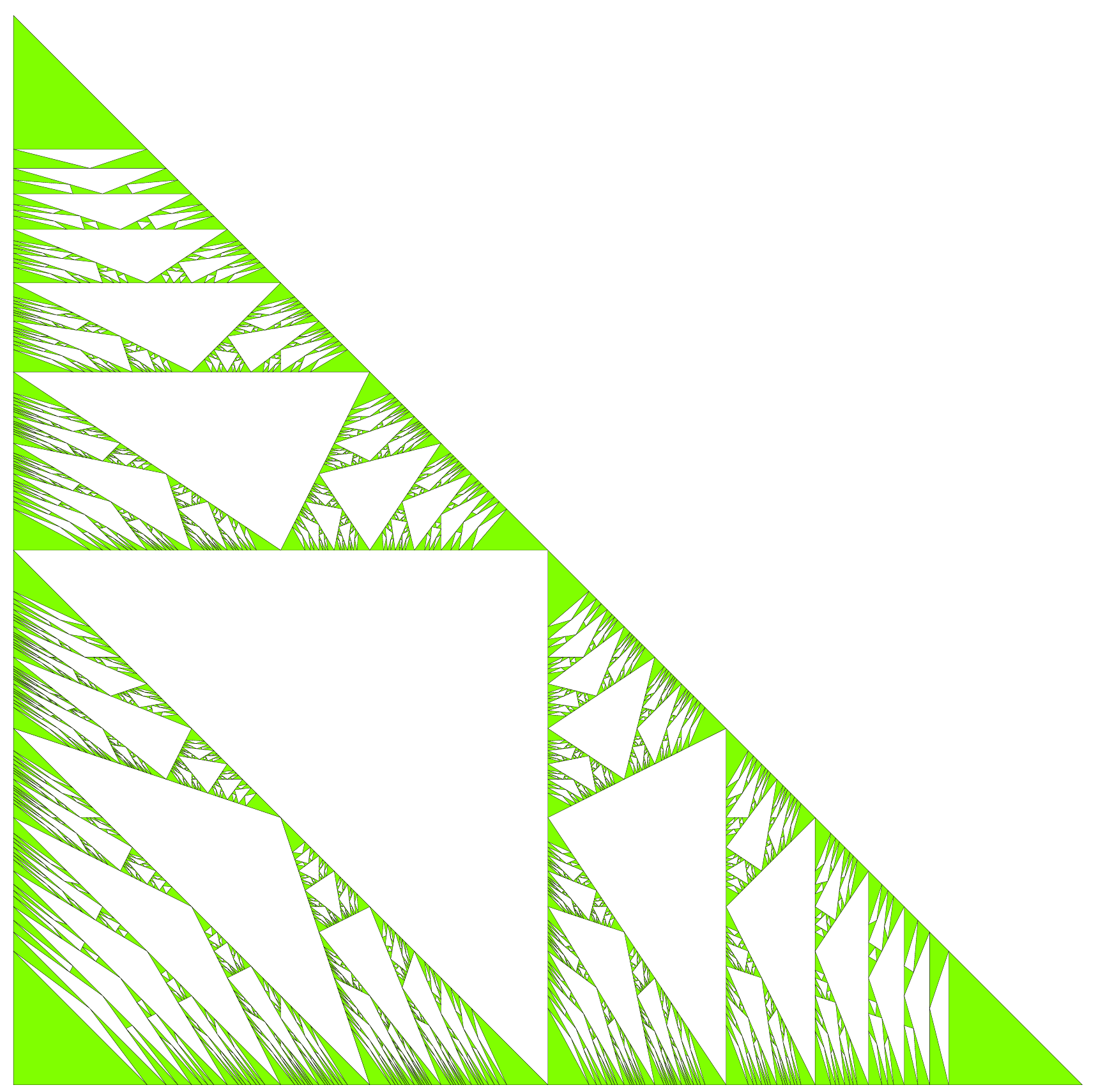}&
    \includegraphics[width=4.5cm,clip=true,trim=80 250 20 100]{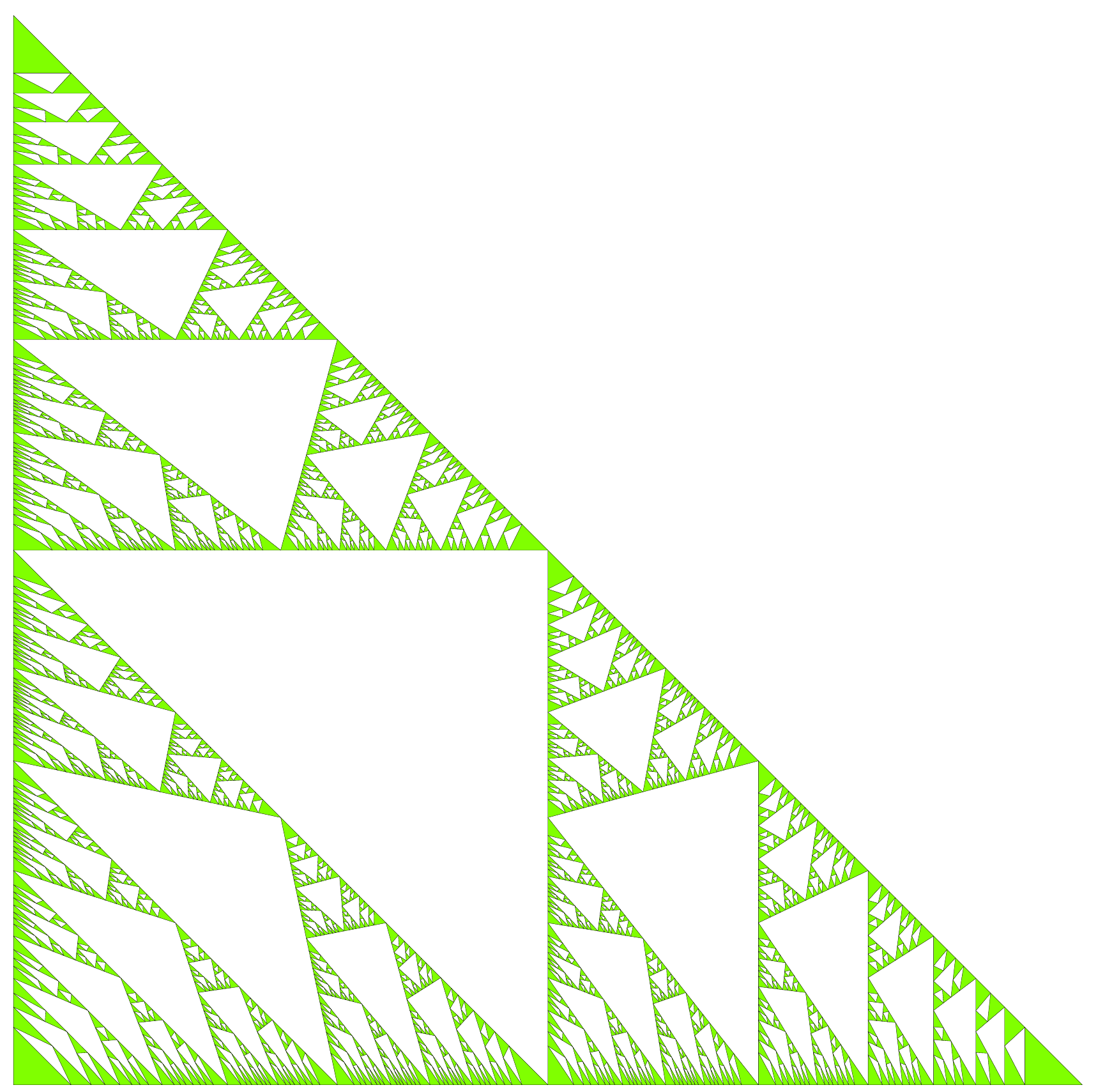}&
    \includegraphics[width=4.5cm,clip=true,trim=80 250 20 100]{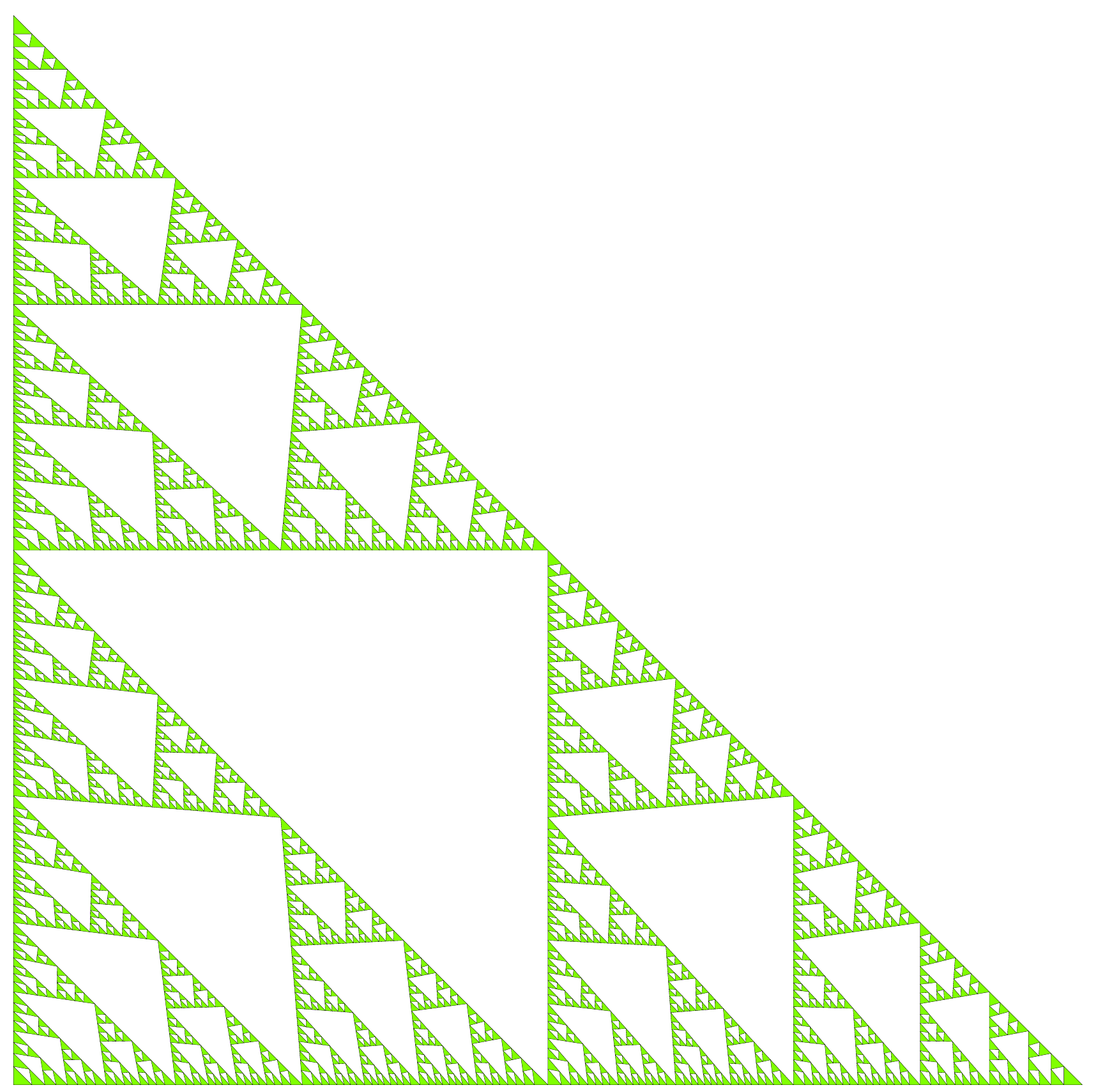}\cr
    1&1.3&1.7\cr
    \includegraphics[width=4.5cm,clip=true,trim=80 250 20 100]{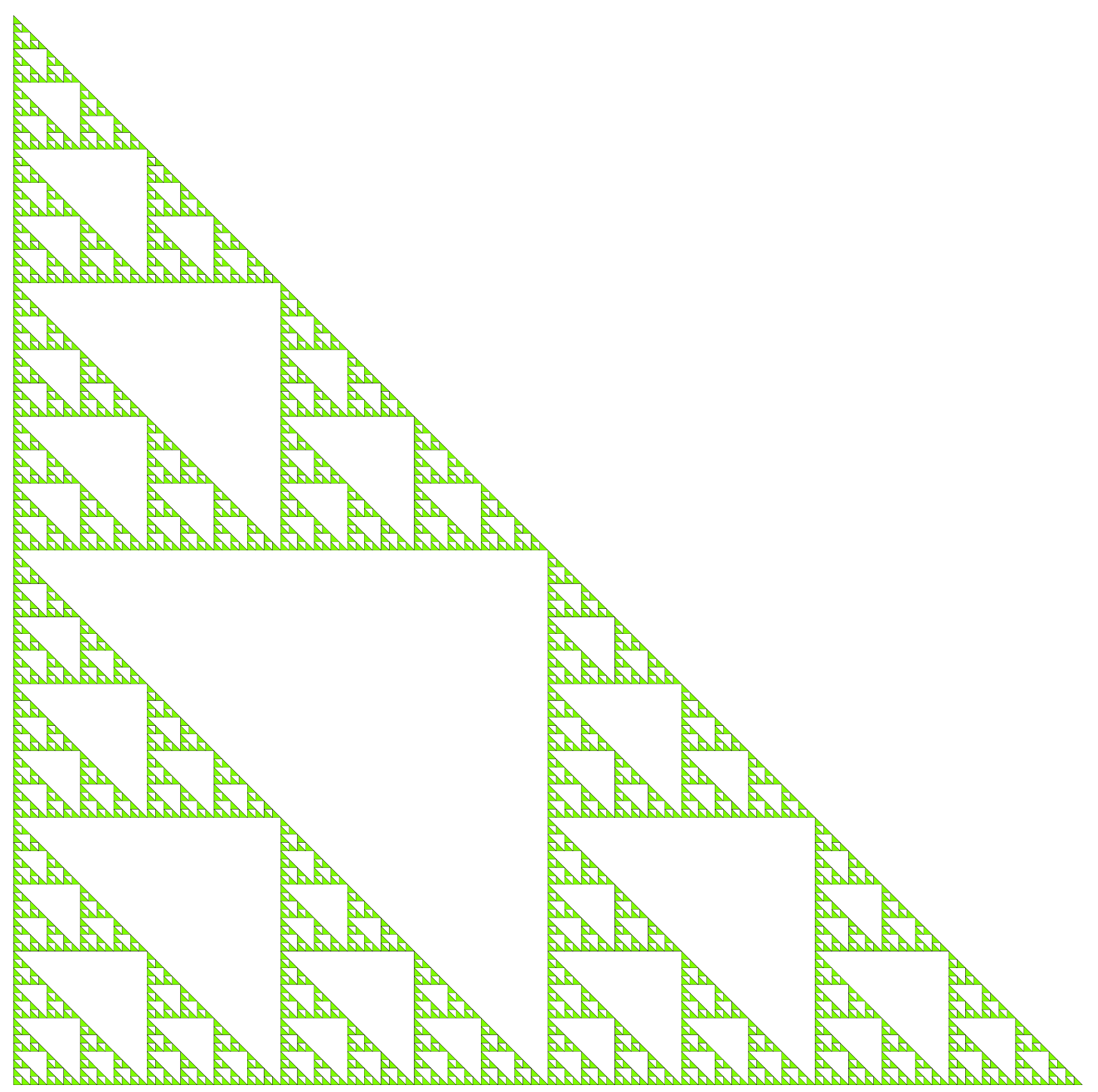}&
    \includegraphics[width=4.5cm,clip=true,trim=80 250 20 100]{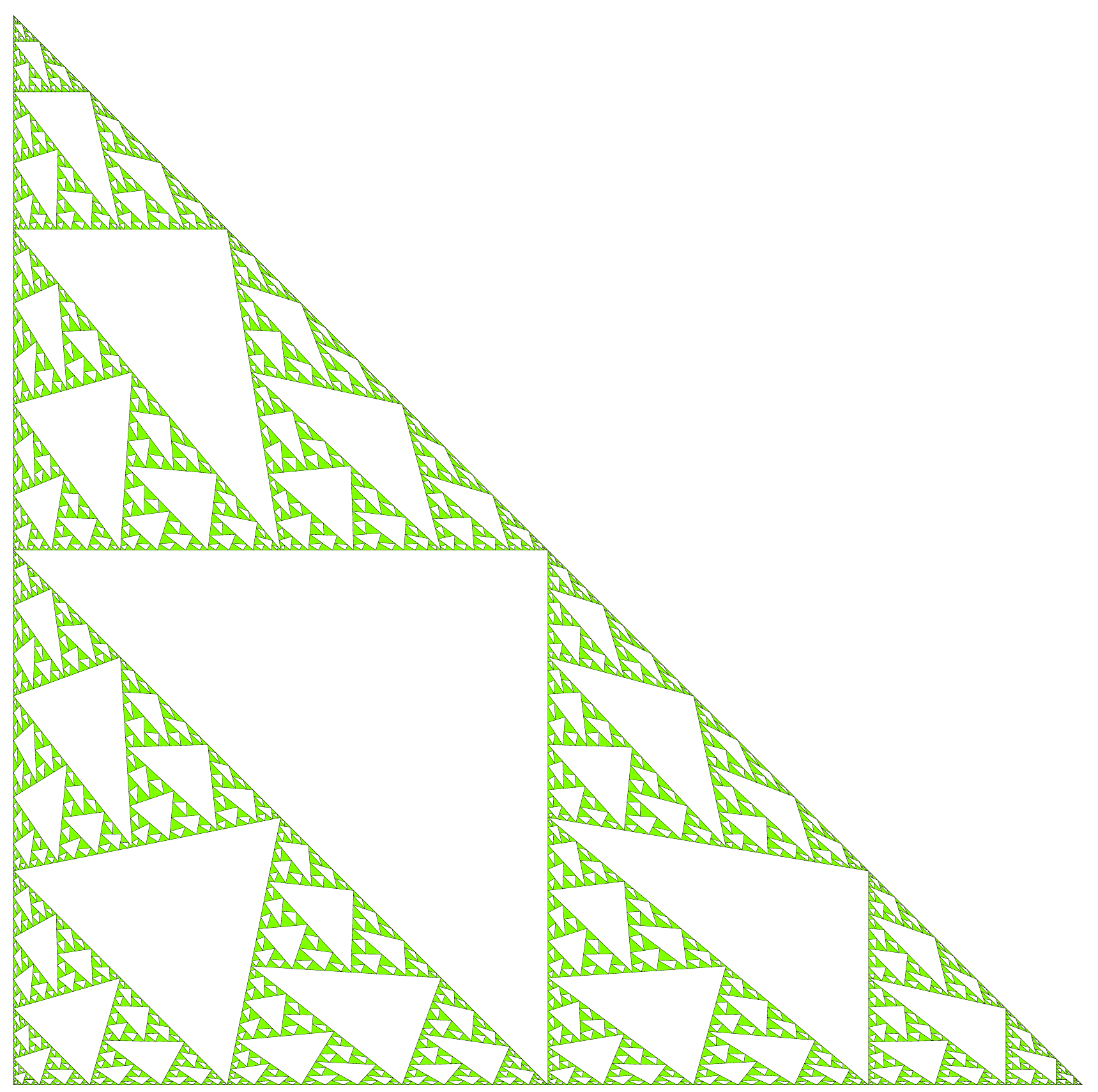}&
    \includegraphics[width=4.5cm,clip=true,trim=80 250 20 100]{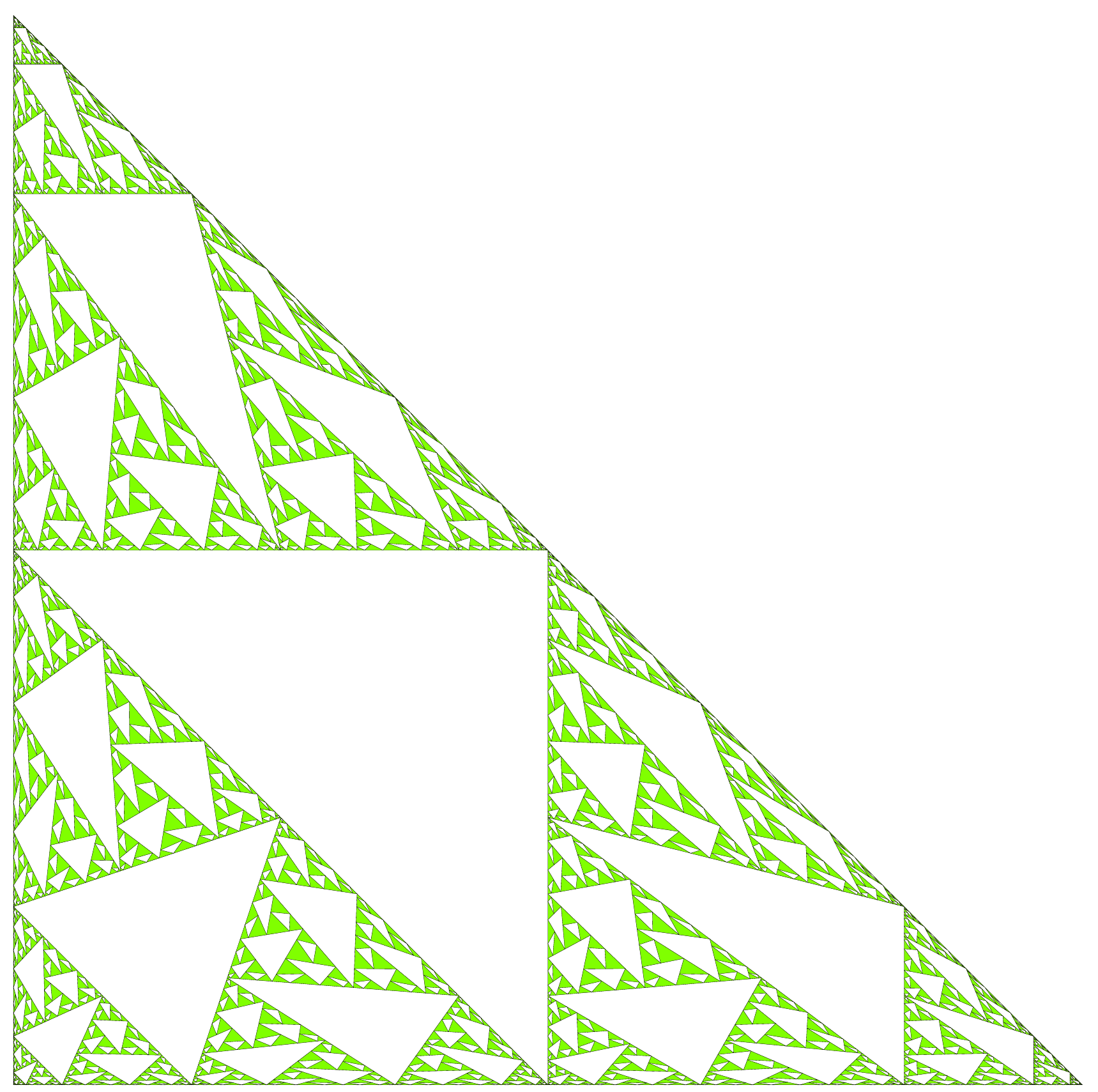}\cr
    2&3&4\cr
    \includegraphics[width=4.5cm,clip=true,trim=80 250 20 100]{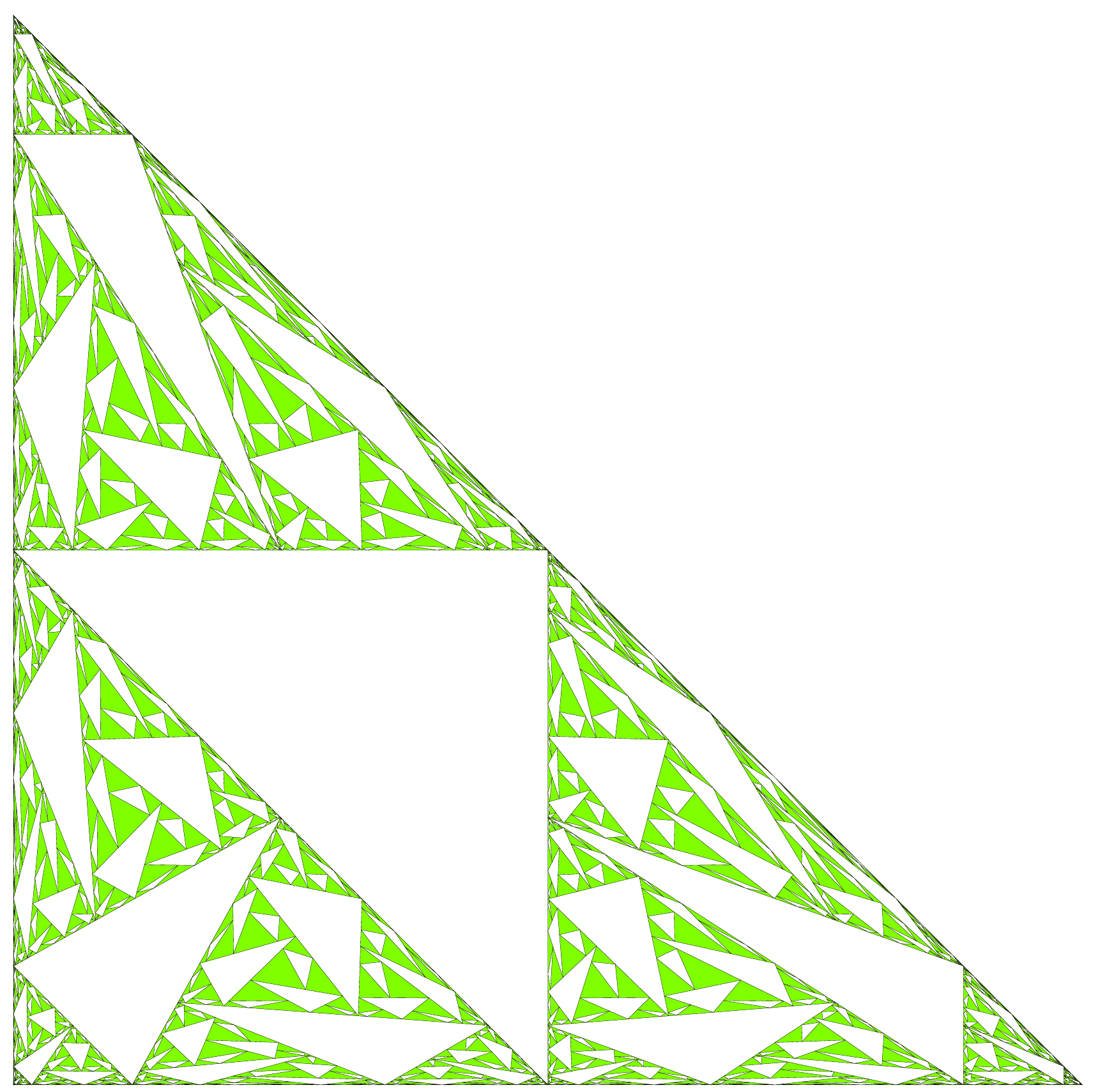}&
    \includegraphics[width=4.5cm,clip=true,trim=80 250 20 100]{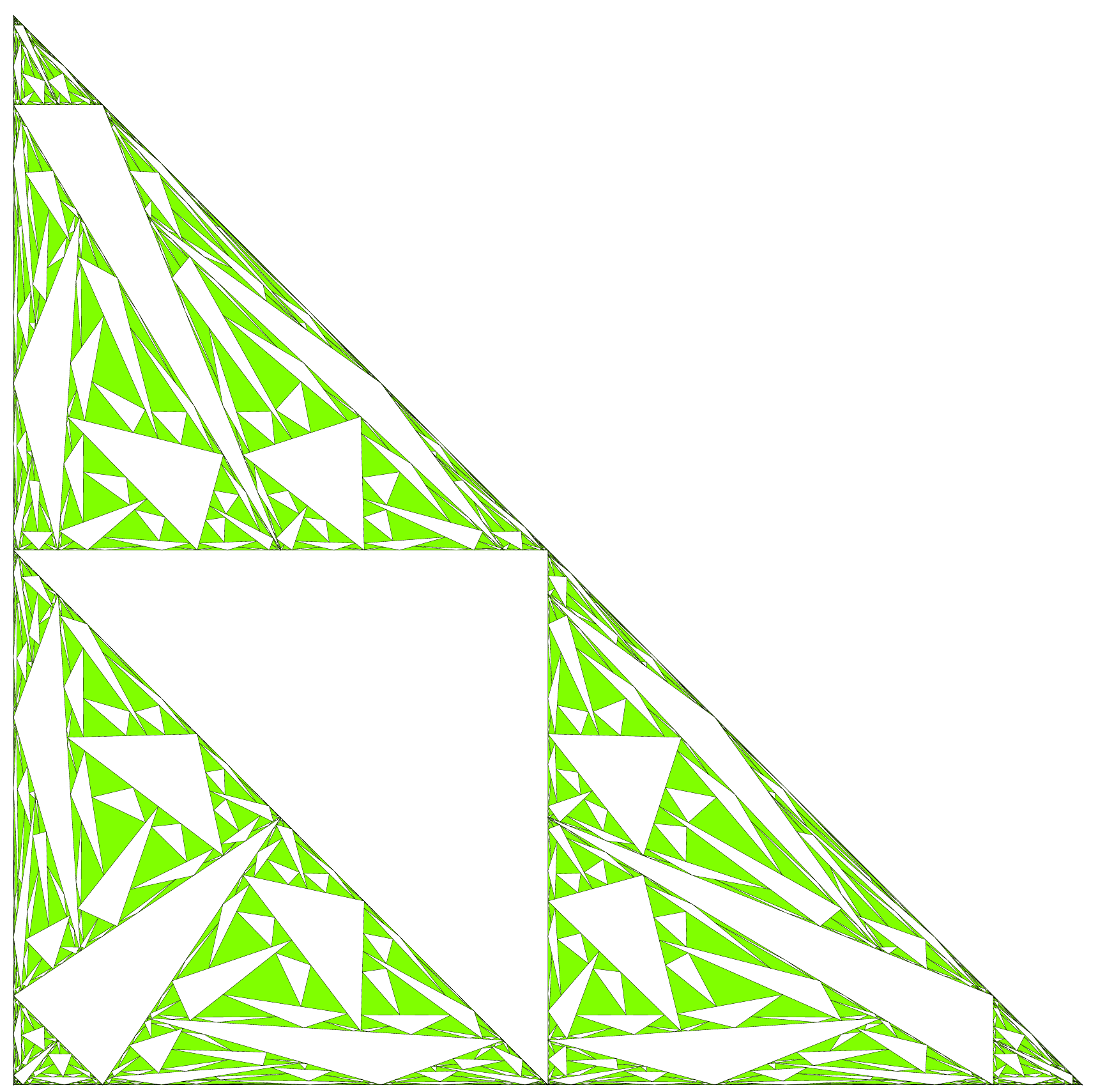}&
    \includegraphics[width=4.5cm,clip=true,trim=80 250 20 100]{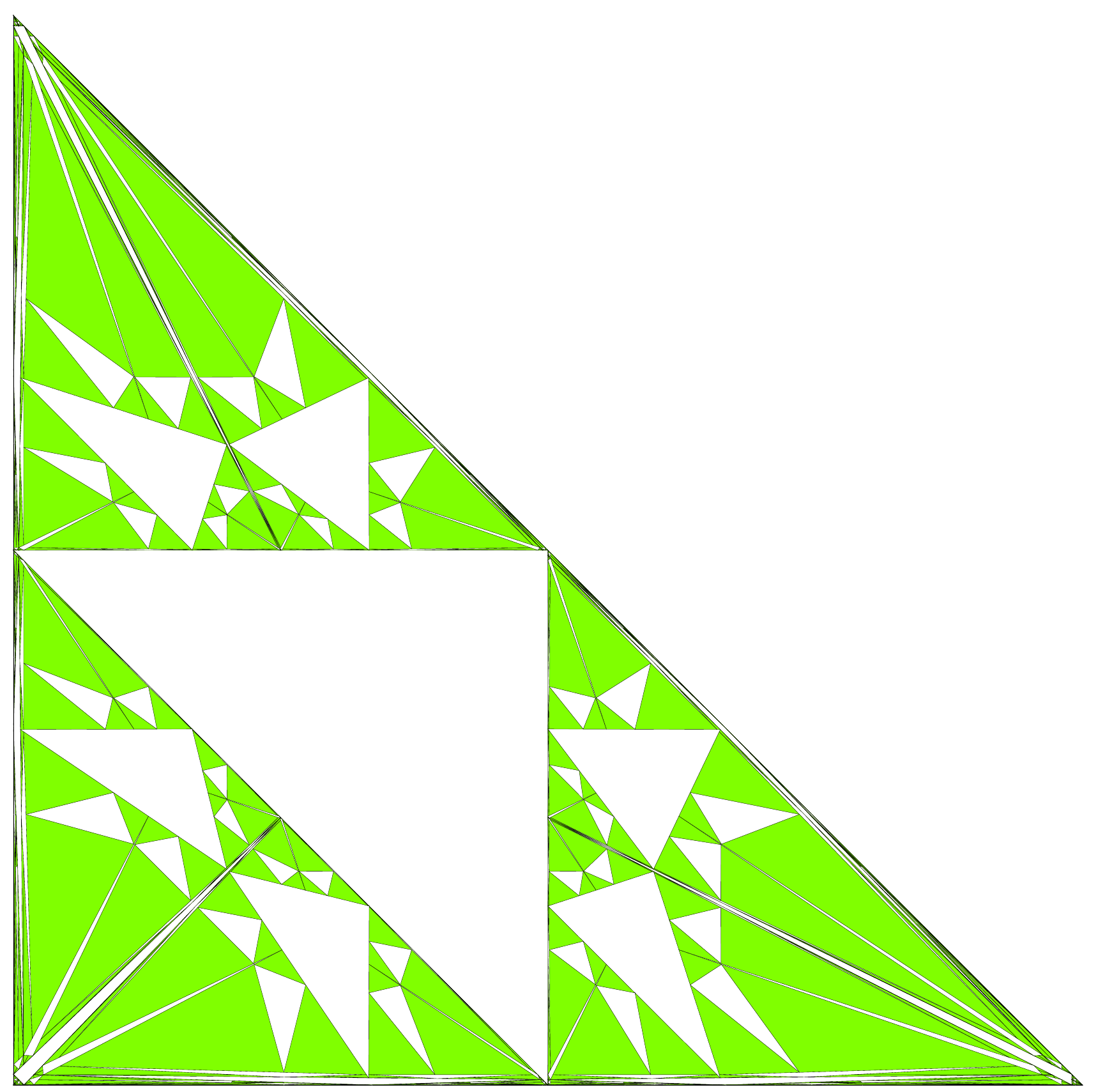}\cr
    7&10&100\cr
  \end{tabular}
  \caption{%
    Real projective Sierpinski gaskets $F_3^\alpha$ for several 
    values of $\alpha$. For each one we plot (in green) the set
    $T_{7,F_3^\alpha}$. Heuristic numerical estimates of their exponents
    and of the box dimension for the corresponding residual sets 
    for $\alpha\leq7$ are given in Table~\ref{tab:rSG}
  }
  \label{fig:rSG}
\end{figure}


%
\begin{table}
  \centering
  \begin{tabular}{|c|c|c|c|}
    \hline
    $\alpha$&$s_{F_3^\alpha}$&$2s_{F_3^\alpha}/3$&$\dim_B R_{F_3^\alpha}$\cr
    \hline
    1&2.447&1.631&1.72\cr
    \hline
    1.3&2.395&1.596&1.72\cr
    \hline
    1.7&2.377&1.585&1.71\cr
    \hline
    2&2.359&1.573&1.59\cr
    \hline
    3&2.378&1.586&1.71\cr
    \hline
    4&2.389&1.593&1.73\cr
    \hline
    7&2.394&1.596&1.76\cr
    \hline
  \end{tabular}
  \caption{%
    Numerical evaluation of the exponent of the
    real projective gaskets $F_3^\alpha$ and of the box 
    dimension of the corresponding residual sets for several values
    of $\alpha$. No analytical formula is known for
    these quantity. These data confirms the relation
    $2s_\bA/3\leq\dim_B R_\bA$ already observed in 
    Table~\ref{tab:aSG} and the fact that roughly
    $2s_\bA/3\geq 9\dim_B R_\bA/10$.
  }
  \label{tab:rSG}
\end{table}
Analytical bounds for the exponents $s_{\bF_3^\alpha}$ can be obtained
from Theorem~1. Here we present calculations for $F^1_3=\cC_3$,
the cubic gasket.
Due to the symmetry between the generators it turns out that
$$
\mu_{\cC_3}(s)=6\mu_{\cC_3 A_{12}}=3\cdot2^{1-s}\zeta(s),
$$
from which, as the unique solution of $\mu_{\cC_3}(s)=3^s$, 
we get the lower bound $1.52\leq s_{\cC_3}$.
To get the first upper bound we must consider the function
$$
\mu_{\cC_3,2}(s)=3\cdot2^{1 - s}\left(
3 \zeta(s) + 2^{2 - s} \zeta(s,\frac{7}{4})
- 2^{1-s} - 3
\right),
$$
from which we get $1.7\leq s_{\cC_3}\leq 7.1$ as the unique solutions
of $\mu_{\cC_3,2}(s)=3^{\pm s}$. In order to get more meaningful 
bounds we should consider some $\mu_{\cC_3,k}$ with a large $k$ but
leave this to a future paper.
Interpolating on the curve 
$\log N_{\cC_3}(k)$ as function of $\log k$ for $k=2^r$, $1\leq r\leq 13$,
we get a reliable estimate of $s_{\cC_3}\simeq 2.444$. A rough numerical
evaluation of the box dimension of $R_{\cC_3}$ by counting the number
of squares needed to cover the fractal gives $\dim_B\simeq 1.72$,
compatible with the relation $3\dim_B R_{\cC_3}\geq2s_{\cC_3}$ 
suggested in Conjecture~\ref{thm:conj}.

\medskip\noindent
$\boldsymbol{n\geq4.}$ In $\bR^n$ we use coordinates $(x^1,\dots,x^n)$
with respect to the frame 
$e'_1=e_1+e_n,\dots,e'_{n-1}=e_{n-1}+e_n,e'_n=e_n$.
For $n=4$ the matrices $A_1$ and $A_4$ are given by
$$
A_1=\begin{pmatrix}\alpha-1&0&0&1\cr0&1&0&0\cr0&0&1&0\cr\alpha-2&0&0&2\cr\end{pmatrix},
A_4=\begin{pmatrix}1&0&0&0\cr 0&1&0&0\cr 0&0&1&0\cr 2-\alpha&2-\alpha&2-\alpha&\alpha\cr\end{pmatrix}
$$
and $A_2$ and $A_3$ can be obtained via permutations of $A_1$. Similarly 
happens for $n\geq4$. Correspondingly we use coordinates $u^i=x^i/x^n$, 
$i=1,\dots,n-1$, and obtain
$$
\psi_1(u^i)=\left(
\frac{(\alpha-1)u^1+2}{(\alpha-1)u^1+2},
\frac{u^2}{(\alpha-1)u^1+2},\dots,
\frac{u^{n-1}}{(\alpha-1)u^1+2}\right),
$$
similarly for $i<n-1$ and
$$
\psi_{n-1}(u^i)=\left(
\frac{u^1}{(2-\alpha)(u^1+\dots+u^{n-1})+\alpha},
\dots,
\frac{u^{n-1}}{(2-\alpha)(u^1+\dots+u^{n-1})+\alpha}
\right).
$$
A direct evaluation of the eigenvalues of the Jacobian matrices
of the $\psi_i$ gives the same result we got for $n=3$. In particular
for every $n\geq3$ we have that the gasket $\bF_n^\alpha$ is a hyperbolic
IFS for $\alpha\in(1,4)$ and a parabolic IFS for $\alpha=1,4$.
The bounds on the Hausdorff dimension of the residual sets give
$$
\min\{\log_n\frac{4}{\alpha},\frac{1}{\log_n\alpha}\}
\leq
\dim_H R_{F_n^\alpha}
\leq
\max\{\log_n\frac{4}{\alpha},\frac{1}{\log_n\alpha}\}.
$$ 
\begin{figure}
  \centering
  \begin{tabular}{cc}
    \includegraphics[width=7cm]{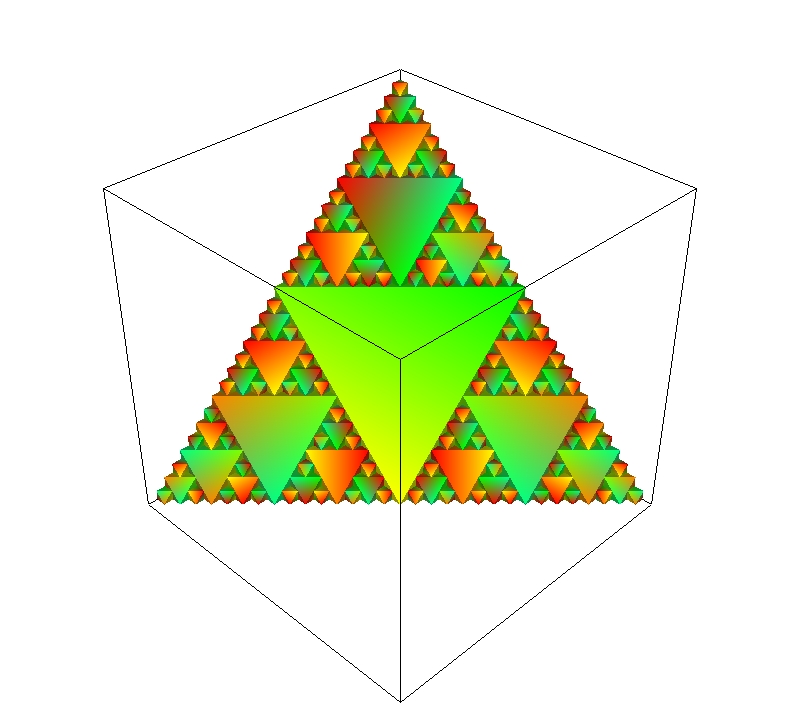}&\includegraphics[width=7cm]{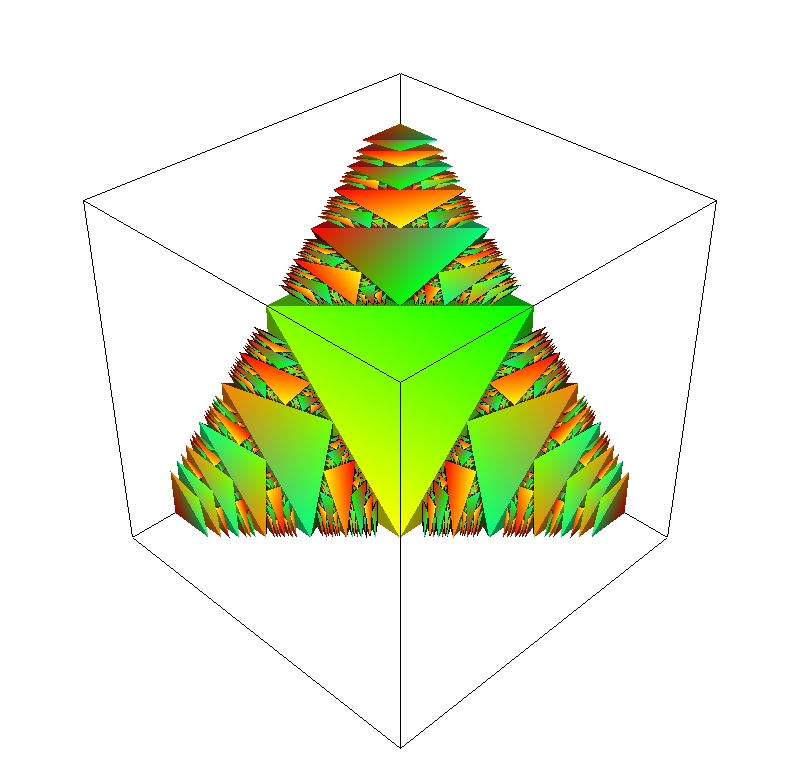}\cr
    \noalign{\medskip}
    \includegraphics[width=5cm]{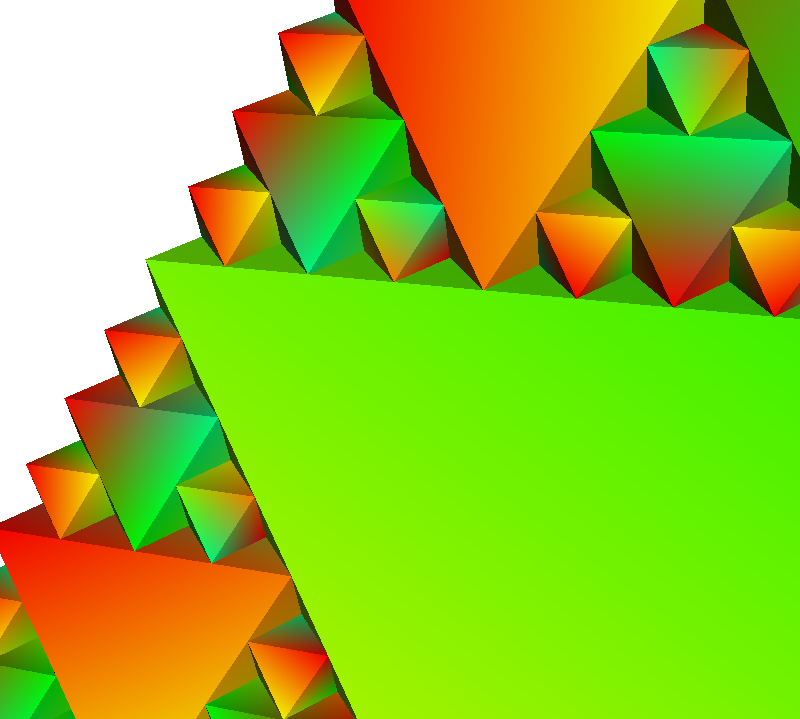}&\includegraphics[width=5cm]{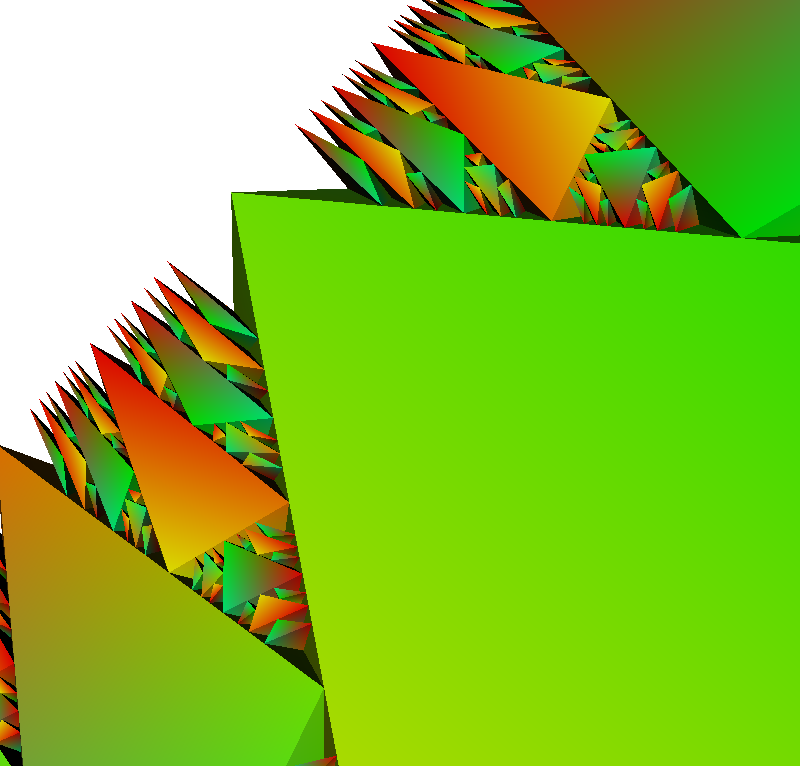}\cr
    \end{tabular}
    \caption{%
      \footnotesize
      Images of the Sierpinski ($F^2_4$) and the cubic ($F^1_4$)
      tetrahedra. In figure we show a full picture (above) and 
      a detail (below) for the sets $T_{5,F^2_4}$ (left) and 
      $T_{5,F^1_4}$ (right).
    }
    \label{fig:rSG4}
\end{figure}
For $n=4,\alpha=2$ we get the well-known result that the dimension
of the standard Sierpinski tetrahedron is equal to 2. Numerical 
evaluations suggest that the same could hold for the 4-dimensional 
version of the cubic gasket (see Fig.~\ref{fig:rSG4} for a picture
of the two sets). 
\subsubsection{The Apollonian gasket}
We conclude the paper with a brief discussion on the Apollonian semigroup,
namely the semigroup $\bH\subset SL_4(\bN)$ generated by the
matrices $H_1,H_2,H_3$ introduced in the Motivational Example~2.
This case was thoroughly studied, somehow implicitly, by Boyd, 
in particular in~\cite{Boy72,Boy73b,Boy82}, in the context of the 
sequence of curvatures in an Apollonian gasket and Boyd's investigation
and arguments were the 
archetype for most results and arguments in Section~\ref{sec:norms} 
of the present paper. 

Recall that the matrix $H_1$ has only one eigenvalue (hence
equal to 1) and therefore, even though $H_2$ and $H_3$ have eigenvalues
larger than 1, $\bH$ is a {\em parabolic} gasket. Next proposition
grants that our results do apply in fact, as expected, 
to $\bH$ itself and shows that $s_\bH\leq\infty$ with arguments
that entail only the matrices themselves.
\begin{lemma}
  Assume that matrices $A_1,\dots,A_m\in M_n(\bN)$ have the following
  properties:
  \begin{enumerate}
  \item they have some number $k\neq1$ of rows containing a single 
    entry equal to 1 and all other equal to 0 and these entries equal 
    to 1 belong all to different columns and in those columns all 
    entries are either 0 or 1;
  \item other rows are such that each of their entries is smaller than 
    the sum of the remaining $n-1$ entries.
  \end{enumerate}
  Then this property is shared by all products of the $A_i$.
\end{lemma}
\begin{proof}
  We prove the lemma by induction. It is enough to consider the products
  of two generic matrices $A=(A^i_j)$, $B=(B^i_j)$, satisfying the hypotheses.

  Assume first that $k=0$ for $B$, namely 
  $\sum_{k\neq\ell}B^i_k\geq B^i_\ell$ for all $i,\ell$. Then 
  $$
  \sum_{k\neq\ell}(AB)^i_k=\sum_{\substack{1\leq j\leq n\\ k\neq \ell}}A^i_j B^j_k=
  \sum_{1\leq j\leq n}A^i_j\sum_{k\neq\ell}B^j_k\geq\sum_{1\leq j\leq n}A^i_jB^j_\ell=(AB)^i_\ell.
  $$
  
  Assume now that $k>1$ for $B$ and denote by $I=(i_1,\dots,i_k)$ the rows
  with a 1 and all other entries equal to 0. Every line (if any) of $A$ 
  with a 1 and all other entries equal to 0 leaves unaltered the corresponding
  row in $B$ and therefore the new line satisfies the conditions in the theorem.
  Otherwise we notice that 
  $$
  \sum_{k\neq\ell}(AB)^i_k=\sum_{\substack{1\leq j\leq n\\ k\neq \ell}}A^i_j B^j_k\geq
  \sum_{\substack{1\leq j\leq n\\ j\not\in I}}A^i_jB^j_\ell
  + \sum_{\substack{1\leq j\leq n\\ j\in I}}A^i_j \sum_{k\neq\ell}B^j_k.
  $$
  If $j\in I$ then $\sum_{k\neq\ell}B^j_k$ is either 0 or 1. Since by hypothesis
  there are at least two such rows and 
  $\sum_{k\neq\ell}B^j_k+\sum_{k\neq\ell}B^{j'}_k\geq1$ for every $j,j'\in I$, 
  $j\neq j'$, and the corresponding entries $A^i_j$ and $A^i_{j'}$ are both
  equal to 1, then
  $$
  \sum_{\substack{1\leq j\leq n\\ j\in I}}A^i_j \sum_{k\neq\ell}B^j_k
  \geq 
  \sum_{\substack{1\leq j\leq n\\ j\in I}}A^i_j B^j_\ell = (AB)^i_\ell,
  $$
  therefore $\sum_{k\neq\ell}(AB)^i_k\geq (AB)^i_\ell$.
\end{proof}
\begin{proposition}
  The Apollonian semigroup is a fast gasket with coefficient $c\geq1/4$.
\end{proposition}
\begin{proof}
  Note first of all that Hirst matrices satisfy previous Lemma's conditions.
  Moreover the entries in the third line are not smaller than all other entries
  in the same column and it is easy to see by induction that this property
  is preserved by products.

  Let $\|A\|_\infty=\max_{1\leq i\leq n}\sum_{1\leq j\leq n}|A^i_j|$.
  A look at the 6 matrices $H_{ij}$, $i\neq j$, shows that their 
  third column has always at least three non-zero entries, so that 
  $\|A_{IJ}\|_\infty\geq \|A_J\|_\infty\sum_{j\neq j_0}|A^i_3|$ where $j_0$ is
  the index of the element of the third column (if any) equal to zero
  (otherwise just set $j_0=1$). By the previous Lemma and the 
  fact that the norm of every $A\in\bH$ in concentrated in the third
  row, the sum of any three entries of the third row of $A$ is always
  larger than $\|A\|$, so that $\|A_{IJ}\|_\infty\geq \|A_J\|_\infty \|A_I\|$.
  Since $4\|A\|\geq\|A\|_\infty\geq\|A\|$, the claim follows.
\end{proof}
Analytical bounds for the exponent $s_\bH$ of the Apollonian 
semigroup were studied in detail by Boyd in~\cite{Boy70,Boy72,Boy73a}
and we do not attempt to improve them here. 

Increasingly accurate {\em numerical} evaluations of $s_\bH$ with
several different techniques have been given over the last 
half-century by Melzak~\cite{Mel69}, Boyd~\cite{Boy82}, Manna and
Herrmann~\cite{MH91}, Thomas and Dhar~\cite{TD94} and 
McMullen~\cite{McM98} giving respectively 
the following values, with a {\em heuristic} error of 1 unit on the 
last digit: 1.306951, 1.3056, 1.30568, 1.30568673, 1.305688.
We remark that, among all these evaluations, the one with the 
largest number of digits, given by Thomas and Dhar, is the only one
based on a {\em heuristic} method, while the others are based
on {\em exact} methods.

Partly to test our own software evaluating the function $N_\bH(k)$
for a generic gasket $\bH$ and partly because the computational
power of computers increased quite a lot over the last fifteen years,
which is how old is the last evaluation of the exponent, we 
repeated the elementary evaluation made by Boyd in 1982 by 
evaluating $N_\bH(k)$ for $k=2^p$, $p=1,\dots,40$ with respect 
to the norm $\|A_I\|=\sum_{1\leq i,j\leq 4}(A_I)_{ij}v^iw^i$,
where $v=(-1,2,2,0)$ and $w=(1,1,1,2)$ (this way $\|A_I\|$ is equal
to the the curvature of the circle of multi-index $I$
in the Apollonian gasket generated by the circles of radius $-1,2,2$), 
and then interpolating the data obtained (see Table~\ref{tab:N}).
We found a value of $s_\bH\simeq1.30568673$ which fully confirms the 
heuristic evaluation of Thomas and Dhar and suggests an error of 2 
on the last digit of the estimate of McMullen.
\begin{table}
  \centering
  \begin{tabular}{|U|T|}
    \hline
    \begin{tabular}{c}
      $\boC_2$\cr(19)\cr
      \end{tabular}
    &\footnotesize 3, 15, 71, 287, 1231, 4911, 19831, 79279, 318383, 1273807, 5098247, 20391887, 81590055, 326364583, 1305483999, 5221928631, 20888160751, 83552534287, 334211194663\\
    \hline
    \begin{tabular}{c}
      $\boC_3$\cr(13)\cr
      \end{tabular}
    &\footnotesize 4, 22, 148, 760, 4594, 24646, 136372, 740650, 4046188, 22022770, 119929126, 652445212, 3550689778\\
    \hline
    \begin{tabular}{c}
      $\boC_4$\cr(12)\cr
    \end{tabular}
    &\footnotesize 5, 37, 293, 2197, 15125, 103669, 714245, 4849045, 32901077, 222724789,1507986917, 10202765749\\
    \hline
    \begin{tabular}{c}
      $\bA_3$\cr(13)\cr
    \end{tabular}
    &\footnotesize 3, 12, 64, 316, 1784, 10004, 58224, 341386, 2033906, 12170708, 73208110, 441772966, 267292497\\
    \hline
    \begin{tabular}{c}
      $\bH$\cr(39)\cr
    \end{tabular}
    &\tiny 0, 1, 3, 8, 18, 48, 113, 278, 681, 1722, 4238, 10488, 25927, 64086, 158266, 391062, 967315, 2390800, 5909752, 14608522, 36115118, 89275994, 220684802, 545546400, 1348603780, 3333755028, 8241076212, 20372155276, 50360227721, 124491161884, 307744098990, 760747405278, 1880578271904, 4648814463680, 11491932849933, 28408221038996, 70225503797745, 173598409768852, 429137646728801\\
    \hline
    \begin{tabular}{c}
      $\bF$\cr(35)\cr
    \end{tabular}
    &\tiny2, 7, 16, 34, 84, 151, 348, 679, 1546, 3034, 6546, 13476, 28409, 59578, 122139, 261698, 531191, 1144823, 2314772, 4986951, 10132768, 21667197, 44400099, 94074745, 194587388, 408651488, 852101402, 1777247239, 3726410796, 7738675037, 16274400897, 33739772516, 71002774691, 147235829060, 309533001058\\
    \hline
  \end{tabular}
  \caption{%
    Values of $N_\bA(2^k)$ for small $k$ for the {\em cubic semigroups} 
    $\boC_i$, $i=2,3,4$, the Apollonian semigroup $\bA_3$, the Hirst semigroup 
    $\bH$ and the semigroup $\bF$ of Example~\ref{ex:Fal}. In the left column 
    it is also reported the numnber of terms displayed in the right one.
  }
  \label{tab:N}
\end{table}

\section*{Acknowledgments}
I am grateful to S.P. Novikov, I.A. Dynnikov and B. Hunt 
for several precious insights and suggestions on the matter presented
in this article and thank them 
and T. Gramchev for several fruitful discussions while 
writing the paper. Finally I am grateful to my wife M. Camba
for helping speeding up considerably my code to evaluate the 
functions $N(k)$. Most calculations were done on the 
$\simeq$100 cores 2,66GHz Intel Xeon Linux cluster of {\em INFN} 
(Cagliari); latest calculations were also performed on 
the iMac cluster of the {\em Laboratory of geometrical methods 
in mathematical physics} (Moscow), recently created by the Russian Government 
(grant no. 2010-220-01-077) for attracting leading scientists 
to Russian professional education institutes. Finally I an grateful
to the IPST and the Mathematics Department of the University of Maryland 
for their hospitality in the Spring and Fall 2011 while I was 
working at the paper.
\bibliography{refs}
\end{document}